\documentclass[11pt]{amsart}
\usepackage{amsmath,amsthm,amsfonts,graphicx,psfrag,amssymb,dsfont,mathrsfs,mathtools,bm}
\usepackage{minitoc,color}
\usepackage{enumitem,linegoal}
\usepackage{thmtools}
\usepackage{thm-restate}
\usepackage{tensor}
\usepackage[top=1.5in, bottom=1.5in, left=0.7in, right=0.7in]{geometry}  
\usepackage[colorlinks]{hyperref}
\hypersetup{
	colorlinks,
	citecolor=blue,
	linkcolor=red
}  
\usepackage[alphabetic,initials]{amsrefs}
\usepackage{bbm}
\usepackage{tikz-cd}
\usepackage[all]{xy}
\usepackage{cleveref}
\usepackage{comment}
\usepackage{cancel}

\theoremstyle{plain}
\newtheorem{thm}{Theorem}[section]
\crefname{thm}{theorem}{theorems}
\newtheorem{thmintro}{Theorem}

\crefname{thmintro}{theorem}{theorems}

\newtheorem{prop}[thm]{Proposition}
\crefname{prop}{proposition}{propositions}
\newtheorem{cor}[thm]{Corollary}
\newtheorem{corintro}[thmintro]{Corollary}

\crefname{corintro}{corollary}{corollaries}
\newtheorem{lemma}[thm]{Lemma}

\theoremstyle{definition}
\newtheorem{example}[thm]{Example}

\newtheorem{defn}[thm]{Definition}

\newtheorem{remark}[thm]{Remark}
\newtheorem*{remark*}{Remark}

\newcommand\xqed[1]{
  \leavevmode\unskip\penalty9999 \hbox{}\nobreak\hfill
  \quad\hbox{#1}}
\newcommand\END{\xqed{$\triangleleft$}}

\numberwithin{equation}{section}

\newcommand{\RR}{\mathbb{R}}
\newcommand{\QQ}{\mathbb{Q}}
\newcommand{\CC}{\mathbb{C}}

\newcommand{\NN}{\mathbb{N}}
\newcommand{\ZZ}{\mathbb{Z}}
\newcommand{\DD}{\mathbb{D}}
\newcommand{\BB}{\mathbb{B}}

\newcommand{\xistd}{\xi_{\mathrm{std}}}

\def\bbC{{\mathbb C}}

\def\bbP{{\mathbb P}}

\def\bbR{{\mathbb R}}
\def\bbS{{\mathbb S}}

\def\bbZ{{\mathbb Z}}

\def\cA{{\mathcal A}}

\def\cC{{\mathcal C}}
\def\cD{{\mathcal D}}
\def\cE{{\mathcal E}}

\def\cH{{\mathcal H}}

\def\cL{{\mathcal L}}
\def\cM{{\mathcal M}}
\def\cN{{\mathcal N}}
\def\cO{{\mathcal O}}
\def\cP{{\mathcal P}}

\def\cS{{\mathcal S}}

\def\cU{{\mathcal U}}
\def\cV{{\mathcal V}}
\def\cW{{\mathcal W}}
\def\cX{{\mathcal X}}

\def\bC{{\bm C}}
\def\bD{{\bm D}}
\def\bE{{\bm E}}

\def\bU{{\bm U}}
\def\bV{{\bm V}}
\def\bW{{\bm W}}

\def\bg{{\bm g}}
\def\be{{\bm e}}
\newcommand{\partbar}{\overline{\partial}}

\def\con{{\mathfrak{Con}}}

\DeclareMathOperator{\supp}{Supp}
\DeclareMathOperator{\dom}{Dom}

\DeclareMathOperator{\Id}{id}
\DeclareMathOperator{\Diff}{Diff}
\DeclareMathOperator{\Fix}{Fix}
\DeclareMathOperator{\CR}{CR}
\DeclareMathOperator{\CZ}{CZ}
\DeclareMathOperator{\Obj}{Obj}
\DeclareMathOperator{\Mor}{Mor}
\DeclareMathOperator{\coker}{coker}
\DeclareMathOperator{\indexop}{index}
\DeclareMathOperator{\diag}{diag}
\DeclareMathOperator{\age}{age}
\DeclareMathOperator{\stab}{stab}
\DeclareMathOperator{\ind}{ind}
\DeclareMathOperator{\Conj}{Conj}
\def\std{{\rm std}}
\def\vit{{\rm Viterbo}}
\def\pt{{\rm pt}}
\def\rd{{\rm d}}
\def\orb{{\rm orb}}

\def\i{\mathbf{i}}

\newcommand{\Addresses}{{
		\bigskip
		\footnotesize
		
	    Fabio Gironella \par\nopagebreak
         \textsc{Laboratoire de Mathématique Jean Leray, Nantes Université, France}\par\nopagebreak
         \textit{E-mail address:} \href{mailto: fabio.gironella@cnrs.fr}{ fabio.gironella@cnrs.fr}, 
         \href{mailto: fabio.gironella.math@gmail.com}{ fabio.gironella.math@gmail.com}
		
		\medskip
		
	     Zhengyi Zhou \par\nopagebreak
        \textsc{State Key Laboratory of Mathematical Sciences, Chinese Academy of Sciences;}\par\nopagebreak
	    \textsc{Morningside Center of Mathematics, Chinese Academy of Sciences;}\par\nopagebreak
         \textsc{Academy of Mathematics and Systems Science, Chinese Academy of Sciences, China}\par\nopagebreak
		\textit{E-mail address}: \href{mailto:zhyzhou@amss.ac.cn}{zhyzhou@amss.ac.cn}

}}

\title{Exact orbifold fillings of contact manifolds}
\author{Fabio Gironella, Zhengyi Zhou}
\date{}

\begin{document}

	\begin{abstract}
		We study exact orbifold fillings of contact manifolds using Floer theories. Motivated by Chen-Ruan's orbifold Gromov-Witten invariants, we define symplectic cohomology of an exact orbifold filling as a group using classical techniques, i.e.\ choosing generic almost complex structures. By studying moduli spaces of pseudo-holomorphic/Floer curves in orbifolds, we obtain various non-existence, restrictions and uniqueness results for orbifold singularities of exact orbifold fillings of many contact manifolds. For example, we show that exact orbifold fillings of $(\mathbb{RP}^{2n-1},\xi_{\std})$ always have exactly one singularity modeled on $\CC^n/(\ZZ/2\ZZ)$ if $n\ne 2^k$. Lastly, we show that in dimension at least $3$ there are pairs of contact manifolds without exact cobordisms in either direction, and that the same holds for exact \emph{orbifold} cobordisms in dimension at least $5$. 
	\end{abstract}

        \maketitle
    
	\section{Introduction}
Orbifolds are a natural generalization of manifolds.
They have been first introduced by Satake \cite{Satake57} under the name of ``V-manifolds'', and were later rediscovered by Thurston \cite{thurston1979geometry}, to whom they owe their current name, in the study of the geometry of $3-$manifolds, as well as by Haefliger \cite{Hae90}, under the name ``orbihedra'', in the study of $\mathrm{CAT(}k\mathrm{)}$ spaces. 
Since then, the geometry of orbifolds has been widely studied in the literature. 
For instance, besides the fundamental fact that moduli spaces of pseudo-holomorphic curves can be orbifolds, another important appearance of orbifolds in symplectic geometry dates back to the original Mirror Conjecture, as the first mirror pair contains a Calabi-Yau orbifold \cite{MR1115626}.  
In view of this, Chen--Ruan \cite{CheRua01} developed the Gromov-Witten theory of symplectic orbifolds. 
The algebraic analogue was developed by Abramovich--Graber--Vistoli \cite{abramovich2008gromov} and was intensely investigated in enumerative geometry.

Additionally to the fact that the study of symplectic orbifolds is of independent interest, these singular objects can be very useful in the study of smooth symplectic topology and dynamics. 
For example, Mak--Smith \cite{mak2021non} used orbifold Lagrangian Floer theory developed by Cho--Poddar \cite{orbiLagrangian} to obtain new families of non-displaceable Lagrangian links in symplectic four-manifolds. The idea was further exploited by Polterovich--Shelukhin \cite{polterovich2021lagrangian} 
leading to recent breakthroughs on dynamics on surfaces and $C^0$ symplectic geometry. 

As far as Floer theory on orbifolds is concerned, it is also worth mentioning the previous work of Moreira \cite{MoreiraThesis}, which deals with the case of global symplectic quotients $X/G$ with isolated singularity. 
The recent preprint by Mak--Seyfaddini--Smith \cite{MSS25} also deals with the global quotient case, with non-isolated singularity, using global Kuranishi charts.

In this paper, we continue the study of the symplectic topology of exact orbifold fillings. 
In particular, we will focus on exact orbifold fillings of contact manifolds and their applications to \emph{smooth} contact topology (as opposed to the contact topology of contact orbifolds). 
We obtain results on uniqueness, existence and restrictions of possible orbifold singularities of exact orbifold fillings.  
This turns out to be useful in showing that the partial order given by the existence of an exact (smooth) cobordism considered in \cite{RSFT} is not a total order.

Our main tool is the symplectic cohomology for exact orbifold fillings of contact manifolds. 
Although Floer theory/SFT is expected to exist for general symplectic orbifolds, i.e.\ symplectic orbifolds with contact orbifold boundary, only considering exact orbifold fillings of contact manifolds brings up several advantages in the construction, both on the technical side and structural side of the theory. 
For example, an exact orbifold filling of a contact manifold can only have \emph{isolated} singularities modeled on $\CC^n/G$ for a finite subgroup $G\subset U(n)$. 
This then allows us to define the symplectic cohomology \emph{group} using a generic almost complex structure, without the need to appeal to virtual techniques. 
Moreover, it turns out that the symplectic cohomology group can in fact be defined over any coefficient ring, and that many of the moduli spaces of Floer cylinders we consider are manifolds instead of orbifolds; our geometric applications rely on this last point crucially.

\begin{remark}
	Although the symplectic cohomology group is well defined for a generic almost complex structure, the study of additional structures on it, like its ring structure, cannot be achieved just by choosing a generic almost complex structure, and more sophisticated techniques are required.
	We make some speculation assuming the ring structure, but all of the results in this paper are obtained using classical transversality methods.
\end{remark}

\subsection{Symplectic cohomology of exact orbifold fillings of contact manifolds}

Our first main result deals with \emph{exact orbifold fillings of smooth contact manifolds}, i.e.\ orbifolds having a smooth (that is, without orbifold singularities) boundary and equipped with an orbifold differential $1$-form, that restricts to a contact form on the boundary and whose differential is non-degenerate everywhere on the orbifold.

\begin{thmintro}\label{thm:SH}
	Let $W$ be an exact orbifold filling of a smooth contact manifold and $R$ a ring.
	Then, there are well-defined $\ZZ/2\ZZ$-graded groups $SH^*(W;R)$ and $SH^*_+(W;R)$, that fit into an exact triangle
	$$\xymatrix{
		H^*_{\CR}(W;R)\ar[rr] & &SH^*(W;R)\ar[ld] \\
		& SH^*_+(W;R) \ar[lu]^{[1]} &}
	$$
	where $H^*_{\CR}(W;R)$ is the Chen-Ruan orbifold cohomology \emph{group}\footnote{That is the cohomology of the orbit space of the inertia orbifold with coefficients in $R$ (see \S \ref{SS:CR}).}
	\cite{CheRua04}.  
\end{thmintro}	
\begin{remark}
If the integral first Chern class $c^{\ZZ}_1(W)\in H^2_{\orb}(W;\ZZ):=H^2(BW;\ZZ)$ vanishes, where $BW$ is the classifying space\footnote{\label{foot:2}
The definition of classifying space $BW$ is due to Segal \cite{Seg68,Seg74}; see \cite[Pages 24-25]{Ruan07} for an overview. As a special case, it is worth mentioning that for an orbifold $X$ associated to the global action of a Lie group $G$ on a smooth manifold $M$ (see \Cref{exa:action_groupoid}), $BX$ is more explicitly given by $EG\times_G X$ \cite[Proposition 1.51]{Ruan07}, thus $H^*(BW;\ZZ)=H^*_G(M;\ZZ)$ is the ordinary equivariant cohomology.
}, then one gets a $\ZZ-$grading, which mod $2$ coincides with the $\ZZ/2\ZZ$ grading in the statement and which is moreover canonically defined on generators that have torsion homotopy class. 
If the rational first Chern class $c^{\QQ}_1(W)\in H^2_{\orb}(W;\QQ):=H^2(BW;\QQ)$ vanishes, then one gets a $\QQ-$grading, whose relation with the $\ZZ/2\ZZ$ grading is in general less direct.
\end{remark}
We have the following basic computation.
\begin{thmintro}\label{thm:quotient}
	Let $G\subset U(n)$ be a finite subgroup, such that $\ker(g-I)=\{0\}$ for $g\ne \Id \in G$, i.e.\ $\CC^n/G$ is an isolated singularity. Then $SH^*(\CC^n/G;R)=0$ iff $|G|$ is invertible in $R$. 
\end{thmintro}
In the above statement, we see $\CC^n/G$ as the Liouville completion of the exact orbifold filling $D^{2n}/G$ of the smooth contact manifold $S^{2n-1}/G$. 

The fact that symplectic cohomology can be defined over any coefficient ring $R$ follows crucially from the geometry of exact orbifold fillings of contact manifolds; 
in particular, it is not clear whether such property will hold for exact orbifold fillings of contact \emph{orbifolds}. 
Moreover, symplectic cohomology is not defined as the direct limit of the Hamiltonian-Floer cohomology of Hamiltonians with slope converging to infinity, but rather as the direct limit over a restricted family that are ($C^2$ close to) zero on the filling. 
The necessity of such difference, which is explained in \S \ref{S4}, can be seen from \Cref{thm:quotient} as $SH^*(\CC^n/G; R)$ should always be zero if we are allowed to define it by the direct limit of the Hamiltonian-Floer cohomology of $H_k=kr^2$ on $\CC^n/G$ like the $\CC^n$ case explained in \cite[(3f)]{biased}.
(c.f.\ \Cref{ex:non_vanish}).

\begin{remark}
McLean--Ritter \cite{McKay} proved that crepant resolutions of isolated quotient singularity have vanishing symplectic cohomology. 
Combined with \Cref{thm:quotient}, this can be viewed as a local case of Ruan's crepant resolution conjecture \cite{Rua01}. It is worth noting that this local crepant resolution conjecture fails with $\ZZ$ coefficient, even though the symplectic cohomology of $\CC^n/G$ and Calabi-Yau strong fillings (e.g.\ crepant resolutions) can be defined over $\ZZ$ (in the latter case, the coefficient ring needs to be the Novikov ring $\ZZ((t))$ over $\ZZ$ in view of Gromov compactness). Indeed, we have $SH^*(\CC^2/(\ZZ/(2\ZZ));\ZZ((t)))\ne 0$, while $SH^*(\cO(-2);\ZZ((t)))=0$ \cite{ritter}. 
We also point out that the elimination pattern of the unit in these two cases are different, as a comparison between the proof of \Cref{thm:quotient} and \cite[Remark 2.16]{Zho} shows.
\end{remark}

\subsection{Existence, uniqueness and restrictions of orbifold singularities}
A natural question regarding orbifold fillings of contact manifolds is that of understanding what kind of singularities they can have. 
Using the symplectic cohomology defined above and the fact that certain moduli spaces of Floer cylinders are manifolds instead of orbifolds, we obtain the following restrictions on the order of isotropy.

\begin{thmintro}
\label{thm:sphere_have_smooth_fillings}
Let $V$ be a Liouville domain.
Then, any exact orbifold filling of $\partial(V\times \DD)$ is smooth, i.e.\ a manifold.
In particular, \cite[Theorem 1.1, 1.2]{zhou2020} can be applied to get cohomology/diffeomorphism-type uniqueness of exact orbifold fillings. 
\end{thmintro}
For instance, it follows that any exact orbifold filling of  $(S^{2n-1},\xi_{std})$ is diffeomorphic to $\DD^{2n}$ for $n\ge 2$.

\medskip

We now denote by $\stab_p$ the stabilizer/isotropy group of the point $p$ in an orbifold $W$. 
More explicitly, if around $p$ the orbifold $W$ is locally described by a ``local uniformizer'' described by a global action of a finite subgroup $G\subset U(n)$ on the unit ball in $\CC^n$, then $\stab_p=G$. 
In the following result, we use $SH^*$ to detect information on the stabilizers of orbifold points of exact orbifold fillings:

\begin{thmintro}\label{thm:order_Brieskorn}
	Let $B(k,n)$ denote the contact manifold given by the link of the singularity $z^k_0+\ldots+z^k_n=0$. 
	Let $W$ be an exact orbifold filling of $B(k,n)$. 
	If $k<n$ and $p$ is an orbifold singularity of $W$, then we have the following.
	\begin{enumerate}
		\item $|\stab_p|$ divides $k!$ .
		\item If $k<\frac{n+1}{2}$, then $|\stab_p|$ moreover divides $(k-1)!$ .
		\item If $k=\frac{n+1}{2}$ is square-free\footnote{By this we mean that it is not divided by any square other than $1$.}, then $|\stab_p|$ moreover divides $(k-1)!$ .
	\end{enumerate}
\end{thmintro}

For $k=2,n=3$, \Cref{thm:order_Brieskorn} implies that any exact orbifold filling of the cosphere bundle $ST^*S^3$ is a manifold. Such property is preserved under products and Lefschetz fibrations by the following result. (In fact, from the proof of \Cref{cor:order_dilation} one can see that all the restrictions of singularities in \Cref{thm:order_Brieskorn} are preserved under products\footnote{For example, the order of an orbifold singularity of a filling of $\partial (\{z^k_0+\ldots+z^k_n=1\}\times \{z^{k'}_0+\ldots+z^{k'}_n=1\})$ would divide $\min\{k!,k'!\}$ when $k,k'<n$.} and Lefschetz fibrations.)

\begin{corintro}\label{cor:order_dilation}
	Let $S$ be the minimal subset of the set of Liouville domains with vanishing first Chern class that contains the unit cotangent bundle $DT^*S^3$ and such that products with Liouville domains in $S$ are also in $S$, and that the total space of any Lefschetz fibration with regular fiber in $S$ is also in $S$.
	Then, for every $W\in S$, any exact orbifold filling of $\partial W$ is a manifold. 
\end{corintro}
For example, \Cref{cor:order_dilation} applies to the link of a singularity in the form of $z_0^2+z_1^2+z_2^2+z_3^2+p(z_4,\ldots,z_n)$ for a polynomial $p$ which has an isolated singularity at $0$.
\begin{thmintro}\label{thm:order}
	Let $W$ be an exact orbifold filling of the lens space $(L(k;1,\ldots,1)=S^{2n-1}/(\ZZ/k\ZZ),\xi_{\std})$ for $n\ge 2$.
	Then the following holds for any singular point $p\in W$.
	\begin{enumerate}
		\item\label{G1} $|\stab_p|$ divides $k!$.
		\item\label{G2} If $k<n$, then $|\stab_p|$ divides $k^n$.  In particular, if $k$ is a prime, then $|\stab_p|=k$.
	\end{enumerate}
\end{thmintro}

It was proved in \cite{Zho} that $(\mathbb{RP}^{2n-1},\xi_{std})$ is not exactly fillable when $n\ne 2^k$. The key step is to deduce that $\dim H^*(W;\QQ)\le 2$ for an exact filling $W$ of $\mathbb{RP}^{2n-1}$. 
Although considering exact orbifold fillings does not yield a new proof of the results in \cite{Zho}, it provides an explanation of the fact that $\dim H^*(W;\QQ)\le 2$, since $\dim H^*_{\CR}(\CC^n/(\ZZ/2\ZZ);\QQ)=2$. 
It is moreover plausible that exact orbifold fillings of $(\mathbb{RP}^{2n-1},\xi_{\std})$ are unique when $n\ge 3$; the following result provides some evidence for this conjecture.

\begin{thmintro}\label{thm:unique}
	If $n\ne 2^k$, exact orbifold fillings of $(\mathbb{RP}^{2n-1},\xi_{std})$ have exactly one orbifold singularity modeled on $\CC^n/(\ZZ/2\ZZ)$. 
\end{thmintro}

\begin{remark}
	\Cref{thm:unique} should also hold for the lens spaces considered in \cite[Theorem 1.3]{Zho} (namely, those of the form $S^{2n-1}/(\ZZ/p\ZZ)$, where $p$ is a prime, $\ZZ/p\ZZ$ acts on $\CC^n$ with the same weight on each coordinate, and $n$ has $p$-adic representation $n=\sum_{s=0}^k a_sp^s$ with $\sum_{s=0}^k a_s>3p-3$). 
	However, this cannot be done with classical transversality tools. 
	The easiest way to obtain such generalizations would be equipping $SH^*(W;\QQ)$ with a unital ring structure.
	However, the pair of pants construction is almost never cut out transversely if one only perturbs the almost complex structure, see \S \ref{SSS:product}. 
	We will discuss algebraic structures on Floer theory/SFT of symplectic orbifolds in more general settings with virtual techniques in a sequel work.
\end{remark}

\subsection{Obstructions to exact cobordisms}
In \cite{RSFT}, Moreno and the second author studied the partial order endowed on the collection of contact manifolds by the existence of exact cobordisms, i.e.\ of smooth cobordisms endowed with a $1$-form $\lambda$ whose exterior derivative is symplectic and whose restriction to the boundary components has kernel the prescribed contact structures. 
It is a natural question to ask whether this partial order is a total order, i.e.\ are there pairs of contact manifolds without exact cobordisms in either directions? 
Although it is tempting to believe the existence of such pairs given the richness of contact manifolds, examples are far from obvious. 
Indeed, the only known example in dimension $3$, explained to us by Chris Wendl, uses several rather deep results in $3$-dimensional contact topology from \cite{bowden2012exactly,etnyre2004planar,ghiggini2005strongly,massot2013weak}; see \S\ref{sec:non_exist_cobord} for the description of the example and the argument. 
With the help of exact orbifold fillings, the question can also be solved for all higher dimensions.
\begin{corintro}
\label{cor:pair_without_exact_cobord}
	For every $n\ge 2$, there exist two $(2n-1)-$dimensional contact manifolds $Y_1,Y_2$, between which there are no exact cobordisms in either direction. 
\end{corintro}

Moreover, the same holds even if we consider exact orbifold cobordisms.
\begin{corintro}\label{cor:no_orbifold_cobordism}
    For every $n\ge 3$, there exist two $(2n-1)-$dimensional contact manifolds $Y_1,Y_2$, between which there are no exact \emph{orbifold} cobordisms in either direction. 
\end{corintro}

\begin{remark}
	On the other hand, the analogous question for strong cobordisms (i.e.\ cobordisms endowed with a symplectic form admitting, near each boundary component, a primitive which defines at the boundary the prescribed contact structures) seems to be open even for dimension $3$. One of the challenges is that we lack the effective tools from SFT to obstruct strong cobordisms. It is worth noting that strong cobordisms in dimension $3$ are, in fact, fairly flexible by the work of Wendl \cite{wendl2013non}.
\end{remark}	

\subsection*{Comparison with the manifold setup.} 
Since we only consider exact orbifold fillings of contact manifolds, we can only have isolated singularities. This would greatly simplify the construction of moduli spaces, as no Floer cylinder will be an orbifold point of the space of maps. In particular, we can construct Banach manifolds (instead of Banach orbifolds) to run the Fredholm theory just like the usual case of Floer theory on manifolds. In particular, we do not need the full generality of results in \S \ref{S2}. The main take-away is that isotropy groups in the space of orbifold maps are inherited from the target orbifold instead of the domain orbifold, see \Cref{cor:smooth}. This property only depends on the groupoidal structure, but not the full Banach orbifold structure. The orbifold structure on the space of maps will be needed for the discussion on pair-of-pants product and continuation maps in \S \ref{SS:fail} (where transversality fails though), as well as future constructions of Floer theories on general orbifolds. 

\subsection*{Organization of the paper.}
We review basics of orbifolds in \S \ref{S2}, including the basic algebraic topology of orbifolds. Most part of \S \ref{S2} is devoted to showing the space of certain orbifold maps between orbifolds form a natural Banach orbifold. Then a Fredholm theory on such a Banach orbifold yields an orbifold structure on the moduli space when transversality holds. The construction of such Banach orbifold structures was originally due to Chen \cite{Chen06}. We give an alternative treatment using groupoidal language for completeness. In \S \ref{S3}, we define symplectic cohomology of exact orbifold fillings of contact manifolds, focusing on the differences between manifolds and orbifolds and discussing its basic properties, including where classical transversality fails. We compute the symplectic cohomology of $\CC^n/G$ in \S \ref{S4}, and prove all the topological applications in \S \ref{S5}.

\subsection*{Acknowledgments}
We would like to thank Chris Wendl for explaining to us the $3-$dimensional examples of pairs of contact manifolds with no exact cobordisms between them in either direction, and Marc Kegel for useful discussions concerning orbifold complex line bundles. We would like thank the referee for careful reading and helpful suggestions. F.G. acknowledges the support from the European Research Council (ERC) under the European Union’s Horizon 2020 research and innovation programme (grant agreement No. 772479). Z.Z. is partially supported by National Science Foundation under Grant No. DMS-1926686, National Key R\&D Program of China under Grant No.\ 2023YFA1010500, the National Natural Science Foundation of China under Grant No.\ 12288201 and 12231010, and is pleased to acknowledge the Institute for Advanced Study for its warm hospitality.

	\section{Exact orbifolds}\label{S2}
In this section, we review briefly the basics of orbifold theory, including their algebraic/differential topology as well as their Chen-Ruan cohomology \cite{CheRua04}. 
Then we introduce the main geometric object of this paper, namely exact orbifold fillings of contact manifolds, and study their basic properties. 

\subsection{Preliminary on orbifolds}\label{ss:orbifold}
Originally, orbifolds were defined as paracompact Hausdorff spaces which are locally modeled on quotients of the form $\RR^n/G$, where $G\subset \Diff(\RR^n)$ is a finite group. 
Such objects are now referred to as \emph{reduced}, or \emph{effective}, orbifolds. 
Using orbifolds defined with such local charts, Chen--Ruan initiated the mathematical study of the string theory of orbifolds, namely they introduced a new cohomology theory of orbifolds \cite{CheRua04} as well as the corresponding orbifold quantum cohomology, or more generally, the Gromov-Witten invariants of orbifolds \cite{CheRua01}. 
In \cite{Moerdijk97}, Moerdijk--Pronk showed that effective orbifolds can be described using effective proper \'etale Lie groupoids. 

Nowadays, the groupoidal definition of orbifolds is considered as the ``right" definition, because it has the advantage of including non-reduced orbifolds and provides a better framework to describe the (bi)category of orbifolds. 
In this paper, we will also adopt the groupoidal perspective. 
Some of the best references on groupoidal orbifolds are \cite{Ruan07, Moerdijk01}. 
We will also often refer to \cite{polyfold}, since it contains detailed proofs of many basic properties of groupoidal orbifolds\footnote{\cite{polyfold} is phrased in the language of sc-calculus and M-polyfolds, but it specializes to the finite dimensional Lie groupoid setting that we consider here.}. 

\subsubsection{Orbifolds and orbibundles.}
We first recall that a groupoid category is a small category in which each arrow is an isomorphism. 
\begin{defn}[{\cite[\S 1]{Moerdijk01}}, {\cite[Definition 1.29, 1.30, 1.36]{Ruan07}}]\label{def:groupoid}
	A \emph{proper \'etale  Lie groupoid $\cC$} is a category with $\Obj(\cC)=C$ and $\Mor(\cC)=\bC$, such that the following holds.
	\begin{enumerate}
		\item $C,\bC$ are both Hausdorff spaces locally modeled on $\RR^n$ with smooth transition maps.
		\item (\'Etale.) The source and target maps $s,t:\bC\to C$ are local diffeomorphisms.
		\item The inverse map $i:\bC\to \bC$, the unit map $u:C\to \bC$ and the multiplication $m:\bC\tensor[_s]{\times}{_t} \bC \to \bC$ given by composition of morphisms are all smooth. 
		(Here, $\bC\tensor[_s]{\times}{_t} \bC$ denotes $\{(\phi,\psi)\in\bC\times\bC\vert s(\phi)=t(\psi)\}$.)
		\item (Proper.) $(s,t):\bC\to C\times C$ is a proper map. 
	\end{enumerate}
\end{defn}

\begin{remark}
Instead of stating $C,\bC$ as finite dimensional manifolds as in \cite{Ruan07,Moerdijk01}, we only require them to be Hausdorff spaces locally modeled on $\RR^n$ (or more generally, Banach spaces) with smooth transition maps. 
The missing of paracompactness or second countability costs us the partition of unity on $C$ and $\bC$.
However, what we really need is that the quotient space $|\cC|$ is paracompact and hence has partition of unity (for the purpose of constructing Riemannian metrics etc.), which is stated as one of conditions in the definition of orbifold below. 
Such simplification on $C,\bC$ relieves us from checking paracompactness for the Banach orbifold of maps between orbifolds in \S \ref{ss:morphism}. 
On the other hand, under the assumption that the quotient space $|\cC|$ is second countable, which is the case for most of our applications including the  Banach orbifold of maps between orbifolds in \S \ref{ss:morphism}, one can show that there is an equivalent groupoid $\cC'$ such that $\Obj\cC',\Mor\cC'$ are second countable.
\end{remark}

Given a proper \'etale Lie groupoid $\cC$, the \emph{orbit set} $|\cC|=C/\bC$, i.e.\ the set of equivalence classes with the equivalence relation $x\sim y$ iff $\phi\colon x \mapsto y$ as an element in $\bC$ (i.e.\ $s(\phi)=x$ and $t(\phi)=y$),
is equipped with the quotient topology and is a Hausdorff space by the proper and étale condition in \Cref{def:groupoid}.  

\begin{example}
\label{exa:action_groupoid}
   Given an action of a finite group $G$ on a manifold $M$, one can form the \emph{action groupoid} $G\ltimes M$.
   This has $\Obj(G\ltimes M)=M$ and $\Mor(G\ltimes M)=G\times M$, with source map $s(g,x)=x$, target map $t(g,x)=g\cdot x$, inverse $i(g,x)=(g^{-1},gx)$, composition $m((g_2,g_1x),(g_1,x))=(g_2g_1,x)$, and unit $u(x)=(1_G,x)$. 
   Then, $G\ltimes M$ is a proper \'etale Lie groupoid, whose orbit set is $\vert G\ltimes M \vert = M/G$. 
   A special case of such groupoid is for $M=\RR^n$ and $G\subset GL(n,\RR)$ finite. \END
\end{example}

A functor between proper étale Lie groupoids is called \emph{smooth} if it is smooth both on the object and morphism level. 
An \emph{equivalence} from $\cC$ to another proper \'etale Lie groupoid $\cD$ is a fully faithful functor $F$ that is a local diffeomorphism on the object level and such that $|F|:|\cC|\to \vert\cD\vert$ is a homeomorphism\footnote{This is the definition in \cite[Definition 10.1.1]{polyfold}, which is equivalent to the ones in \cite[\S 2.4]{Moerdijk01} and \cite[Definition 1.42]{Ruan07}.}. Let $F,G$ be two smooth functors from $\cC$ to $\cD$; a \emph{natural transformation} from $F$ to $G$, denoted by $F\Rightarrow G$, is a smooth map $\tau:C \to \bD$ such that the following commutes for any $\phi\in \bC$,
\[
\xymatrix{
F(x) \ar[r]^{F(\phi)}\ar[d]^{\tau(x)} & F(y)\ar[d]^{\tau(y)} \\
G(x) \ar[r]^{G(\phi)} & G(y)}
\]
In the groupoidal situation, every natural transformation is invertible. 

A \emph{generalized map} from $\cC$ to $\cD$ is a diagram of smooth functors as follows 
\[f:\cC \stackrel{F}{\leftarrow}\cE \stackrel{G}{\to }\cD,\]
where $F$ is an equivalence in the above sense. 
If $G$ is also an equivalence, then $\cC$ and $\cD$ are called \emph{Morita equivalent} to each other. 
A \emph{refinement of the generalized map $f$} is a diagram $\cC \stackrel{F'}{\leftarrow}\cE' \stackrel{G'}{\to }\cD$, such that we have the following diagram
\[
\xymatrix{
\cC & \cE\ar[l]_{F} \ar[r]^{G} \ar@{}[ld]^(.15){}="a"^(.5){}="b" \ar@{=>} "a";"b" \ar@{}[rd]^(.15){}="c"^(.5){}="d" \ar@{=>} "c";"d"  & \cD \\
& \cE'\ar[u]_{I}\ar[ru]_{G'}\ar[lu]^{F'} &}
\]
where $I:\cE'\to \cE$ is an equivalence and there are natural transformations from $F\circ I$ to $F'$ and $G\circ I$ to $G'$. 
Two generalized maps are then declared to be \emph{equivalent} if there is a common refinement; this is indeed an equivalence relation \cite[Lemma 10.3.4]{polyfold}. 
Equivalence classes of generalized maps will simply be referred to as \emph{maps}.
Note that a map from $\cC$ to $\cD$ induces a well-defined map $|\mathfrak{f}|:|\cC|\to |\cD|$ between the orbits spaces. 

We point out that maps, i.e.\ equivalence classes of generalized maps, naturally define morphisms between proper \'etale Lie groupoids, with the composition of $[\cC_1\leftarrow \cE_1\to \cC_2]$ and $[\cC_2\leftarrow \cE_2\to \cC_3]$ defined as 
$[\cC_1 \leftarrow \cE_1\times_{\cC_2}\cE_2\to \cC_3]$, where $\cE_1\times_{\cC_2}\cE_2$ is the pull back of $\cE_1\to \cC_2\leftarrow \cE_2$ \cite[Theorem 10.2.2]{polyfold} (c.f.\ the notion of ``weak fibered product'' recalled in \S\ref{sec:fiber_products}), 
$\cE_1\times_{\cC_2}\cE_2\to \cC_1$ is the composition of $\cE_1\times_{\cC_2}\cE_2\to \cE_1$ and $\cE_1\to \cC_1$, and $\cE_1\times_{\cC_2}\cE_2\to \cC_3$ is the composition of $\cE_1\times_{\cC_2}\cE_2\to \cE_2$ and $\cE_2\to \cC_3$; the arrows can be schematically depicted as follows:
\begin{equation*}
\begin{tikzcd}
      &                             & \cE_1\times_{\cC_2}\cE_2 \arrow[ld] \arrow[rd] &                             &       \\
      & \cE_1 \arrow[rd] \arrow[ld] &                                                & \cE_2 \arrow[ld] \arrow[rd] &       \\
\cC_1 &                             & \cC_2                                          &                             & \cC_3
\end{tikzcd}
\end{equation*}

\noindent
In this way, we get a category of proper \'etale Lie groupoids.

\begin{defn}[{\cite[Definitions 1.47 and 1.48]{Ruan07}}]
\label{defn:orbifold}
	An \emph{orbifold structure} $(\cC,\alpha)$ on a paracompact Hausdorff space $W$ is a proper \'etale Lie groupoid $\cC$ with a homeomorphism $\alpha\colon|\cC|\to W$. 
	Two orbifold structures $(\cC,\alpha),(\cD,\beta)$ are said to be \emph{equivalent} if there is a Morita equivalence $\cC\stackrel{\mathfrak{f}}{\to} \cD$ such that the following diagram commutes.
	\begin{equation}
	\label{eqn:commutative_diagr_def_orbifold}
	\xymatrix{
	|\cC| \ar[r]^{\alpha}\ar[d]^{|\mathfrak{f}|} & W \ar[d]^{\Id}\ar[d]\\
	|\cD| \ar[r]^{\beta} & W
    } 
	\end{equation}
	An \emph{orbifold} is a paracompact Hausdorff space $W$ equipped with an equivalence class of orbifold structures $(\cC,\alpha)$.
\end{defn}

\begin{example}
    \label{exa:global_quotient_orbifold}
    Let $G$ be a finite group acting by diffeomorphisms on a smooth manifold $M$.
    Then, the global quotient $M/G$ has a natural orbifold structure, called \emph{global quotient orbifold}, represented by the orbifold structure given by the action groupoid $G\ltimes M$ described in \Cref{exa:action_groupoid} and the obvious identification $\vert G \ltimes M\vert = M/G$. \END
\end{example}

A \emph{morphism} from the orbifold structure $(W,\cC,\alpha)$ to $(V,\cD,\beta)$ consists of a continuous map $f:W\to V$ and a generalized map $\mathfrak{f}:\cC\to \cD$ such that $\beta\circ|\mathfrak{f}|=f\circ \alpha$. Note that $f$ is completely determined by $\mathfrak{f}$ and the two orbifold structures.
Two morphisms $(f,\mathfrak{f}):(W,\cC,\alpha)\to (V,\cD,\beta)$ and $(g,\mathfrak{g}):(W,\cC',\alpha')\to (V,\cD',\beta') $ between orbifold structures for orbifolds $W $ and $V$ are \emph{equivalent} iff $f=g$ and there exist generalized isomorphisms $\mathfrak{h}_1:\cC\to \cC',\mathfrak{h}_2:\cD\to \cD'$ such that $|\mathfrak{h}_2|\circ |\mathfrak{f}|=|\mathfrak{g}|\circ |\mathfrak{h}_1|$, as in the following diagram.
\begin{equation*}
\begin{tikzcd}
 & \vert\cC'\vert \arrow[rrd, "\alpha'"] \arrow[dd, "|\mathfrak{g}|"'] &  &                     \\
\vert\cC\vert \arrow[ru, "|\mathfrak{h}_1|"] \arrow[dd, "|\mathfrak{f}|"] \arrow[rrr, "\alpha"] &                                                                     &  & W \arrow[dd, "f=g"] \\
& \vert\cD'\vert \arrow[rrd, "\beta'"]                                &  &                     \\
\vert\cD\vert \arrow[ru, "|\mathfrak{h}_2|"] \arrow[rrr, "\beta"]                               &                                                                     &  & V                  
\end{tikzcd}
\end{equation*}
One can check that this is indeed an equivalence relation.

\begin{remark}
    What was discussed until now suggests defining the \emph{orbifold category} as an honest category having as objects equivalence classes of orbifold structures and as morphisms the equivalence classes of morphisms between orbifold structures.
    However, a more meaningful way of describing the morphisms space is keeping the information of equivalences, by realizing the set of equivalence classes of morphisms as the orbit space of another groupoid. 
    In particular, the $C^k-$ and $C^\infty-$morphisms between orbifolds form a Banach and Fr\'echet orbifold, respectively. 
    This gives the foundations for describing the structure of moduli spaces of holomorphic curves in an orbifold, which has been done by Chen \cite{Chen06}.
    This bi-categorical approach to morphisms will be detailed in \S \ref{ss:morphism}.
\end{remark}

\begin{defn}[{\cite[Definition 2.25]{Ruan07}}]\label{def:vb}
	A \emph{vector bundle} over a proper \'etale Lie groupoid $\cC$ consists of a vector bundle $\pi:E\to C$ and a smooth $\cC$-action $\mu: \bC \tensor[_s]{\times}{_\pi}  E = \{(\phi,e)\in\bC\times E \vert s(\phi)=\pi(e)\}\to E$ such that the following holds:
	\begin{enumerate}
		\item $\pi(\mu(\phi,e))=t(\phi)$,
		\item $\mu$ is a linear isomorphism on fibers of $E$,
		\item $\mu(\Id_x,e)=e$ for $x\in C,e\in E_x$,
		\item $\mu(\phi\circ \psi,e)=\mu(\phi,\mu(\psi,e))$. 
	\end{enumerate}
\end{defn}

Now, one can define a category $\cE$ as follows: $\Obj(\cE)=E$ and $\Mor(\cE)= \bC \tensor[_s]{\times}{_\pi}  E$, with structure maps $s(\phi,w)=w$ and $t(\phi,w)=\mu(\phi,w)$. 
Such $\cE$ is actually a proper \'etale Lie groupoid equipped with a smooth functor $\pi:\cE \to \cC$, which is a vector bundle on both object and morphism levels (where on the latter it is given by the projection onto the first factor). 
It is easy to check that Definition \ref{def:vb} is equivalent to having a smooth functor $\pi:\cE\to \cC$ that is a vector bundle on both object and morphism levels. 
One can similarly define generalized morphisms between vector bundles over Lie groupoids and introduce Morita equivalences between them. 
Then, one can also naturally define \emph{orbibundle structures}, \emph{orbibundles}, and \emph{orbibundle morphisms}, just as done for the base orbifold case.
In particular, for an orbifold $W$ modeled on an orbifold structure $(\cC,\alpha)$, the $\cE$ constructed as above represents an orbifold $X$, which is moreover an orbibundle with a natural projection map $X\to W$ induced by the smooth functor $\pi: \cE\to\cC$.

Given a functor $F:\cC\to \cD$, one can pull back any vector bundle $\pi\colon \cE\to\cD$ to a vector bundle $F^*\cE$ over $\cC$, which on the object level is $C \tensor[_F]{\times}{_\pi} E$, i.e.\ the pullback of the vector bundle on the object level. 
Moreover, if $F:\cC\to \cD$ is an equivalence and $\cE$ is a vector bundle over $\cC$, then one can define the pushforward $F_*\cE$ over $\cD$ \cite[\S 2.2]{Ruan07}. 
As a consequence, one can pullback a vector bundle by a generalized map, and hence pull back an orbibundle by an orbifold morphism \cite[Theorem 2.43]{Ruan07}.

\subsubsection{Differential forms and integration.}
Two particular examples of orbibundles are the tangent and cotangent bundles $TW, T^*W$ of an orbifold $W$ \cite[Theorem 8.1.1, 10.1.7]{polyfold}. 
One can then define the space of differential forms on $W$ as the space of smooth sections of the bundle $\wedge^*T^*W$, i.e.\ the space of smooth functors $\sigma: \cC \to \cE$ such that $\pi\circ \sigma = \Id$, where $\pi:\cE\to \cC$ is a proper \'etale Lie groupoid bundle representing the orbibundle $\wedge^*T^*W \to W$. By \cite[Theorem 11.2.1]{polyfold}, the differential graded algebra of differential forms is canonically isomorphic for different choices of orbifold structures, which yields a well-defined and functorial construction of the de Rham complex on orbifolds. 

An alternative (and equivalent) point of view on differential forms on orbifolds is the following.

\begin{defn}[{\cite[Definition 8.2.1]{polyfold}}]
	Let $\cC$ be a proper \'etale Lie groupoid. 
	For every $\phi\in \bC$, there is a local diffeomorphism $\widehat{\phi}=t\circ (s^{-1})\vert_{U}$ from a neighborhood $U$ of $s(\phi)$ to a neighborhood of $t(\phi)$. 
	The space of \emph{differential $k$-forms on $\cC$} is then
	\[
	\Omega^k(\cC)=\left\{\omega\in \Omega^k(C)\left| \widehat{\phi}^*\omega=\omega, \forall \phi \in \bC\right. \right\}.
	\]
	As the exterior derivative $\rd$ on $C$ naturally restricts to $\Omega^*(\cC)$, one defines the \emph{de Rham complex} of $\cC$ to be $(\Omega^*(\cC),\rd)$. 
	Moreover, as Morita equivalences induce isomorphisms of de Rham complexes \cite[\S 11.2]{polyfold}, one can define the \emph{de Rham complex $\Omega^*(W)$ of an orbifold $W$} to be that of any orbifold structure representing it\footnote{This is a slightly imprecise version of the definition. 
	A more accurate interpretation of the de Rham complex is as a functor from the category of equivalent orbifold structures of $W$, where a morphism is an equivalence of orbifold structures. The functoriality of the above definition follows from \cite[Lemma 11.2.5]{polyfold}.}.
\end{defn}

One can then define the \emph{de Rham cohomology} of $W$ as the cohomology of $(\Omega^*(W),\rd)$. As it turns out, this is actually isomorphic to the singular cohomology with real coefficients of the underlying topological space of $W$ \cite[\S 2.1]{Ruan07}. 
Therefore, we will abuse notation and just denote the de Rham cohomology by $H^*(W)$. 
In this and all the sections that follow, we will also make another notational simplification, namely, we will denote with the same symbols differential forms on an orbifold $W$ and their lifts to the object space $C$ of a proper \'etale Lie groupoid $\cC$ representing $W$. 

A differential form $\omega$ on $\cC$ is called \emph{compactly supported} if $|\supp \omega|:=\supp \omega/ \bC$ is compact. 
Then, one can also consider the de Rham complex of compactly supported differential forms $\Omega^*_c(W)$. 
Both the de Rham complexes, i.e.\ $\Omega^*(W)$ and $\Omega^*_c(W)$, satisfy, under (proper, in the compactly supported case) orbifold morphisms, the same functoriality properties as de Rham complexes in the smooth manifold setting. 

Moreover, we also have an integration theory, for which an equivalent of the smooth Stokes' theorem holds. 
To introduce it, we first recall the concept of ``local uniformizers'', which are local charts of orbifolds that can be used to recover Satake's ``V-manifold" from a groupoidal orbifold \cite[Theorem 1.45]{Ruan07}.

\begin{defn}[{\cite[Definition 7.1.20]{polyfold}}]
\label{defn:loc_unif}
	Let $\cC$ be a proper \'etale Lie groupoid and $x\in C$. 
	A \emph{local uniformizer} around $x$ is a smooth and fully faithful functor 
	\[\Psi_x: \bC_x\ltimes U_x\to \cC,\]
	with $U_x\subset C$ a neighborhood of $x$ and $\bC_x\subset \bC$ the \emph{isotropy group} of $x$, given by all the finitely-many (due to the proper property) $\phi\in \bC$ such that $s(\phi)=t(\phi)=x$, such that the following holds:
	\begin{enumerate}
		\item\label{inclusion} on the object level $\Psi_x$ is the inclusion $U_x\to C$,
		\item\label{orbit_homeo} $|\Psi_x|:U_x/\bC_x=\vert \bC_x\ltimes U_x\vert\to |\cC|$ is a homeomorphism onto an open subset of $|\cC|$.
	\end{enumerate}
\end{defn}
To be precise, \eqref{orbit_homeo} actually follows from fully-faithfulness and \eqref{inclusion}, but it has been stated directly in the definition for clarity.

Local uniformizers always exist, see e.g.\ \cite[Proposition 7.1.19]{polyfold}.
An orbifold/groupoid is called \emph{effective} if, for every $x$, the $\bC_x-$action in the local uniformizer around $x$ is effective. 
We also point out that, in the effective case, one can moreover prove that local uniformizers exist with $U_x$ the unit ball in $\RR^n$ and $\bC_x$ a subgroup of $O(n)$, where the latter acts on $U_x$ by the standard linear action.

An \emph{orientation} of an orbifold $W$ is a section of the projectivization of the line orbibundle $\wedge^{\dim W} TW$ over $W$.
If such a section exists, the orbifold is said to be \emph{orientable}.
Then, for a $G$-action on an oriented smooth manifold $M$, the global quotient groupoid $W=\vert G\ltimes M\vert$ is orientable if and only if the $G$ acts by orientation preserving diffeomorphisms of $M$. 
In this case, given $\omega\in \Omega^{\dim M}_c(G\ltimes M)$, i.e.\ given a $G$-invariant compactly supported form $\omega$ on $M$, we define
\[\int_{M/G} \omega \coloneqq \frac{1}{|G|}\int_M \omega.\]
In the general case of oriented orbifolds, $W$ has a cover $\{U_x/\bC_x\}_{x\in W}$, for $\Psi_x\colon \bC_x\ltimes U_x\to \bC$ a local uniformizer around $x$. 
Then we can pick a partition of unity $\{\overline{\rho}_x\}$ subordinate to $\{U_x/\bC_x\}$, such that $\overline{\rho}_x$ is induced from a $\bC_x$-invariant compactly supported smooth function $\rho_x$ on $U_x$. Then we define
$$\int_W \omega \coloneqq \sum \frac{1}{|\bC_x|}\int_{U_x}\rho_x \omega.$$
By \cite[\S 9.5]{polyfold}, the integration is well-defined, i.e.\ independent of every choice in the construction. 
Using this and the resulting Stokes' theorem, we have the Poincar\'e pairing on an oriented orbifold $W$ of dimension $n$ given by
\begin{equation}
\label{eqn:PD_singular_cohomol}
H^*(W)\otimes H_c^{n-*}(W)\to \RR, \quad \alpha \otimes \beta \mapsto \int_{W} \alpha \wedge \beta.
\end{equation}
Such a pairing is non-degenerate for any orbifold by the Mayer-Vietoris arguments analogous to those in \cite[\S 5]{Bott82}.

\subsubsection{Fiber products of orbifolds.}
\label{sec:fiber_products}

Given two groupoid functors $F:\cW\to \cU,G:\cV\to \cU$, one can define the \emph{weak fiber product} $\cW\times_\cU \cV$ by
$$\Obj(\cW\times_\cU \cV)\coloneqq W\times_U \bU \times_U V =\{ (w,g,v)\in W\times \bU \times V \mid g:F(w)\to G(v) \}.$$
A morphism from $(w,g,v)$ to $(w',g',v')$ is pair of morphisms $(f,h)\in \bW\times \bV$ from $w,v$ to $w',v'$ respectively, such that $g\circ F(f)=G(h)\circ g'$.
Then, the following diagram is commutative up to a natural transformation.
\[\xymatrix{
\cW \times_{\cU} \cV \ar[r]\ar[d] & \ar[d] \cW \ar[d]  \ar@{}[ld]^(.3){}="a"^(.7){}="b" \ar@{=>} "a";"b"   \\
\cV \ar[r] & \cU}\]
Notice that, when the maps $F\vert_{\Obj}\colon W \to U$ and $G\vert_{\Obj}\colon V\to U$ are transverse to each other, $\Obj(\cW\times_\cU \cV)$ and $\Mor(\cW\times_\cU \cV)$ are transversely cut-out, and
$\cW\times_{\cU}\cV$ is a proper \'etale Lie groupoid.

Now, consider orbifolds $X,Y,Z$, and two orbifold maps $f:X\to Z, g:Y\to Z$ between them, such that $\mathrm{im} (\rd f(x)) + \mathrm{im} (\rd g(y))=T_zZ$\footnote{One can interpret the relation on the level of objects with any choice of local uniformizers. It is straightforward to check that the relation does not depend on the choice of local uniformizers.} for any $x,y$ with $f(x)=g(y)=z$. 
Suppose that $f,g$ are represented by $F:\cW \to \cU_1,G:\cV\to \cU_2$, where $\cW,\cV$ represent $X,Y$ respectively and both $\cU_1,\cU_2$ represent $Z$, and that we have a Morita equivalence $\cU_1\leftarrow \cU \to \cU_2$. 
Then $f,g$ are also represented by $\cW\times_{\cU_1}\cU \to \cU$ and $\cV\times_{\cU_2}\cU\to \cU$. 
Hence one can form their weak fiber product, and one can explicitly check that this gives rise to a well-defined proper \'etale Lie groupoid up to equivalence, whose orbit space shall be denoted by $X\tensor[_f]{\times}{_g} Y$ and called the \emph{orbifold fiber product}.

\medskip

In the sections that follow, we will also need to count zero-dimensional compact orbifolds in order to define the differential for the symplectic cohomology. 
Since the critical points (periodic orbits) could be in our situation orbifold points, the differential is defined through pullback and pushforward instead of just counting the moduli space. 
A typical situation is as follows. 

Assume that we have two orbifold maps/functors $\bullet/H_1\stackrel{s}{\leftarrow} \bullet/G\stackrel{t}{\to}\bullet/H_2$, where the point-orbifold $\bullet/G:=G\ltimes \{\mathrm{pt}\}$ will be the moduli spaces of Floer trajectories, $\bullet/H_1,\bullet/H_2$ are critical points and $s,t$ are the evaluation maps at the two ends. 
Then, the contribution of the diagram to the differential will be given by $t_*\circ s^*\colon H^*(\bullet/H_1)\to H^*(\bullet/H_2)$, where $t_*$ is defined by $\int_{\bullet/H_2} t_*(\alpha)\wedge \beta = \int_{\bullet/G} \alpha \wedge t^*\beta$ for all $\alpha\in H^{*}(\bullet/G)$ and $\beta\in H^*(\bullet/H_2)$ (note these are necessarily supported in degree $0$). 
As the dual of $1\in H^*(\bullet/K)$ with respect to the duality \eqref{eqn:PD_singular_cohomol} is $|K|\cdot 1$, we have 
\begin{equation}\label{eqn:pushpull}
    t_*\circ s^* (1) =\frac{|H_2|}{|G|}\cdot 1.
\end{equation}

The key property that will be used in the proof of the fact that the differential squares to zero is the following.

\begin{prop}\label{prop:composition}
Given diagrams $\bullet/H_1\stackrel{s_1}{\leftarrow}\bullet/G_1\stackrel{t_1}{\to}\bullet/H_2 \stackrel{s_2}{\leftarrow} \bullet/G_2\stackrel{t_2}{\to}\bullet/H_3$, consider the associated diagram for the fiber product $\bullet/H_1\stackrel{s_3}{\leftarrow}\bullet/G_1\times_{\bullet/H_2}\bullet/G_2 \stackrel{t_3}{\to}\bullet/H_3$. 
Then,  $(t_2)_*\circ s_2^*\circ(t_1)_*\circ s_1^*=(t_3)_*\circ s^*_3$. In other words,  we have 
$$(t_3)_*\circ s^*_3(1)=(t_2)_*\circ s_2^*\circ(t_1)_*\circ s_1^*(1) = \left((t_2)_*\circ s_2^*(1)\right) \cdot \left((t_1)_*\circ s_1^*(1)\right)$$
\end{prop}

\begin{proof}
By \eqref{eqn:pushpull}, we have $(t_2)_*\circ s_2^*\circ(t_1)_*\circ s_1^*(1)=\frac{|H_2||H_3|}{|G_1||G_2|}\cdot 1$. 
On the other hand, $\bullet/G_1\times_{\bullet/H_2}\bullet/G_2$ is the groupoid $(G_1\times G_2)\ltimes H_2$ given by the action $(g_1,g_2)\cdot h = s_2(g_2)\circ h \circ t_1(g^{-1}_1)$. 
Then, at the morphisms level, $t_3$ is defined to be $(g_1,g_2)\mapsto g_2$. 

Now, $H_2$ has a decomposition into orbits $\overline{h}_1\cup\ldots \cup \overline{h}_k$ w.r.t.\ the $G_1\times G_2$ action. 
Let then $\stab_i$ denote the isotropy group of an element in $\overline{h}_i$. 
Hence, $|\stab_i|\cdot |\overline{h}_i|=|G_1| \cdot |G_2|$ and  $\bullet/G_1\times_{\bullet/H_2}\bullet/G_2$ is equivalent to $\cup_{i=1}^k \bullet/\stab_i$. Therefore we have
\[
(t_3)_*\circ s_3^*(1)=\sum_{i=1}^k \frac{|H_3|}{|\stab_i|}\cdot 1 = \sum_{i=1}^k\frac{|H_3||\overline{h}_i|}{|G_1||G_2|}\cdot 1= \frac{|H_2||H_3|}{|G_1||G_2|}\cdot 1=(t_2)_*\circ s_2^*\circ(t_1)_*\circ s_1^*(1).
\qedhere
\]
\end{proof}

In applications, $(t_2)_*\circ s_2^*\circ(t_1)_*\circ s_1^*$ will correspond to the square of the differential and the fiber product $\bullet/G_1\times_{\bullet/H_2}\bullet/G_2$ to broken curves that are boundaries of a one-dimensional moduli space. 
Hence, the fact that the differential squares to zero will follow from Stokes' theorem.

\subsection{Symplectic orbifolds} 

We now introduce the main object of interest of this paper, namely symplectic orbifolds, and study their properties in the setting of Liouville orbifolds with convex contact manifold boundaries.
\begin{defn}
	A \emph{symplectic orbifold} $(W,\omega)$ is a $2n$-dimensional (oriented) orbifold $W$ equipped with a smooth closed $2$-form $\omega$ such that $\omega^n> 0$. 
	An \emph{exact symplectic orbifold} $(W,\lambda)$ is an orbifold $W$ equipped with a smooth $1$-form $\lambda$, such that $\rd \lambda$ is a symplectic form. 
	An \emph{exact orbifold filling} of a smooth contact manifold $(Y,\xi)$ is an exact symplectic orbifold $(W,\lambda)$, such that $\partial W$ has trivial isotropy and $\partial W=Y$ with $\xi=\ker \lambda|_{\partial W}$, and the Liouville vector field (i.e.\ $Z$ such that $\iota_Z\rd \lambda= \lambda$) points out along $\partial W$.
    In this last case, $\lambda$ is moreover called \emph{Liouville form}.
\end{defn} 

Our primary examples are \emph{global quotients}:
\begin{example}\label{ex:quotient}
	Let $G$ be a finite subgroup of $U(n)$ such that $\ker (g-I) = \{0\}$ for $g\ne \Id  \in G$.
	Then, $G$ preserves the standard Liouville form $\lambda_{\mathrm{std}}:=-\frac{1}{2}\rd r^2\circ J$ on $\CC^n$ and $G$ acts freely on the unit sphere $S^{2n-1}$. 
	The quotient $(S^{2n-1}/G,\xistd:=\ker(\lambda_{\mathrm{std}}\vert_{S^{2n-1}}))$ is a smooth contact manifold with an exact orbifold filling $D^{2n}/G\subset \CC^n/G$, with its orbifold structure given by the action groupoid $G\ltimes \CC^n$ and Liouville form  $\lambda_{\mathrm{std}}=-\frac{1}{2}\rd r^2\circ J$. \END
\end{example} 

\begin{remark}
	For any finite subgroup $G$ of $U(n)$, $\CC^n/G$ is always an exact orbifold. However, if there is a nontrivial $g$ with $\dim \ker(g-I)\ge 1$, the contact boundary $S^{2n-1}/G$ is not smooth, but rather a contact orbifold. In a sequel paper, we will develop the theory for contact orbifolds.
\end{remark}

An \emph{almost complex structure} on a proper \'etale Lie groupoid $\cC$ is an almost complex structure $J$ on $C$, such that $\widehat{\phi}_*J=J$ for any $\phi \in \bC$; recall that $\widehat{\phi}$ denotes the (well defined near $s(\phi)$) local diffeomorphism $t\circ s^{-1}\vert_{Op(s(\phi))}\colon Op(s(\phi))\to Op(t(\phi))$, where $Op(\cdot)$ denotes a sufficiently small open neighborhood (whose specific choice is not relevant). 
Similarly, a \emph{Riemannian metric} on $\cC$ is a Riemannian metric $\mu$ on $C$, such that $\widehat{\phi}^*\mu=\mu$. 
Then we can define almost complex structures and Riemannian metrics on orbifolds as the respective structures on orbifold structures (or, more precisely, functors from the category of equivalent orbifold structures). Given a symplectic orbifold, an almost complex structure $J$ is said to be \emph{compatible} with $\omega$ if $\omega(\cdot,J\cdot)$ is a metric, as in the smooth setting. 
Then, by an analogous argument to that of \cite[Proposition 4.1.1]{McDuff17}, the set of compatible almost complex structures is non-empty and contractible.

\begin{prop}\label{prop:sub}
	Let $G\ltimes U$ be a local uniformizer of an orbifold structure $(\cC,\omega)$ representing a symplectic orbifold $(W,\omega)$, around a point $x\in C$. 
	For $g\in G$, the fixed point set $\Fix(g)$ is a symplectic submanifold of $U$. 
	More precisely, if $J$ is a compatible almost complex structure, then $\Fix(g)$ is a $J$-holomorphic submanifold.
	Lastly, denoting by $C(g)$ the centralizer of $g$ and considering its natural action on $\Fix(g)$, the functor $C(g)\ltimes \Fix(g)\to G\ltimes U$ induced by inclusions is an embedding in the sense of \cite[Definition 2.3]{Ruan07}. 
\end{prop}

\begin{remark}[The notion of suborbifold]
	In \cite[Definition 7.1.21]{polyfold},  a subgroupoid of $\cC$ is a full subcategory $\cD$ such that $D$ is a submanifold of $C$; suborbifolds are then defined naturally in \cite[Definition 16.1.15]{polyfold} as subsets in the quotient space of the orbifold that are represented by subgroupoid structure. 
	With this definition, $C(g)\ltimes \Fix(g)$ is not a suborbifold of $G\ltimes U$ in general. 
	Suborbifolds defined as such arise, for instance, from transverse zero sets of sections of orbibundles. 
	On the other hand, \cite[Definition 2.3]{Ruan07} defines the notion of \emph{embedding} of orbifolds and, relying on it, the notion of suborbifolds. 
	In our setting, the functor $C(g)\ltimes \Fix(g)\to G\ltimes U$ \emph{is} an embedding in the latter sense, as stated in \Cref{prop:sub}. 
        (Note that, according to \cite[Definition 2.3]{Ruan07}, this does not imply e.g.\ that the induced map $| \Fix(g)/C(g)|\to |\CC^n/G|$ is injective.)

        \noindent
	A note on conventions in this paper: we will adhere to the notion of suborbifold in \cite{polyfold} and to the notion of embedding in \cite{Ruan07}; in particular, not all orbifold embedding give rise to suborbifolds with this convention.
	(The interested reader can consult \cite{BorBru15,MesWei20} for  detailed discussions concerning these notions.)
\end{remark}	

\begin{proof}
	Firstly, $\Fix(g)$ is a locally closed submanifold of $U$. 
	Now, consider an auxiliary compatible almost complex structure $J$ (that exists by an argument analogous to the one in \cite[Proposition 4.1.1]{McDuff17}, as previously pointed out); in particular, $g_*J=J$. 
    Hence, for $x\in \Fix(g)$ and $v\in T_x \Fix(g)$, one has $g_*Jv=Jv$, i.e.\ $Jv\in T_{g(x)}\Fix(g)$; in other words, $J$ restricts to an almost complex structure on $\Fix(g)$. 
	Since any $J$-holomorphic submanifold is automatically symplectic by compatibility of $J$ with $\omega$, we have that $\Fix(g)$ is a symplectic submanifold of $U$.  
    Lastly, according to \cite[Example 2.5]{Ruan07} $C(g)\ltimes \Fix(g)\to G\ltimes U$ is an embedding in the sense of \cite[Definition 2.3]{Ruan07}.
    (In fact, even its restriction to $\cup_{h\in G} (C(hgh^{-1}) \ltimes \Fix(hgh^{-1}))$ is an embedding.)
\end{proof}

Let $\cC$ be a proper \'etale Lie groupoid, and consider the pullback diagram
$$
\xymatrix{
S_{\cC} \ar[r]\ar[d]^{\beta} & \bC \ar[d]^{(s,t)}\\
C  \ar[r]^{\Delta} & C\times C}
$$
where $\Delta$ is the diagonal map and $S_{\cC}=\{g\in \bC| s(g)=t(g)\}$ is naturally equipped with the map $\beta:S_{\cC}\to C, g \mapsto t(g)=s(g)$. 
Any $h\in \bC$ induces a diffeomorphism $\beta^{-1}(s(h))\to \beta^{-1}(t(h))$ via the action by conjugation $h\cdot g = h g h^{-1}$ for $g\in \beta^{-1}(s(h))$. 
In particular, we can form a \emph{$\cC$-action groupoid} $\cC \ltimes S_{\cC}$ with $\Obj=S_{\cC}$ and $\Mor=\bC\tensor[_s]{\times}{_\beta}S_{\cC}=\{(h,g)\in\bC\times S_{\cC} \vert s(h)=\beta(g)\}$, with source map $s(h,g)=g$ and target map $t(h,g)=hgh^{-1}$.

\begin{defn}[{\cite[Definition 2.49]{Ruan07}}]
The \emph{inertia groupoid} $\wedge \cC$ is the action groupoid $\cC \ltimes S_{\cC}$. 
\end{defn}

One can check that Morita equivalences between two groupoids induce Morita equivalences between their inertia groupoids, hence there is a well defined \emph{inertia orbifold} associated to each orbifold. 
Notice also that, as explained in \cite[Example 2.48]{Ruan07}, $\wedge \cC$ should be understood as the space of constant generalized maps from $\bbS^1$ to $\cC$, i.e.\ the space of constant loops in the orbifold context. We shall explain this in more detail in \S \ref{ss:morphism}.

We now describe a useful decomposition of $\wedge \cC$ into components. 
For this, we define an equivalence relation $\simeq$ on $S_{\cC}$ as follows. 
On each local uniformizer $G\ltimes U$, for $x,y\in U$, any two $g_1\in G_x\subset G$ and $g_2\in G_y\subset G$ are equivalent with respect to $\simeq$ if $g_1,g_2$ are conjugated in $G$. 
More generally, $g_1\simeq g_2$ iff $g_1,g_2$ are connected by a sequence $\{h_1=g_1,h_2,\ldots,h_{n-1}, h_n=g_2\}\subset S_{\cC}$, such that $h_i,h_{i+1}$ are equivalent in a local uniformizer. 
We denote by $(g)$ the equivalence class of $g\in S_{\cC}$ under the just defined relation $\simeq$, and by $T^1$ the set $S_{\cC}/\simeq$ of all such equivalence classes. 
Then $\wedge \cC$ can be decomposed as 
\[\wedge \cC  = \bigsqcup_{(g)\in T^1}  \cC^1_{(g)},\] 
where $\cC^{1}_{(g)}$ is the $\cC$-action groupoid on the $(g)$-component of $S_{\cC}$.
Notice in particular that $\cC^1_{(\Id)}$ is naturally isomorphic to $\cC$.

Let's now go back to the symplectic setting.
Since a local uniformizer of $\wedge \cC$ around $g\in \bC$ is given by $C(g)\ltimes \Fix(g)$ according to \cite[p. 76]{Kaw78} (cf.\ \cite[Lemma 3.1.1]{CheRua04}), Proposition \ref{prop:sub} implies that $\wedge \cC$ is a(n exact) symplectic orbifold if $\cC$ is a(n exact) symplectic orbifold. 
What's more, according to such description of local uniformizers for the inertia groupoid $\wedge \cC$, each connected component $\cC^1_{(g)}$ gives, via the map $\beta\colon \cS_{\cC}\to C$, an immersion into $C$.

\begin{prop}
\label{prop:decomposition}
	Let $W$ be an exact orbifold filling of a smooth contact $Y$. 
	Then, $W$ is (up to Liouville homotopies) a composition of:
	\begin{enumerate}[label=(\roman*)]
	    \item exact orbifold fillings $\sqcup_{i\in I} \CC^n/G_i$ of contact manifolds $\sqcup_i(S^{2n-1}/G_i,\xistd)$, with $G_i\subset U(n)$ for $i$ in a finite set $I$ of indices;
	    \item a (smooth) exact symplectic cobordism from $\sqcup_i(S^{2n-1}/G_i,\xistd)$ to $Y$.
	\end{enumerate}
\end{prop}
\begin{proof}
	Let $\cC$ be a groupoid representing $W$. 
    Recall that the inertia groupoid $\wedge \cC$ contains $\cC$ as a component; the other components come from $(g)\ne (\Id)$. 
	Since $\partial W=Y$ is a smooth contact manifold, $\cC^1_{(g)}$ must have closed orbit space if $(g)\ne (\Id)\in T^1$. 
	Hence, exactness of the symplectic form on $\wedge \cC$ induced from \Cref{prop:sub} via the above description of local uniformizers for $\wedge \cC$ implies that $\vert\cC^1_{(g)}\vert$ must be discrete for $(g)\ne (\Id)$. 
	In particular, singularities of $W$ must be isolated.
	
	Let now $\{x_i\}_{i\in I}\subset \cC$ be a set such that $\{|x_i|\}_{i\in I}$ is the set of points in $W=|\cC|$ with nontrivial isotropy group. 
	Notice that $I$ must be a finite set by the compactness of $W$. 
	Let also $G_i$ be the isotropy group at $x_i$ and $G_i^*$ the set $G_i\setminus \{\Id\}$;
	then, $\wedge \cC$ is Morita equivalent to $\cC \cup_i G_i\ltimes G_i^*$, where $G_i$ acts on $G^*_i$ by conjugation.  
	Moreover, up to Morita equivalence, one has 
	\[G_i\ltimes G_i^*=\bigsqcup_{(g)\in \mathrm{Conj}^*(G_i)} \bullet/C(g),\]
	where $\mathrm{Conj}^*(G_i)$ is the set of nontrivial conjugacy classes and $\bullet/C(g)=C(g)\ltimes \{\mathrm{pt}\}$. 

    Lastly, using (the equivariant version of) Moser's trick and the fact that $H^1(\CC^n/G_i)=0$, one can see that, near any $x\in \cC$ with nontrivial isotropy $G_i$, there are local uniformizers with Liouville form as in Example \ref{ex:quotient}. 
    This concludes the proof.
\end{proof}

\begin{remark}
The singular locus need not be made of only isolated points in the case of non-exact orbifold symplectic fillings, even if one restricts to the smooth boundary setting, as the following example shows.
\end{remark}

\begin{example}
    Let $(a_1,\ldots,a_n)$ be an $n$-tuple of positive natural numbers, and consider the action $\rho\colon \bbC^*\to U(n)$ on $\bbC^n\setminus\{0\}$ given by $\rho(\lambda)=diag(\lambda^{a_1},\ldots,\lambda^{a_n})$.
    The quotient $(\bbC^{n}\setminus\{0\})/\bbC^*$ is then naturally an orbifold, called \emph{weighted projective space}  $W = W\bbP(a_1,\ldots,a_n)$, which is effective if and only if  $(a_1,\ldots,a_n)$ are coprime. 
    Moreover, the action $\rho|_{S^1}$ is a Hamiltonian circle action, and $W$ can be viewed as the symplectic reduction. 
    In particular, $W$ carries a canonical symplectic form $\omega$ and is hence a symplectic orbifold. 
    Up to rescaling by a factor depending on the $a_i$'s, one can moreover assume that $\omega/2\pi$ is an integer cohomology class, i.e.\ in the image of $H^*_{\orb}(W,\ZZ) \to H^*_{\orb}(W,\bbR)$; here, $H^*_{\orb}(W;R)$ denotes the \emph{orbifold cohomology} of $W$ with $R$-coefficients, defined as the (singular) homology of the classifying space $BW$ with $R$-coefficients (see e.g.\ \cite[p. 24-25]{Ruan07} and the footnote \ref{foot:2} in the Introduction). 
    Given an integral lift $e\in H^*_{\orb}(W,\bbZ)$, it is well known (see e.g.\ \cite[Theorem 7.1.6]{boyergalickibook}), and goes under the name of \emph{Boothby-Wang contact orbifold}, that the total space of the principal $\bbS^1$-orbibundle $\pi \colon M \to W$ with Euler class $e$ is in fact a contact orbifold, that we denote here by $(M,\alpha)$, satisfying $d\alpha = \pi^* \omega$.
    According to \cite[discussion after Corollary 1.6]{KegLan20}, for any choice of $(a_1,\ldots,a_n)$, one can find an integral lift $e$ such that the resulting $M$ is actually a smooth manifold.
    In this case then we can consider the complex line orbibundle $P=M\times_{\bbS^1}\bbC\to W$ associated to the principal $\bbS^1$-bundle $M\to W$ (i.e.\ given by the action $\lambda\cdot(p,z)=(\overline{\lambda}\cdot p, \lambda z)$ of $\bbS^1$ on $M\times \bbC$); the unit disc bundle of $P$ is then a symplectic orbifold filling of $(M,\alpha)$, cf.\ \cite[Proposition 4.4]{NiePas09}.
    Now, $P$ has $ M\times_{\bbS^1}\{0\} = W$ as symplectic suborbifold.
    In particular, if, for instance,  $(a_{i_1},\ldots,a_{i_k})$ are not coprime for $\{i_1,\ldots,i_k\}\subset \{1,\ldots,n\}$, then $W$, hence $P$ too, will have non-isolated orbifold singular locus. \END
\end{example}

\subsection{Chen-Ruan orbifold cohomology}\label{SS:CR}
From a string theory point of view, the correct cohomology theory of an almost complex orbifold is the Chen-Ruan cohomology, equipped with a ring structure, introduced in \cite{CheRua04}, which can be viewed as the constant part of the orbifold quantum product introduced in \cite{CheRua01}. 
Therefore, in the definition of orbifold symplectic cohomology, the natural substitute of the zero action part of symplectic cohomology, i.e.\ the regular/Morse cohomology of the filling, is the Chen-Ruan cohomology of the orbifold filling. 
In this section we thus recall its definition and main properties that we will need in the rest of the paper.

Given an almost complex proper \'etale Lie groupoid $\cC$, the isotropy group $\bC_x$ acts on $T_xC$ giving rise to a representation $\rho_x:\bC_x\to GL(n,\CC)$. Since $g\in \bC_x$ has finite order, $\rho_x(g)$ is conjugated to a diagonal matrix with eigenvalues $\{e^{2\pi \i m_{1}/o(g) },\ldots, e^{2\pi \i m_{n}/o(g)}\}$, where $o(g)$ is the order of $\rho_x(g)$ and $0\le m_{i}<o(g)$. 
Moreover, such a matrix only depends on the conjugacy class of $g$. 
In other words, we can define a function\footnote{We follow \cite{McKay}  to use $\age$ instead of $\iota$ in \cite[\S 4.2]{CheRua04}.} $\age:\wedge \cC \to \QQ$ by 
\begin{equation}\label{eqn:age}
    \age(g)=\sum_{i=1}^n \frac{m_{i}}{o(g)},
\end{equation}
with the property that $\age(g)$ only depends on $(g)\in T^1$, as $\age(g)=\age(g')$ if $g,g'\in G$ are conjugated in a local uniformizer $G\ltimes \CC^n$.

\begin{defn}[{\cite[Definition 3.2.3]{CheRua04},\cite[Definition 4.8]{Ruan07}}]
\label{defn:CR_cohomol}
	The \emph{Chen-Ruan cohomology} of an almost complex proper étale Lie groupoid $\cC$ is 
	\[H^*_{\CR}(\cC)=\bigoplus_{(g)\in T^1} H^*(\cC^1_{(g)})[-2\age(g)] \, ,\]
    where all coefficients are intended in $\mathbb{R}$.
\end{defn}

\noindent
Here, $[-2\age(g)]$ denotes a degree shift of $-2\age(g)$. In particular, $H^*_{\CR}$ has, in general, a rational grading.
By \cite[Theorem 4.13]{Ruan07}, the Chen-Ruan cohomology of an almost complex groupoid is invariant under Morita equivalence, hence it defines a cohomology theory for almost complex orbifolds. 
Lastly, we point out that \emph{compactly supported} Chen-Ruan cohomology groups $H^*_{CR,c}(\cC)$ for an almost complex proper étale groupoid $\cC$ can be defined in a completely analogous fashion using the compactly supported cohomology of the $\cC^1_{(g)}$'s, and this gives a well defined invariant of almost complex orbifolds as well; c.f.\ \cite[Remark 3.2.5]{CheRua04} and \cite[Remark 4.10]{Ruan07}.

Let now $I$ be the map sending $g$ to $g^{-1}$; then $I$ induces a functor, still denoted by $I$,  from $\wedge \cC$ to $\wedge \cC$, mapping the component $\cC^1_{(g)}$ to $\cC^{1}_{(g^{-1})}$.
According to \cite[Proposition 4.11]{Ruan07}, one can then equip the Chen-Ruan cohomology with the following non-degenerate pairing:
\begin{equation}
\label{eqn:CR_pairing}
\langle \cdot, \cdot \rangle_{\CR} \colon H^*_{\CR}(\cC)\otimes H^{2n-*}_{\CR,c}(\cC)\to \RR, \quad 
\langle \alpha,\beta \rangle_{\CR}=\int_{\wedge \cC} \alpha \wedge I^*\beta.
\end{equation}

The most interesting part of the Chen-Ruan cohomology is its product structure, which we shall review next. 
This ring structure will correspond to the ring structure on the zero-action part of symplectic cohomology when we give a tentative definition of the ring structure on symplectic cohomology.

First, similarly to $S_{\cC}$, one can define 
\[S^k_{\cC}=\left\{(g_1,\ldots,g_k)\left|g_i\in \bC, s(g_1)=t(g_1)=\ldots=s(g_k)=t(g_k)\right.\right\},\]
which is equipped with a map $\beta_k(g_1,\ldots,g_k)=s(g_1)$. 
Now, $S^k_{\cC}$ admits a left action of $h\in \bC$ by $h(g_1,\ldots,g_k)=(hg_1h^{-1},\ldots,hg_kh^{-1})$ if $s(h)=s(g_1)$. 
The resulting $\cC$-action groupoid $\cC\ltimes S^k_{\cC}$ is denoted by $\cC^k$ (and, in particular, $\cC^1=\wedge \cC$). 
Moreover, similarly to what was done for $\wedge \cC$, we can define $T^k$ to be the equivalence classes of $(g_1,\ldots,g_k)$ up to conjugacy in a chain of charts, this results in a similar decomposition in disjoint action groupoids: 
\[\cC^k=\sqcup_{(\bg)\in T^k} \cC^k_{(\bg)} \, . \]

Let's now denote $T^k_0\subset T^k$ the subset of equivalence classes $(\bg)=(g_1,\ldots,g_k)$ such that the product $g_1\ldots g_k$ is $\Id$, and define 
\[\cM_k(\cC):=\bigsqcup_{(\bg)\in T^k_0} \cC^k_{(\bg)}.\]
Then, following \cite[Example 2.50]{Ruan07}, $\cM_k(\cC)$ can be viewed as the space of constant orbifold maps from an orbifold Riemann sphere to $\cC$.
Namely, for the component $\cC^k_{(\bg)}$ with $(g)=(g_1,\ldots,g_k)$, the orbifold Riemann sphere is equipped with $k$ fixed marked points, and only marked points can have non-trivial isotropy $\ZZ/o(g_i)\ZZ$, where $o(g_i)$ is the order of $g_i$. 
Notice also that $\cM_k(\cC)$ comes equipped with $k$ different evaluation maps $e_1,\ldots,e_k$ to $\wedge \cC$, given by $e_i((\bg))=g_i$, for $(\bg)=(g_1,\ldots,g_k)$.

Now, in order to obtain the product structure on the Chen--Ruan cohomology, one needs to look at the moduli spaces of constant maps from $S^2$ with three marked points, i.e.\ $\cM_{3}(\cC)$, and to a certain obstruction orbibundle $E$ over it;
we summarize here the description from \cite[p. 89-90]{Ruan07}, to which we refer for details.

More precisely, the linearized Cauchy Riemann operator defines a Fredholm section\footnote{To be more precise, one should use Sobolev version of  $\Omega^0(f_y^*T\cC)$  and $\Omega^{0,1}(f_y^*T\cC)$, see \S \ref{ss:Morphism}.}  $\overline{\partial}_y\colon \Omega^0(f_y^*T\cC) \to \Omega^{0,1}(f_y^*T\cC)$, where $f_y\colon \cS\to\cC$ is the constant map with image $y\in C$, and $\cS$ is a $2$-dimensional orbifold Riemann surface with genus $0$ and three marked points $z_1,z_2,z_3$ of multiplicities $o_1,o_2,o_3$. 
Notice that $f_y$ can also be seen as an element $(\bg)=(g_1,g_2,g_3)$ of $\cM_3(\cC)$ with $o(g_1)=o_1,o(g_2)=o_2,o(g_3)=o_3$.
The index of $\partbar_y $ is then 
\[
\indexop(\partbar_y) = 2n-2\age(g_1)-2\age(g_2)-2\age(g_3).
\]
As $f_y\in\cC^3_{(\bg)}$, one can also prove that $\ker(\overline{\partial}_y)$ can be canonically identified with $T_{f_y}\cC^3_{(\bg)}$; in particular, such kernel has constant dimension over $\cC^3_{(\bg)}$.
Therefore, so does $\coker(\overline{\partial}_y)$, thus giving a vector bundle $E_{(\bg)}\to \cC^3_{(\bg)}$.
Then, $E\to \cM_3(\cC)$ can just be defined as the vector bundle that restricts to $E_{(\bg)}\to \cC^3_{(\bg)}$ over each $\cC^3_{(\bg)}\subset \cM_3(\cC)$.

\begin{remark}
In this situation, one can actually describe $E_{(\bg)}$
explicitly.
For $(\bg)\in \cC^3$, we consider the vector bundle $e^*T\cC$ over $\cC^3_{(\bg)}$, where $e\colon \cC^3_{(\bg)}\to \cC$ is given by mapping a triple of $g_i$'s to their base point in $C$.
Now, let $N$ be the subgroup of $\bC_{e((\bg))}$ generated by the three elements $g_1,g_2,g_3$ such that $(\bg)=(g_1,g_2,g_3)$.
One can prove (c.f.\ \cite[Lemma 4.5]{Ruan07}) that $N$ is locally constant on  $\cC^3$.
Now, for the orbifold Riemann sphere $\cS = (S^2,(z_1,z_2,z_3),(o_1,o_2,o_3))$\footnote{Following the notation in \cite{Ruan07}, $(S^2,(z_1,z_2,z_3),(o_1,o_2,o_3))$ denotes the orbifold Riemann sphere with three marked points $z_1$, $z_2$, $z_3$ with isotropy groups $\ZZ/o_1\ZZ$, $\ZZ/o_2\ZZ$ and $\ZZ/o_3\ZZ$ respectively.}, where $o_i$ is the order of $g_i$ in $\bC_{e((\bg))}$, we have $\pi_1^{\orb}(\cS)=\{\lambda_1,\lambda_2,\lambda_3\vert \lambda_i^{o_i}=1,\lambda_1\lambda_2\lambda_3=1\}$ (see \S \ref{sec:homotopy_groups} for the definition of $\pi_1^{\orb}$), 
where $\lambda_i$ is a loop around $z_i$.
One then has a natural surjective homomorphism $\phi\colon\pi_1^{\orb}(\cS)\to N$, having kernel of finite index.
Denote then $\Sigma \coloneqq \widetilde{\cS}/\ker(\phi)$, where $\widetilde{\cS}$ is the \emph{smooth} universal cover of the orbifold $\cS$.
One then gets a covering map $p\colon \Sigma \to \cS = \Sigma/N$; moreover, as $N$ contains the relations $g_i^{o_i}=1$, $\Sigma$ is smooth as well.

Consider now a uniformizer chart $\bC_y\ltimes U_y$ around some point $y\in C$.
Then, the constant orbifold morphism $f_y$ lifts to a constant map $\widetilde{f}_y\colon \Sigma \to U_y$.
In particular, $\widetilde{f}_y^*T\cC=T_y\cC$ is a trivial bundle over $\Sigma$.
One can also lift $\partbar_y$ to an operator $\partbar_{\Sigma}\colon \Omega^0(\widetilde{f}^*_y T\cC)\to \Omega^{0,1}(\widetilde{f}^*_y T\cC)$,
in such a way that $\partbar_y$ is just the $N$-invariant part of $\partbar_{\Sigma}$.
Now, 
$\coker(\partbar_{\Sigma})=H^{0,1}(\Sigma)\otimes T_y\cC$; varying $y$, one then obtains the bundle 
$H^{0,1}(\Sigma)\otimes e^*T\cC$. 
Then, $N$ acts on such bundle, and what has been previously defined as $E_{(\bg)}$ is nothing else than the $N-$invariant part, i.e.\ $E_{(\bg)}=(H^{0,1}(\Sigma)\otimes e^*T\cC)^N$.
\end{remark}

The three-point function for $\alpha,\beta \in H^*_{\CR}(\cC),\gamma \in H^*_{\CR,c}(\cC)$ is defined by
\begin{equation}
    \label{eq:3pt_funct}
    \langle \alpha,\beta,\gamma \rangle = \int_{\cM_3(\cC)} e_1^*(\alpha) \wedge e_2^*(\beta) \wedge e_3^*(\gamma) \wedge \be(E),
\end{equation}
where $e_i:\cM_3(\cC)\to \Lambda \cC$ are the (proper) evaluation map and $\be(E)\in H^*(\cM_3(\cC))$ is the Euler class of $E$. 
Then the \emph{Chen-Ruan cup product} is defined via the following relation, involving the Chen-Ruan pairing in \Cref{eqn:CR_pairing} and the three-point function in \Cref{eq:3pt_funct}:
$$\langle \alpha \cup \beta, \gamma \rangle_{\CR}=\langle \alpha,\beta, \gamma\rangle.$$

\begin{example}
\label{ex:CR_coeff}
	For an orbifold of the type $\CC^n/G$, with $G\subset U(n)$ having only isolated singularity at $0$, we have that the orbit space $|\cC^1_{(g)}|$ of the inertia orbifold component $\cC^1_{(g)}$ is a singleton for all $(g)\in T^1$, hence
	$$H^*_{\CR}(\CC^n/G)=\oplus_{(g)\in \mathrm{Conj}(G)} \RR[-2\age(g)].$$
	We now want to describe the product structure explicitly. 
	To this purpose, we denote by $[(g)]$ the (positive) generator in cohomology for the conjugacy class $(g)$. For $(g_1),(g_2)\ne (\Id)$, one can define 
	\[I_{g_1,g_2}=\{(h_1,h_2)|h_i\in (g_i), h_1h_2\ne \Id, \age(h_1)+\age(h_2)=\age(h_1h_2)\}.\footnote{In \cite[Example 4.27]{Ruan07}, there is an additional traversality condition. But in the isolated singularity case considered here, the condition is a consequence of $\age(h_1)+\age(h_2)=\age(h_1h_2)$.}\]
	Then by \cite[Examples 4.26 and 4.27]{Ruan07}, we have 
	\begin{equation}
	\label{eqn:decomposition}
	[(g_1)]\cup [(g_2)]=\sum_{(h_1,h_2)\in I_{g_1,g_2}} \frac{\vert C(h_1h_2)\vert}{|C(h_1)\cap C(h_2)|} [(h_1h_2)],
	\end{equation}
	for $(g_1),(g_2)\ne (\Id)$, while $[(\Id)]$ is the ring unit in $H^*_{\CR}(\CC^n/G)$. 
	In particular, denoting $\mathrm{Conj}^*(G)$ the set of nontrivial conjugacy classes in $G$, $\oplus_{(g)\in \mathrm{Conj}^*(G)} \RR[-2\age(g)]$ is a sub-monoid in $H^*_{\CR}(\CC^n/G)$.
	
	\noindent
	In the general case of an exact orbifold filling $W$ of a contact manifold, let $\{x_i\in W\}_{i\in I}$ be the finite set of points with nontrivial isotropy $\{G_i\}_{i\in I}$ (recall \Cref{prop:decomposition}). 
	Then, 
    \[H^*_{\CR}(W)=H^*(W)\oplus \bigoplus_{\substack{i \in I \\ (g)\in \mathrm{Conj}^*(G_i)}} \RR[-2\age(g)],\]
	where the ring structure is the direct sum of the usual ring structure on $H^*(W)$ and, for each $i\in I$, the monoidal structure on $\oplus_{(g)\in\mathrm{Conj}^*(G_i)}\RR[-2\age(g)]$.
    
    In fact, one can naturally extend \Cref{defn:CR_cohomol} using as coefficient ring any ring $R$, which would give, in the case of an exact orbifold filling $W$ of a contact manifold, a   Chen-Ruan cohomology with coefficients in $R$ as follows:
	\begin{equation}\label{eqn:CR_R}
	H^*_{\CR}(W;R)=H^*(W;R)\oplus \bigoplus_{\substack{i \in I \\ (g)\in \mathrm{Conj}^*(G_i)}} R[-2\age(g)],
	\end{equation}
	where $H^*(W;R)$ is the $R$-coefficient cohomology of the underlying topological space $W$. Since $C(h_1)\cap C(h_2)$ is a subgroup of $C(h_1h_2)$, the coefficients in \Cref{eqn:decomposition} are integer hence the ring structure is also defined on $H^*_{\CR}(W;R)$ in the special case of exact orbifold fillings of contact manifolds. \END
\end{example}

\subsection{Morphisms between orbifolds}\label{ss:Morphism}
\label{ss:morphism}
In the classical construction of moduli spaces of holomorphic curves in a symplectic manifold \cite{MS12}, one first realizes the space of \emph{holomorphic maps} as the zero set of a Fredholm section of a certain Banach vector bundle on the Banach manifold of maps from a Riemann surface to the target symplectic manifold. 
In this subsection, we explain the analogue process in the orbifold setting, which will lay the foundation for the construction of Floer cylinders in orbifolds.
The key difference is that the space of maps between two orbifolds is an orbifold; this shall also explain the appearance of the inertial orbifold and $\cM_k(\cC)$ in the previous subsection. 

In \S \ref{ss:orbifold}, we introduced the \emph{set} of morphisms between two orbifolds. 
One can then similarly define the set $C^k(\Sigma,W)$ of \emph{$C^k$ morphisms} from $\Sigma$ to $W$. When $\Sigma$ is compact, the set can be topologized and equipped with a \emph{Banach orbifold} structure,  defined analogously to finite-dimensional orbifolds as in \Cref{defn:orbifold} but via \Cref{def:infinite_dim_groupoid}. Similarly, when $\Sigma$ is compact, one can define the set $W^{k,p}(\Sigma,W)$ of \emph{Sobolev $W^{k,p}$ morphisms}, which is also a Banach orbifold:

\begin{thm}[\cite{Chen06}]\label{thm:morphism_orbifold}
Let $\Sigma$ be a compact orbifold, and $W$ an orbifold.
Then, $C^k(\Sigma,W)$ and $W^{k,p}(\Sigma,W)$ are \emph{canonically}\footnote{This means different metrics on $\Sigma$ and norms on $W$ yield equivalent Banach norms.} Banach orbifolds for $k\ge 1$ and $\dim \Sigma < (k-1)p$ respectively.
Moreover, $C^{\infty}(\Sigma,W)$ is \emph{canonically} a Fr\'echet orbifold.
\end{thm}

\cite{Chen06} worked with compact $\Sigma$, as the author's main motivation was to define orbifold Gromov-Witten invariants \cite{CheRua01}. The benefit of assuming $\Sigma$ compact is that $C^k(\Sigma,W)$ is \emph{canonically} defined and has a \emph{canonical} Banach orbifold structure, namely the set and the Banach orbifold structure are independent of the auxiliary choices used in the construction, e.g.\ metrics on $\Sigma$ and $W$. One the other hand, for Floer theory on orbifolds, one needs to consider maps from punctured/open surfaces $\Sigma$ such as cylinders to a symplectic orbifold $W$, and the corresponding space $W^{k,p}(\Sigma,W)$, or even the version $W^{k,p,\delta}(\Sigma,W)$ with the additional condition of exponential decay at the cylindrical ends. However, such spaces are \emph{not canonically defined} for non-compact surfaces. 
This issue is not related to the orbifold setup and already arises with maps to manifolds: for instance, c.f.\ \cite[\S 8.2]{AudDamBook} for the case of classical Floer theory on manifolds. 
An even simpler example is that \ $C^1(S^1,\RR)$ is canonically a Banach space while $C^1(\RR,\RR)$ is not a Banach space.

\smallskip

While \cite{Chen06} used a more local approach with local uniformizers looking at what are called good maps, we will give here a detailed proof of \Cref{thm:morphism_orbifold} directly using the groupoidal language to explain the new features from the orbifold structures.
More precisely, we will follow \cite{groupoid} to describe $C^k(\Sigma,W),W^{k,p}(\Sigma,W)$ as orbits spaces of groupoids, and then we use the local constructions in \cite{Chen06} in order to show that these groupoids have topologies that make them proper \'etale Lie groupoids modeled over Banach manifolds.

The proof will be done in three steps:
\begin{enumerate}
    \item Construction of the groupoid of maps, which is essentially set theoretic and does not require the domain to be compact.
    \item Construction of the smooth structures, including verifications of topological properties making the map space a proper \'etale Banach orbifold. In this part, the compactness of the domain is needed, see e.g. \Cref{rmk:paracompact}.
    \item Establishing the canonical property. 
\end{enumerate}
The purpose of reproving Chen's theorem is two-folded: (1) the more global perspective in this paper and \cite{groupoid} makes the proof more organized, which allow us to see more clearly the changes needed for non-compact domains for our purpose. (2) This shall give us some details of the structure of the map space and hence some simplification when possible, see \S \ref{sss:simplication}. The Banach structure will only be explicitly proved for $C^k$ maps, as the cases for $W^{k,p}$ maps are similar. 
Similarly, we will not explicitly prove the Fr\'echet orbifold structure for the space of $C^\infty$-maps, as this will not be needed in the following and the proof will be analogous. We will discuss the needed changes for the case of non-compact $\Sigma$ in \S \ref{ss:non-compact}, where we will directly deal with the case of source manifold $\RR\times S^1$ following the same strategy as in \Cref{thm:morphism_orbifold}. 

\begin{remark}
    \label{rmk:banach_structure_not_needed}
    As already anticipated in the introduction, while having a Banach orbifold structure on the space of Floer cylinders would a priori be needed if one tries to define a Hamiltonian Floer theory on symplectic orbifolds, in our specific setup of exact symplectic orbifolds this will strictly speaking not be needed.
    Indeed, the key result we will leverage is \Cref{cor:smooth}, which is based only on the understanding of isotropy groups on the space of orbifold maps from the background groupoidal structures, i.e.\ without the need for its Banach groupoid structure.
    As we will detail in \S \ref{SS:analytical}, \Cref{cor:smooth} will imply that relevant spaces of Floer cylinders in our setup are actually Banach \emph{manifolds}. The only important modification with respect to the manifold case will be that the asymptotic $1$-periodic orbits are orbifold loops (possibly an honest orbifold point in the space of orbifold loops), which leads to consider orbifold fiber products of manifolds of Floer cylinders over the shared asymptotic orbifold loop in order to prove that the differential squares to zero using the discussion in \S \ref{sec:fiber_products}; see \S \ref{ss:gluing} and \S \ref{sec:sympl_cohom} for details.
\end{remark}

\begin{remark}
The assumptions that $k\ge 1$ and $\dim \Sigma < (k-1)p$ are for the following two purposes.  

\noindent 
(1) When maps are at least $C^0$, we can pull back orbibundles and model the mapping spaces by sections of the pullback bundle. 

\noindent
(2) The conditions on the indices are also needed for the properness of the Banach groupoid in \Cref{prop:proper}, where we use the compact embedding of $C^k$ maps into $C^{k-1}$ maps (or the compact embedding of $W^{k,p,\delta}$ maps into $W^{k-1,p,\delta'}$ maps for $\delta'<\delta$ on a non-compact surface $\Sigma$).
\end{remark}

If we only consider smooth functors as the morphisms between Lie groupoids, then the ``category of Lie groupoids” embeds into the category of categories, which is actually a $2$-category with $2$-morphisms being natural transformations. 
In the groupoid case, any natural transformation is automatically a natural equivalence; this endows the set of functors between two Lie groupoids with the structure of a groupoid. 
Roughly speaking, the orbifold structure on the space of maps between orbifolds is precisely this groupoid structure after localization w.r.t. Morita equivalences of Lie groupoids (or orbifold structures). The localization, e.g.\ representing maps using more and more refined open covers, is necessary to obtain all the orbifold maps, see \Cref{ex:refine}.

\begin{remark}
A few remarks regarding morphisms between orbifolds are in order.
\begin{enumerate}
    \item The original orbifold maps were defined by Satake \cite{Satake57} as continuous maps between orbit spaces admitting local equivariant smooth lifts to the local uniformizers.
    \item The above definition is not very ``well-behaved'', in particular, one can not pull back a vector bundle in general. In \cite{CheRua01,CheRua04}, Chen-Ruan introduced the notion of good maps to fix this problem.
    \item The notion of generalized maps as morphisms between orbifolds was introduced in \cite{Moerdijk97}.
    \item For effective orbifolds,  good maps are equivalent to generalized maps \cite{lupercio2004gerbes}. For unreduced orbifolds, the notion of generalized maps is more general.
    \item A discussion on the bicategory of orbifolds, as well as the bibundle description of orbifold maps, can be found in \cite{lerman2010orbifolds}. 
\end{enumerate}
\end{remark}

\subsubsection{Step 1: construction of the groupoid (not requiring $|\cC|$ compact).} \label{sss:groupoid}
Let $\cC,\cD$ be two proper \'etale Lie groupoids.
We use $C^k(\cC,\cD)$ to denote the set of equivalence classes of $C^k$ generalized maps from $\cC$ to $\cD$. 
Recall that, by definition, two 
generalized maps  
are equivalent  if there is a diagram as follows
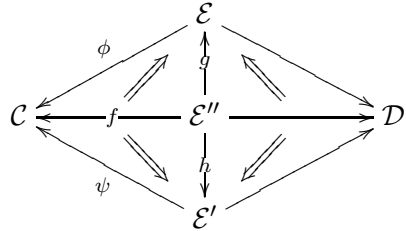
\[
\xymatrix{
& & \cE\ar[lld]_{\phi}\ar[rrd] & & \\
\cC &  
  \ar@{}[rd]^(.1){}="a"^(.7){}="b" \ar@{=>} "a";"b"  
  \ar@{}[ru]^(.1){}="a"^(.7){}="b" \ar@{=>} "a";"b" 
  & 
  \cE''\ar[u]^{g} \ar[d]_{h} \ar[ll]|f \ar[rr] 
  & 
  \ar@{}[ld]^(.1){}="a"^(.7){}="b" \ar@{=>} "a";"b"  
  \ar@{}[lu]^(.1){}="a"^(.7){}="b" \ar@{=>} "a";"b" 
  & 
  \cD \\
& &\cE'\ar[rru]\ar[llu]^{\psi}  & &}
\]
where all labeled functors are equivalences. 
Let now $\tau$ be the natural transformation $\phi\circ g \Rightarrow \psi\circ h$, and define $k:\cE''\to \cE\times_{\cC}\cE'$ by $x\mapsto (g(x),\tau(x),h(x))$. 
Notice that $k$ is an equivalence.
Then, we obtain the following diagram, where the two triangles not marked with $\Rightarrow$ are commutative (in the strict sense, i.e.\ not only up to natural transformation but actually commutative).

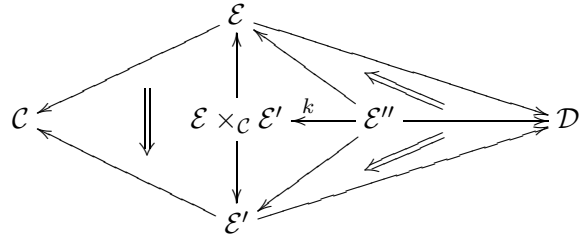
\[
\xymatrix{
&  \ar@<3ex>@{}[dd]^(0.3){}="a"^(0.7){}="b" \ar@{=>} "a";"b"  &  \cE\ar[lld]_{\phi}\ar[rrrd]  & & &\\
\cC &  & \cE\times_{\cC} \cE'\ar[u]\ar[d]  & \cE''\ar[lu]\ar[ld]\ar[l]_{k} \ar[rr] & \ar@{}[llu]^(0.1){}="a"^(0.5){}="b" \ar@{=>} "a";"b" \ar@{}[lld]^(0.1){}="a"^(0.5){}="b" \ar@{=>} "a";"b"    & \cD \\
&  &\cE'\ar[llu]^{\psi}\ar[rrru]  & & &}
\]
We claim that there is a natural transformation $\alpha$ such that we have the following diagram.
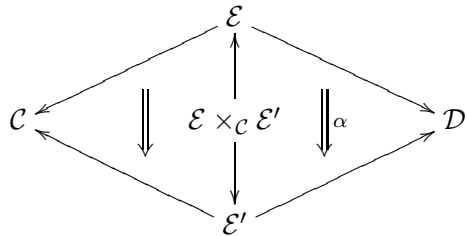
\[
\xymatrix{
&  \ar@<3ex>@{}[dd]^(0.3){}="a"^(0.7){}="b" \ar@{=>} "a";"b"  &  \cE\ar[lld]\ar[rrd]  &  \ar@<-3ex>@{}[dd]^(0.3){}="a"^(0.7){}="b" \ar@{=>}^{\alpha} "a";"b"  &\\
\cC &  & \cE\times_{\cC} \cE'\ar[u]\ar[d]  &   & \cD \\
&  &\cE'\ar[llu]\ar[rru]  &  &}
\]
Let $\beta$ be the natural transformation from $\cE''\to\cE \to \cD$ to $\cE''\to \cE'\to \cD$. 
We use $\eta,\xi$ to denote the functors $\cE\times_{\cC}\cE'\to \cE\to \cD$ and $\cE\times_{\cC}\cE'\to \cE'\to \cD$ respectively. 
We define the natural transformation $\alpha$ as follows: for $x\in \cE\times_{\cC}\cE'$, $\alpha(x)$ is defined as $\xi(a^{-1})\circ \beta(y)\circ \eta(a)$, where $y\in \cE''$ and $a$ is a morphism $x\to k(y)$, whose existence follows from the fact that $k$ is an equivalence. 
It is direct to check that such $\alpha$ is well-defined and is a natural transformation, using the facts that $\beta$ is a natural transformation and $k$ is an equivalence.

We call such an $\alpha$ a \emph{morphism} between those generalized maps. 
We also use $C^k\Mor^0(\cC,\cD)$ to denote the set of $C^k$ generalized maps from $\cC$ to $\cD$, and $C^k\Mor^1(\cC,\cD)$ to denote the set of $C^k$ natural transformations $\alpha$ as above (along with the two generalized maps). 
\begin{prop}[\cite{groupoid}]
$(C^k\Mor^0(\cC,\cD),C^k\Mor^1(\cC,\cD))$ is a groupoid and we have $$C^k\Mor^0(\cC,\cD)/C^k\Mor^1(\cC,\cD)=C^k(\cC,\cD).$$
\end{prop}
Although the statements in \cite{groupoid} are for groupoids, the proof applies to proper \'etale Lie groupoids. 

Sometimes, it is helpful to replace the weak fiber product $\cE\times_{\cC}\cE'$, where the pullback square is only commutative up to natural transformations, to a smaller model whose pullback square is actually commutative.  
Following \cite{groupoid}, a \emph{full equivalence} $\psi:\cC\to \cD$ is an equivalence where $\psi$ is surjective on the object level. 
Assume $\phi,\psi:\cE,\cE' \to \cC$ are two full equivalences; the \emph{strict fiber product} $\cE\otimes_\cC\cE'$ is defined to be 
$$\Obj(\cE\otimes_\cC\cE')=E\tensor[_\phi]{\times}{_\psi} E', \quad  \Mor(\cE\otimes_\cC\cE')=\bE\tensor[_\phi]{\times}{_\psi} \bE'$$
with obvious structural maps. 
Moreover, $(x,y)\mapsto (x,\mathrm{id}_{\phi(x)},y)$ defines a natural equivalence functor from $\cE\otimes_{\cC}\cE'$ to $\cE\times_{\cC} \cE'$.
Following again \cite{groupoid}, we denote by $C^k\mathrm{FMor}^0(\cC,\cD)$ the set of $C^k$ generalized maps $\cC\stackrel{\psi}{\leftarrow}\cE \stackrel{\phi}{\rightarrow}\cD$ such that the equivalence $\psi$ is a full equivalence, and $C^k\mathrm{FMor}^1(\cC,\cD)$ the set of $C^k$ natural transformations as in the case of $C^k\mathrm{Mor}$, but with the weak fiber product replaced by the strict fiber product.

\begin{prop}[\cite{groupoid}]\label{prop:full_equivalence}
$(C^k\mathrm{FMor}^0(\cC,\cD),C^k\mathrm{FMor}^1(\cC,\cD))$ is a groupoid. Moreover, it is equivalent to $(C^k\Mor^0(\cC,\cD),C^k\Mor^1(\cC,\cD))$.
\end{prop}

For what follows, we will need to work with special representatives of orbifold structures for the orbifolds. We then give the following definition:
\begin{defn}
\label{defn:covering_groupoid}
Let $\cC$ be a proper \'etale Lie groupoid, and $\cup_{i\in I} U_i=C$ an open cover. 
We call a \emph{covering groupoid}, and denote it by $\Gamma\{U_i\}_I$, the proper \'etale Lie groupoid with object set $\sqcup_i U_i$ and morphism set $\sqcup_{i,j}\Gamma(U_i,U_j)$, where $\Gamma(U_i,U_j)=\{\phi \in \bC| s(\phi)\in U_i,t(\phi)\in U_j\}$;
the structural maps on $\Gamma\{U_i\}_I$ are the obvious ones. 
\end{defn}
It is important to notice, although straightforward to check, that \emph{the natural map $\Gamma\{U_i\}\to \cC$ defines a functor, which is also a full equivalence}. 
In particular, every orbifold admits a representative orbifold structure given by a covering groupoid; this will be used in several places in the rest of \S \ref{ss:morphism} and in \S \ref{sec:homotopy_groups}.
Lastly, we point out that, for any choice of covers $\cup_{i\in I} U_i=C$ for $\mathcal{C}$ and $\cup_{j\in J} V_j = D$ for $\mathcal{D}$, we have  $\Gamma\{U_i\}_I\otimes_{\cC} \Gamma\{V_j\}_J= \Gamma\{U_i\cap V_j\}_{I\times J}$. 

\begin{defn}
We define $C^k\mathrm{CMor}^0(\cC,\cD)$ to be the set of generalized maps in the form of $\cC \stackrel{\phi}{\leftarrow}\Gamma\{U_i\} \stackrel{\psi}{\to} \cD$ such that $\psi|_{U_i}$ extends to a $C^k$ map on $\overline{U_i}$, for some covering $\cup_{i\in I} U_i = C$ for $\mathcal{C}$. 
We also denote the set of natural transformations by $C^k\mathrm{CMor}^1(\cC,\cD)$, and the groupoid $(C^k\mathrm{CMor}^0(\cC,\cD),C^k\mathrm{CMor}^1(\cC,\cD))$ by $\cC^k(\cC,\cD)$. 
\end{defn}

\begin{prop}\label{prop:cover_equivalence}
$\cC^k(\cC,\cD)$ is equivalent to  $(C^k\Mor^0(\cC,\cD),C^k\Mor^1(\cC,\cD))$. In particular, $|\cC^k(\cC,\cD)|=C^k(\cC,\cD)$.
\end{prop}
\begin{proof}
    By definition, the inclusion of $(C^k\mathrm{CMor}^0(\cC,\cD),C^k\mathrm{CMor}^1(\cC,\cD))$ in $(C^k\mathrm{FMor}^0(\cC,\cD), C^k\mathrm{FMor}^1(\cC,\cD))$ is fully faithful.
    Hence, in view of \Cref{prop:full_equivalence}, the inclusion into $(C^k\mathrm{Mor}^0(\cC,\cD),C^k\mathrm{Mor}^1(\cC,\cD))$ is also fully faithful, so it remains to show that the latter is essentially surjective as well. 
    Let $\cC \stackrel{\phi}{\leftarrow}\cE \to \cD$ be a generalized map with $\phi$ a full equivalence. 
    Then, by definition of full equivalence, we can find an open cover $\{U_i\}_i$ of $\Obj(\cC)$, such that $\phi:U'_i\to U_i$ is a diffeomorphism for some $U'_i\subset \Obj(\cE)$.
    As a consequence of fully faithfulness of $\phi$, we get a strict (i.e.\ not only up to natural transformations) commutative diagram.
    \[
    \xymatrix{
    \cC & \cE \ar[r]\ar[l] &\cD\\
    & \Gamma\{U_i\}\ar[ur]\ar[ul]\ar[u] &
    }
    \]
    Therefore we have
    \[
    \xymatrix{
    &  \ar@<3ex>@{}[dd]^(0.3){}="a"^(0.7){}="b" \ar@{=>} "a";"b"  &  \cE\ar[lld]\ar[rrrd]  & & &\\
    \cC &  & \cE\times_{\cC} \Gamma\{U_i\}\ar[u]\ar[d]  & \Gamma\{U_i\}\ar[lu]\ar[ld]_{=}\ar[l] \ar[rr] &    & \cD \\
    &  &\Gamma\{U_i\}\ar[llu]\ar[rrru]  & & &}
    \]
    That is, $\cC \leftarrow\cE \to \cD$ is equivalent to $\cC \leftarrow \Gamma\{U_i\} \to \cD$ in $(C^k\Mor^0(\cC,\cD),C^k\Mor^1(\cC,\cD))$, hence the claim follows.
\end{proof}

\subsubsection{Step 2: endowment of smooth structures.} In the following, we will endow $\cC^k(\cC,\cD)$ with the structure of a proper \'etale Banach groupoid and the orbit space $C^k(\cC,\cD)$ with a Hausdorff paracompact topology. 

\begin{defn}
A covering groupoid $\Gamma\{U_i\}$ is \emph{good} if each $U_i$ is a connected local uniformizer and $s:\Gamma(U_i,U_j)\to U_i$ and $t:\Gamma(U_i,U_j)\to U_j$ restricted to each connected component are homeomorphisms onto their image. 
\end{defn}
It is straightforward to check that any \emph{refinement of a good covering groupoid}, by which we mean another covering groupoid obtained from a refinement of the original covering, is also good and good cover groupoids always exist \cite[Proposition 2.1.3.]{Chen06}. 
With this goodness property, for every element $\xi \in \pi_0(\Gamma(U_i,U_j))$, there is an induced diffeomorphism $\widehat{\xi}:=t\circ s^{-1}:s(W_{\xi})\to t(W_\xi)$, where $W_\xi$ is the connected component of $\Gamma(U_i,U_j)$ representing $\xi$.  
In the particular case of $\Gamma(U_i,U_i)$, we have $\pi_0(\Gamma(U_i,U_i))=G_i$, where $G_i$ is the group for the local uniformizer $U_i$ and $\widehat{\xi}$ is the action by $\xi \in G_i$. 
Then by \cite[Proposition 2.1.1.]{Chen06}, the set $\sqcup_{i,j}\pi_0(\Gamma(U_i,U_j))$ has ``groupoid-like" properties, i.e.\ partially defined compositions and inversions, which are the information that encode the gluing of $\{G_i\ltimes U_i\}$ to $\Gamma\{U_i\}$.  

The benefit of working with such covering groupoids is the following lemma. 
Let $\xi\in \pi_0(\Gamma(U_i,U_j)),\eta \in \pi_0(\Gamma(U_j,U_k))$ and $x\in \widehat{\xi}^{-1}(\dom(\widehat{\eta}))$; here, $\dom(\widehat{\eta})$ denotes the domain of $\widehat{\eta}$, i.e.\ $s(W_\eta)$. Notice that the map $\eta\circ\xi$ is in fact locally constant, we then denote by $\eta\circ \xi(x)$ the connected component in $\pi_0(\Gamma(U_i,U_k))$ of the morphism $x\stackrel{\widehat{\eta}\circ \widehat{\xi}}{\longrightarrow} \widehat{\eta}\circ \widehat{\xi}(x)$. Notice that the map $\eta\circ\xi$ is in fact locally constant.

\begin{lemma}[{\cite[Lemma 2.2.1.]{Chen06}}]\label{lemma:local}
Let $\Gamma\{U_i\}_I,\Gamma\{U'_j\}_J$ be good covering groupoids for $\cC,\cD$ respectively. 
Then, any smooth functor $\phi\colon\Gamma\{U_i\}_I\to \Gamma\{U'_j\}_J$ is defined by the data of
\begin{enumerate}
    \item $\sigma:I\to J$ such that $|\phi|(|U_i|)\subset |U'_{\sigma(i)}|$, where $|\phi|$ is the continuous map $|\cC|\to |\cD|$ induced by $\phi$,
    \item smooth maps $\phi_i:U_i\to U'_{\sigma(i)}$,
    \item $\phi_{ji}:\pi_0(\Gamma(U_i,U_j))\to \pi_0(\Gamma(U'_{\sigma(i)},U'_{\sigma(j)}))$,
\end{enumerate}
such that the following holds:
\begin{enumerate}[label=(\alph*)]
    \item $\widehat{\phi_{ji}(\xi)}\circ \phi_i(x)=\phi_{\sigma(j)}\circ \widehat{\xi}(x)$ for $\xi\in \pi_0(\Gamma(U_i,U_j))$ and $x\in \dom(\widehat{\xi})$.
    \item For $\xi\in \pi_0(\Gamma(U_i,U_j)),\eta \in \pi_0(\Gamma(U_j,U_k))$ and $x\in \widehat{\xi}^{-1}(\dom(\widehat{\eta}))$, we have $\phi_{ki}(\eta\circ \xi(x))=\phi_{kj}(\eta)\circ \phi_{ji}(\xi)(\phi_i(x))$.
\end{enumerate}
\end{lemma}
By definition, two functors $\phi,\psi$ with the same index function $\sigma$ are naturally equivalent (i.e.\ there is a smooth natural transformation from one to another) iff there are $\{g_i\in G'_{\sigma(i)}\}_{I}$ such that 
\begin{equation}\label{eqn:natural_transformation}
    \psi_i=g_i\circ \phi_i, \qquad \psi_{ji}(\xi)=(g_{j}\circ \phi_{ji}\circ g_i^{-1})(\xi)
\end{equation}
where $g_{j}\circ \phi_{ji}\circ g_i^{-1}$ denotes the composition of $\phi_{ji}$ and the isomorphism on $\pi_0(\Gamma(U'_{\sigma(i)},U'_{\sigma(j)}))$ defined by $\xi \mapsto g_j \circ \xi \circ g_i^{-1}$.

Let now $\Gamma\{U_i\}$ be a covering groupoid for $\cC$ and $\phi$ a functor from $\Gamma\{U_i\}\to \cD$. 
We call a \emph{good refinement} of $\{U_i\}_I$ any covering $\{V_j\}_J$ with $\Gamma\{V_j\}$ a good covering groupoid, and such that the index set $J$ can be decomposed as $J=\sqcup_{i\in I} J_i$ in such a way that $\cup_{j\in J_i} V_j=U_i$. 
Then, we claim that, given $\Gamma\{U_i\}$ and $\phi$, one can find a good refinement $\Gamma\{V_i\}$ of $\Gamma\{U_i\}$ and a good covering $\Gamma\{V'_i\}$ of $\cD$, such that we have the commutative diagram: 
\begin{equation}\label{eqn:refine}
\xymatrix{\Gamma \{U_i\}\ar[r]^{\phi} & \cD \\ \Gamma\{V_i\}\ar[u]\ar[r] & \Gamma\{V'_i\}\ar[u]}
\end{equation}
Indeed, this can be arranged by first fixing a covering $\cup_{i\in I'}V'_i$ for $\mathcal{D}$ by sufficiently small local uniformizers $V'_i$, and then a sufficiently small refinement of the covering $\cup_{i\in I} U_i$ made by uniformizers that are sent via $\phi$ completely inside one of the $V'_i$'s.
Moreover, up to shrinking each $V_i$'s, we can arrange that $\phi_i(V_i)$ has a compact closure in $V'_{\sigma(i)}$. We will first topologize the space of functors from $\Gamma\{V_i\}$ to $\Gamma\{V'_i\}$, then follow the diagram above to put a topology on the space of functors from $\Gamma\{U_i\}$ to $\cD$.

\begin{prop}\label{prop:Banach_refinement}
For $k\ge 0$, the set of $C^k$ functors from $\Gamma\{V_i\}$ to $\Gamma\{V'_i\}$ is a Hausdorff space modeled smoothly on Banach spaces.
\end{prop}
\begin{proof}
The set of $C^k$ functors from $\Gamma\{V_i\}$ to $\Gamma\{V'_i\}$ is a disjoint union of functors with different index function $\sigma$. Then it suffices to prove the claim for functors with a fixed index function $\sigma$. For simplicity, we will write $\sigma=\Id$ with a little abuse of notation.  

We first fix two Riemannian metrics on $\Gamma\{V_i\}$ and $\Gamma\{V'_i\}$, or equivalently on $\cC,\cD$. 
Given a functor $\phi:\Gamma\{V_i\}\to \Gamma\{V'_i\}$, we consider $E_{\phi}$ the Banach space of $C^k$ sections of $\phi^*T\Gamma\{V'_i\}$.
More precisely, a vector $v$ in $E_\phi$ consists of sections $v_i:V_i\to \phi_i^*TV'_i$ that are $G_i-$equivariant and satisfying the following: for $\xi\in \pi_0(\Gamma(V_i,V_j))$ and $x\in\mathrm{Dom}(\widehat{\xi})$, one has $v_j(\widehat{\xi}(x))=\widehat{\phi_{ji}(\xi)}_*(v_i(x))$.
The Banach norm is defined to be $\max |v_i|_{C^k}$; by the fact that $|\cC|$ is compact\footnote{With those fixed metrics, $|D^r \phi|$ is a continuous function well-defined on $|\cC|$.}, this number is indeed finite, and the resulting norm is independent (up to equivalence) on the choices of metrics. 
Then for $\epsilon \ll 1$ and any $v\in E_{\phi}$ with $|v|<\epsilon$, we have that $\phi_v:=(\{\exp_{\phi_i}v_i\}_i,\{\phi_{ji}\}_{i,j})$ is a functor from $\Gamma\{V_i\}$ to $\Gamma\{V'_i\}$  by the properties of $v_i$ above and Lemma \ref{lemma:local}, where $\exp$ is the exponential map on $\Gamma\{V'_i\}$ resulting from the given metric. 
This endows the set of $C^k-$functors from $\Gamma\{V_i\}$ to $\Gamma\{V'_i\}$ with local charts, hence a topology for which these charts are a base. 
The smoothness of the transition maps can be proven similarly to the smoothness of transition maps in the construction of the Banach manifold of $C^k$ maps between two manifolds, as the only data that needs to be transferred is that of the $\exp_{\phi_i}v_i$'s. 

Lastly, we prove that the topology is Hausdorff. 
We first show that the set of functors with a fixed $\phi_{ji}$ is Hausdorff. 
Assume $\phi,\psi$ are two different functors with same $\phi_{ji}=\psi_{ji}$;
then, there must be some index $i$ and $p\in V_i$ such that $\phi_i(p)\ne \psi_i(p)$. 
We pick an open ball $U_p$ of $p$ such that $\overline{U_p}\subset V_i$. It is straightforward to check that the map to $C^k(\overline{U_p},V'_i)$ given by $\psi\mapsto \psi_i|_{\overline{U_p}}$ is continuous, hence $\phi,\psi$ can be separated by open sets by the Hausdorff-ness of $C^k(\overline{U_p},V'_i)$. 
On the other hand, if $\phi_{ji}\neq\psi_{ji}$ for some $(i,j)$, then $\phi$ and $\psi$ clearly lie in different connected components.
This concludes the proof.
\end{proof}

\begin{remark}\label{rmk:paracompact}
Note that $|\cC|$ being compact implies that a finite subset of $\{V_i\}$ covers $|\cC|$. 
Then a functor $\phi$ from $\Gamma\{V_i\}$ to $\Gamma\{V'_i\}$ with fixed $\phi_{ji}$ is determined by the restriction on that finite cover, and can be recovered from the restriction by the functorial property using $\phi_{ji}$. 
Now we assume in addition that $V_i$ has as closure a subdomain with boundary in $U_k$, such that moreover $\phi_i(\overline{V_i})\subset V'_i$. 
Then the set of functors from $\Gamma\{V_i\}$ to $\Gamma\{V'_i\}$ with the $C^k$ topology is a closed subspace of $\prod_{i\in I}C^k(\overline{V}_i,V'_i)$ for a finite set $I$ subject to relations in Lemma \ref{lemma:local}. 
Then the set of functors with a fixed $\phi_{ji}$ is paracompact, and the total space of functors (i.e.\ without fixing $\phi_{ji}$ and index function $\sigma$) is also paracompact.
\end{remark}

\begin{prop}\label{prop:smooth_functor}
For $k\ge 0$, the set of $C^k$ functors from $\Gamma\{U_i\}$ to $\cD$ is a Hausdorff space modeled smoothly on Banach spaces.
\end{prop}
\begin{proof}
Let $\phi$ be a $C^k$ functor from $\Gamma\{U_i\}$ to $\cD$. 
Then we have a refinement in the form of \eqref{eqn:refine}, where we denote by $\psi$ the functor from $\Gamma\{V_i\}$ to $\Gamma\{V'_i\}$. 
Such a $\psi$ then satisfies the following properties, where we assume again (i.e.\ as in the previous proof) for notational simplicity that the index function $\sigma$ is $\Id$.
\begin{enumerate}
    \item\label{refine:1} Assume $V_i, V_j$ are both contained in $U_k$. Then, for all $x\in V_i\cap V_j$, the image in $\cD$ of $\psi_i(x)\in V'_i$ coincides with that of $\psi_j(x)\in V'_j$. 
    \item\label{refine:2} Given $V_{i'}\subset U_i,V_{j'}\subset U_j$, let $f\in \Gamma(V_{i'},V_{j'})$, then $\psi(f)\in \Gamma(V'_{i'},V'_{j'})$ viewed as in $\cD$ only depends on $f$ viewed as in $\Gamma(U_i,U_j)$.
\end{enumerate}
It is moreover clear that any $\psi$ satisfying the above properties is a refinement in the form of \eqref{eqn:refine}. 

We now claim that for $\epsilon\ll 1$, every element in an open ball (as described in the proof of \Cref{prop:Banach_refinement}) of radius $\epsilon$ centered at $\psi$ that moreover has the same $\psi_{ji}$ is a refinement of functors from $\Gamma\{U_i\}$ to $\cD$ as in \eqref{eqn:refine}. 
Indeed, first note that \eqref{refine:1} is equivalent to the fact that $\psi_{ji}$ sends, for any $V_i,V_j$ with $V_i\cap V_j\ne \emptyset \subset U_k$, the component representing the identity morphism to the component representing the identity morphism. 
Hence \eqref{refine:1} holds on this open $\epsilon-$ball. 
For \eqref{refine:2}, first note that $\psi(f)\in \Gamma(V'_{i'},V'_{j'})$ is completely determined by $\psi_{j'i'}$ and $\psi_{i'}(s(f))$, and that the latter only depends on $f$ viewed as in $\Gamma(U_i,U_j)$ by \eqref{refine:1}.
Assume now that $f\in \Gamma(U_i,U_j)$ can be viewed as both $f'\in \Gamma(V_{i'},V_{j'})$ and $f''\in \Gamma(V_{i''},V_{j''})$, then $\psi(f'')^{-1}\circ \psi(f')$ viewed in $\cD$ is in the isotropy group of $\psi_{i'}(s(f))$. 
Now, when viewed as taking value in $G'_{i'}$ in the local uniformizer $G'_{i'}\ltimes V'_{i'}$, this assignment is locally constant and is completely determined by $\psi_{j'i'}$. 
Note moreover that \eqref{refine:2} is equivalent to the fact that $\psi(f'')^{-1}\circ \psi(f')$ is always the identity. 
Since \eqref{refine:2} holds for $\psi$, \eqref{refine:2} must then hold for the open $\epsilon-$ball as $\psi_{j'i'}$ doesn't change in this ball, as desired.

Therefore, we can use \Cref{prop:Banach_refinement} to endow that set of $C^k$ functors from $\Gamma\{U_i\}$ to $\cD$ with local charts and hence a topology.
To prove that the transition maps are smooth, we argue as follows.

Notice that given a further refinement $\psi'$ of $\psi$, the transition map is a linear isomorphism between $E_{\psi'}$ and $E_{\psi}$, where the latter are the local charts as in the proof of \Cref{prop:Banach_refinement}. 
Then the smoothness of transition maps follows from the fact that any two refinements admit a common refinement and that the transition maps as in the proof of \Cref{prop:Banach_refinement} are smooth. 

Concerning Hausdorff-ness, let $\phi,\phi'$ be two different functors from $\Gamma\{U_i\} \to \cD$. 
Then either $\phi,\phi'$ are different at the objects level on some $U_i$ or $\phi,\phi'$ are different on the morphisms level on some $\Gamma(U_i,U_j)$. 
In both cases, the same argument as in \Cref{prop:Banach_refinement} allows us to conclude. 
\end{proof}

The above then tells that $C^k\mathrm{CMor}^0(\cC,\cD)$ is a Hausdorff space modeled on Banach spaces.
The same is in fact true for the morphism space:

\begin{prop}\label{prop:natural_transformation}
$C^k\mathrm{CMor}^1(\cC,\cD)$ is a Hausdorff space modeled on Banach spaces.
Moreover, the source and target maps $s,t:C^k\mathrm{CMor}^1(\cC,\cD) \to C^k\mathrm{CMor}^0(\cC,\cD)$ are local diffeomorphisms, and the structural maps $m,u,i$ are smooth.
\end{prop}
\begin{proof}
Let $\alpha\in C^k\mathrm{CMor}^1(\cC,\cD)$ be a natural transformation, as follows.
\[
\xymatrix{
& &  \Gamma\{U_i\}\ar[lld]\ar[rrd]  &  \ar@<-3ex>@{}[dd]^(0.3){}="a"^(0.7){}="b" \ar@{=>}^{\alpha} "a";"b"  &\\
\cC &  & \Gamma\{U_i\cap U'_j\}\ar[u]\ar[d]  &   & \cD \\
&  &\Gamma\{U'_i\}\ar[llu]\ar[rru]  &  &}
\]
Let $\Gamma\{V_i\}$ be a good refinement of $\Gamma\{U_i\cap U'_j\}$ and $\Gamma\{V'_i\}$ a good refinement of $\cD$, and denote by $\phi:\Gamma\{V_i\}\to \Gamma\{V'_i\}$ and $\psi:\Gamma\{V_i\}\to \Gamma\{V'_i\}$ the refinements (as in \eqref{eqn:refine}) of, respectively, the composition $\Gamma\{U_i\cap U'_j\}\to \Gamma\{U_i\}\to \cD$ and the composition $\Gamma\{U_i\cap U'_j\}\to \Gamma\{U'_i\}\to \cD$. 
For notational simplicity, we also assume that the index function of $\phi$ is the identity and that the one of $\psi$ is a certain function $\sigma$. 
Then we have the following diagram.
\[
\xymatrix{
& \ar@<-3ex>@{}[d]^(0.35){}="a"^(0.8){}="b" \ar@{=} "a";"b" &  \ar@<-3ex>@{}[d]^(0.35){}="a"^(0.8){}="b" \ar@{=} "a";"b" & \\
\Gamma\{V_i\} \ar[r]\ar@/^25pt/[rr]^\phi \ar@/_25pt/[rr]_\psi & \Gamma\{U_i\cap U_j'\} \ar@<-3ex>@{}[d]^(0.2){}="a"^(0.65){}="b" \ar@{=>}^{\alpha'} "a";"b"\ar@/^25pt/[rr]|!{[u];[r]}\hole \ar@/_25pt/ [rr]|!{[d];[r]}\hole  & \Gamma\{V'_i\} \ar@<-3ex>@{}[d]^(0.2){}="a"^(0.65){}="b" \ar@{=>}^{\alpha} "a";"b" \ar[r] & \cD \\
& & &
}
\]
Here the natural transformation $\alpha'\colon \phi \Rightarrow \psi$ is determined by $\alpha$, and any $\alpha'$ satisfying the following condition $(\star)$ can be glued back to an $\alpha$:
\begin{enumerate}
    \item[$(\star)$] If $V_i,V_j$ are both contained in $U_k\cap U'_l$, then, for every $x\in V_i\cap V_j$, $\alpha'|_{V_i}(x)\in \Gamma(V'_i,V'_{\sigma(i)})$ and $\alpha'|_{V_j}(x)\in \Gamma(V'_j,V'_{\sigma(j)})$ are the same when viewed in $\Mor(\cD)$.
\end{enumerate}
Note that $\alpha'$ is determined by the following data:
\begin{enumerate}
    \item[$(\bullet)$] $\xi_i\in \pi_0(\Gamma(V'_i,V'_{\sigma(i)}))$, such that $\psi_{ji}(\xi)=(\xi_j\circ \phi_{ji}\circ \xi^{-1}_i)(\xi)$ for each $\xi\in\pi_0(\Gamma(V_i,V_j))$, where $\xi_j\circ \phi_{ji}\circ \xi^{-1}_i:\pi_0(\Gamma(V_i,V_j))\to \pi_0(\Gamma(V'_{\sigma(i)},V'_{\sigma(j)}))$ is defined as the map induced on $\pi_0$ by the following map:
    \[
    \Gamma(V_i,V_j) \ni f \quad \mapsto  \quad    \left(\psi_i(s(f)) \stackrel{\widehat{\xi_j}\circ \widehat{\phi_{ji}(\xi)}\circ \widehat{\xi_i}^{-1}}{ \xrightarrow{\hspace*{2cm}}} \psi_j(t(f)) \right) \in \Gamma(V'_{\sigma(i)},V'_{\sigma(j)}),
    \]
    where $\xi$ is the connected component containing $f$.
\end{enumerate}
As a consequence, the local chart around $\phi$ in the proof of \Cref{prop:smooth_functor}, along with the fixed $\{\xi_i\}$ above, gives rise to a family of natural transformations near $\alpha'$. 
Moreover, with fixed $\{\xi_i\}$, condition ($\star$) (which can also be phrased using $\{\xi_i\}$, up to a locally constant ambiguity, similarly to what is done in the proof of \Cref{prop:smooth_functor}) holds for those nearby natural transformations. 
Hence, we constructed charts for the set of natural transformations. 

The smoothness of transition maps can also be inferred similarly to \Cref{prop:smooth_functor} by using refinements,
and the induced topology is again Hausdorff by the same argument as in the proof of \Cref{prop:smooth_functor}. 

The source map $s$ is a local diffeomorphism by construction. 
To see that the target map $t$ is a local diffeomorphism, let $\widetilde{\phi}$ be a nearby functor in the neighborhood of $\phi$ and $\widetilde{\alpha}'$ be the associated nearby natural transformation in the above-described chart. 
Then, $\widetilde{\psi}=t(\widetilde{\alpha}')$ is a functor from $\Gamma\{V_i\}$ to $\Gamma\{V'_i\}$ determined by the index function $\sigma$ and the following data:
\begin{enumerate}
    \item $\widetilde{\psi}_i=\widehat{\xi_i}\circ \widetilde{\phi}_i$;
    \item $\widetilde{\psi}_{ji}=\xi_j\circ \widetilde{\phi}_{ji}\circ \xi_i^{-1}=\xi_j\circ \phi_{ji}\circ \xi_i^{-1}=\psi_{ji}$ by property ($\bullet$). 
\end{enumerate}
Hence it is straightforward to check that the target map is smooth and a local diffeomorphism: indeed, using the charts in \Cref{prop:Banach_refinement}, the target map is given in the local chart $E_\phi$ near $\phi$ by a linear isomorphism from $E_{\phi}$ to $E_{\psi}$ given by the pushforward via the $\{\widehat{\xi_i}\}_i$; the linearity follows from the fact that $\widehat{\xi_i}$ preserves the metric and hence $\widehat{\xi_i}(\exp_{\phi_i}v_i)=\exp_{\psi_i}(\widehat{\xi_i}_*v_i)$.

The smoothness of $m,u,i$ can be proven similarly, thus concluding the proof.
\end{proof}

The local uniformizers in $\cC^k(\cC,\cD)$ can be easily seen from the construction. Let $B_{\psi}$ be an open ball near $\psi:\Gamma\{V_i\}\to \Gamma\{V'_i\}$, with index function $\sigma$, that refines $\phi:\Gamma\{U_i\}\to \cD$ according to \eqref{eqn:refine}. 
Suppose we use $B_\psi$ to give a chart near $\phi$ in the set of functors from $\Gamma\{U_i\}\to \cD$. 
Let $G_{\psi}$ denote the set of self natural transformations, i.e.\ the set of collections of elements $\{g_i'\in G_{\sigma(i)}'\}_i$ in the isotropy groups $G_{\sigma(i)}'$ of the local uniformizers $V'_{\sigma(i)}$ satisfying \eqref{eqn:natural_transformation} (with $\phi=\psi$). 
Notice that $G_{\psi}$ is naturally a subgroup of $\prod_{I} G'_{\sigma(i)}$. 
Following the proof of \Cref{prop:natural_transformation}, we can describe an action of $G_{\psi}$ on $B_{\psi}$ by the following:
$$(\{g_i\}, \widetilde{\psi}) \mapsto (\widehat{g_i}\circ \widetilde{\psi}_i, \psi_{ji}), \quad \widetilde{\psi}\in B_\psi, \{g_i\}\in G_\psi.$$
This gives rise to a local uniformizer of $\cC^k(\cC,\cD)$ around $\phi$. We will see that $G_{\psi}$ is always finite, which is necessary for $\cC^k(\cC,\cD)$ to be proper. We first state the following special case, which is helpful in applications.
\begin{cor}[{\cite[Theorem 1.6]{Che05}}]\label{cor:smooth}
Let $\phi \in C^k(\cC,\cD)$ and assume $|\cC|$ is connected. Then the isotropy of $\phi$ in the groupoid $\cC^k(\cC,\cD)$ is a subgroup of the isotropy of $|\phi|(p)$ for any $p\in |\cC|$. In particular, if $\mathrm{im} |\phi| $ contains a smooth point, then $\phi$ is smooth in  $\cC^k(\cC,\cD)$.
\end{cor}
\begin{proof}
We claim that, if $|\cC|$ is connected, then $G_{\psi}$ is a subgroup of each of the $G'_{\sigma(i)}$'s; this would conclude.

Note first that the action of $G_j'$ on $\pi_0(\Gamma(V'_i,V'_j))$ by $(g,\xi)\mapsto g\circ \xi$ has no fixed point for $g\ne \Id \in G_j'$. 
Now, if $\Gamma(V_i,V_j)\ne \emptyset$, then so is $\Gamma(V_i',V_j')$ because of the existence of $\phi$; this implies that $g_i$ determines $g_j$ uniquely by the relation \eqref{eqn:natural_transformation}. 
Since $|\cC|$ is connected, for any two $V_i,V_j$, there exists a sequence $V_{k_1},\ldots, V_{k_n}$ with $V_{k_1}=V_i,V_{k_n}=V_j$ and $\Gamma(V_{k_s},V_{k_{s+1}})\ne \emptyset$. Hence $\{g_i\in G'_{\sigma(i)}\}$ is determined by any one of them, so that $G_\psi$ is in fact a subgroup of each of the $G'_{\sigma(i)}$, and not only of $\prod_I G_{\sigma(i)}'$.

To finish the proof, for a functor $\phi$, we can find good refinements $\Gamma\{V_i\},\Gamma\{V'_i\}$ such that $V'_{i_0}$ contains a point $q$ in the image of the functor with $|q|=p$ in $|\cC|$, and $V'_{i_0}$ is a uniformizer around $q$.
The claim about smoothness of $\cC^k(\cC,\cD)$ at $\phi$ then follows from what was proven in the previous paragraph.
\end{proof}

\begin{remark}\label{rmk:smooth}
The proof of \Cref{cor:smooth} does not use the Banach orbifold structure from \Cref{thm:morphism_orbifold} but just the groupoid structure on $\cC^k(\cC,\cD)$. In particular, $|\cC|$ could be non-compact. \Cref{cor:smooth} also holds for Sobolev case as long as the map is continuous by Sobolev embedding.
\end{remark}

So far, we proved that $\cC^k(\cC,\cD)$ is an \'etale Banach groupoid. In the infinite-dimensional case that we consider here, the properness of a groupoid needs the following modification.
\begin{defn}
\label{def:infinite_dim_groupoid}
$\cC$ is a proper \'etale Banach groupoid iff $\Obj(\cC),\Mor(\cC)$ are Hausdorff spaces modeled on open sets in Banach spaces with smooth transition maps with the same properties for structural maps as in \Cref{def:groupoid} except that the proper property is changed to the following:
\begin{enumerate}
    \item[($\dagger$)](Proper.) For every $x\in \Obj(\cC)$, there is an open neighborhood $U_x$, such that $t:s^{-1}(\overline{U_x})\to \Obj(\cC)$ is proper.
\end{enumerate}
\end{defn}

In the finite-dimensional case, since $\Obj(\cC)$ is locally compact, the two notions of properness in \Cref{def:groupoid} and \Cref{def:infinite_dim_groupoid} are equivalent. 
However, in the infinite-dimensional case, we no longer have local compactness. 
Such properness above is also used in the definition of polyfolds \cite[Definition 7.1.3]{polyfold}, and it implies that the orbit space with the quotient topology is Hausdorff \cite[Lemma 7.3.4]{polyfold}.

\begin{prop}\label{prop:proper}
$\cC^k(\cC,\cD)$ is a proper \'etale Banach groupoid.
\end{prop}
\begin{proof}
By \cite[Proposition 3.9]{quotient}, it suffices to prove the following three properties of the \'etale groupoid $\cC^k(\cC,\cD)$.
\begin{enumerate}
    \item\label{p1} The isotropy group $\mathrm{stab}_\phi$ is finite for any object $\phi$.
    \item\label{p2} For every object $\phi$, there is an open neighborhood $U_\phi$, such that for any $\psi\in U_\phi$, $\{\alpha\in \Mor| s(\alpha)=\psi, t(\alpha)\in U_\phi\}$ has exactly $\mathrm{stab}_\phi$ elements.
    \item\label{p3} There exists an open neighborhood $V_\phi\subset \overline{V_\phi}\subset U_\phi$ of $\phi$, such that $t(s^{-1}(\overline{V_\phi}))$ is a closed subset in the space of objects. 
\end{enumerate}
\eqref{p1} follows from the proof of \Cref{cor:smooth} and the fact that $|\cC|$ has finitely many connected components. 
\eqref{p2} follows from the description of the local uniformizers given before the statement of \Cref{cor:smooth}. 
Hence, it only remains to prove \eqref{p3}. 

For this, let $V_{\phi}$ be a small open ball contained in the local uniformizer around $\phi:\Gamma\{U_i\}\to \cD$. 
By the Arzel\'a-Ascoli theorem, we can assume that $V_{\phi}$ has a compact closure in the $C^{k-1}$ topology. 
We can also assume such closure $\overline{V_\phi}$ is contained in a local uniformizer $\widetilde{V}_{\phi}$ in $C^{k-1}\mathrm{CMor}(\cC,\cD)$. 

Now, assume by contradiction that \eqref{p3} doesn't hold for $V_\phi$. 
Since $C^k\mathrm{CMor}^0(\cC,\cD)$ is locally metrizable, there is a sequence of functors $\psi_i:\Gamma\{U'_i\}\to \cD$ converging to a functor $\psi_\infty:\Gamma\{U'_i\} \to \cD$ not in  $t(s^{-1}(\overline{V_\phi}))$ and a sequence of functors $\phi_i\in \overline{V_\phi}$, such that, for all $i<\infty$, there is a natural transformation $\alpha_i$ as follows:

\[
\xymatrix{
& &  \Gamma\{U_i\}\ar[lld]\ar[rrd]^{\phi_i}  &  \ar@<-3ex>@{}[dd]^(0.3){}="a"^(0.7){}="b" \ar@{=>}^{\alpha_i} "a";"b"  &\\
\cC &  & \Gamma\{U_i\cap U'_j\}\ar[u]\ar[d]  &   & \cD \\
&  &\Gamma\{U'_i\}\ar[llu]\ar[rru]_{\psi_i}  &  &}
\]
Now, by the assumption that $V_\phi$ has compact closure, (a subsequence of) $\phi_i$ converges to some $\phi_\infty$ in $C^{k-1}\mathrm{CMor}^0(\cC,\cD)$. 
By \eqref{p1} and \eqref{p2}, we know that there are open neighborhoods $\widetilde{V}_{\psi_\infty},\widetilde{V}_{\phi_\infty}$ of $\psi_\infty,\phi_\infty$ in $C^{k-1}\mathrm{CMor}^0(\cC,\cD)$ such that the number of transformations between $f\in \widetilde{V}_{\psi_\infty}, g\in  \widetilde{V}_{\phi_\infty}$ is universally bounded. 
As explained in the proof of \Cref{prop:natural_transformation}, each of the natural transformations $\alpha_i$ is represented by a locally constant collection $\{\xi_j\}_{J_i}$. 
As a result, only finitely many such $\{\xi_j\}_{J_i}$ are needed to describe the natural transformations between $\widetilde{V}_{\psi_\infty}$ and $\widetilde{V}_{\phi_\infty}$. 
Hence there is a subsequence of $\alpha_i$ given by the same collection $\{\xi_j\}_{J}$, and such that $t(\alpha_i)$ converges to $\psi_\infty$ and $s(\alpha_i)$ converges to $\phi_\infty$. 
Then, $\alpha_i$ converges to $\alpha_\infty$ in $C^{k-1}\mathrm{CMor}^1(\cC,\cD)$ with $s(\alpha_\infty)=\phi_\infty$ and $t(\alpha_\infty)=\psi_\infty$. 
Now observe that, in the description of natural transformation $\alpha$ using $s(\phi)$ (or $t(\phi)$) and the collection $\{\xi_i\}$, the regularity of the functor/orbifold map $s(\phi)$ is the same as the regularity of $t(\phi)$\footnote{This is important for the construction of the level structures in the polyfolds of maps between orbifolds.}.  
In particular, $\phi_\infty$ is in $C^k\mathrm{CMor}(\cC,\cD)$. 
Moreover, since the $\alpha_i$'s are induced from the same $\{\xi_j\}_{J}$ and $\psi_i$ converges in the $C^k$ topology, we have that $\phi_i$ also converges in the $C^k$ topology. 
Hence, $\phi_\infty\in \overline{V_{\phi}}$ and $\alpha_\infty\in s^{-1}(\overline{V_\phi})$ with $t(\alpha_\infty)=\psi_\infty$, thus reaching a contradiction.
This concludes the proof.
\end{proof}

Chen's construction \cite{Chen06} uses functors between good covering groupoids represented as in Lemma \ref{lemma:local}, equivalences of self natural transformations in the form of \eqref{eqn:natural_transformation}, and common refinements \cite[Definition 2.2.6]{Chen06}. 
Nevertheless, the set of equivalence classes of orbifold maps in \cite{Chen06} is isomorphic to $C^k(\cC,\cD)$, and the topology coincides with the quotient topology from the Banach groupoid $\cC^k(\cC,\cD)$, since the base of the topology we describe is precisely the base in \cite{Chen06}. 
Therefore, we have the following.
\begin{cor}[{\cite[Theorem 1.4]{Chen06}}]\label{prop:banach_orbifold}
The orbit space $C^k(\cC,\cD)$ with the quotient topology is a Hausdorff and paracompact space. In particular, $C^k(\cC,\cD)$ is a Banach orbifold.
\end{cor}

Here, \emph{Banach orbifold} is to be intended in a completely analogous way to the definition of finite-dimensional orbifolds as in \Cref{defn:orbifold} but using the notion of proper \'etale Banach groupoid in \Cref{def:infinite_dim_groupoid}.

To be more precise, \cite[Theorem 1.4]{Chen06} shows that $C^k(\cC,\cD)$ is Hausdorff and second countable. By \Cref{prop:proper} and \cite[Lemma 7.3.5]{polyfold}, we have $C^k(\cC,\cD)$ is a regular space (i.e.\ a $\mathrm{T_3}$ space). 
Then the second countability together with \cite[Theorem 1.2]{paracompact} imply that $C^k(\cC,\cD)$ is paracompact. 
Note that, paracompactness, instead of second countability, is the property we really need, as we will need partitions of unit.

\subsubsection{Step 3: independence of the construction on the choices\protect\footnote{Note that the independence is not required for geometric applications in this paper, as we only need one Banach orbifold structure to define one symplectic cohomology to obtain the geometric applications.}.}
Let $\Sigma,W$ be two orbifolds and assume $\Sigma$ is compact; consider moreover two orbifold structures $(\cC,\alpha),(\cD,\beta)$ of $\Sigma,W$ respectively.
By definition, we have that $C^k(\cC,\cD)$ is isomorphic to $C^k(\Sigma,W)$ as a set. Therefore, \Cref{prop:banach_orbifold} endows $C^k(\Sigma,W)$ with a topology and the structure of a Banach orbifold. 
To finish the proof of Theorem \ref{thm:morphism_orbifold}, we still need to show the construction does not depend on the choices of orbifold structures. In the following, we give a sketch of the proof.

\begin{prop}\label{prop:right_invariance}
Let $\phi:\cD\to \cE$ be an equivalence. Then $\cC^k(\cC,\cD)$ and $\cC^k(\cC,\cE)$ are Morita equivalent.
\end{prop}
\begin{proof}
We first assume $\phi$ is a full equivalence.
Let $L_\phi$ be the functor from $\cC^k(\cC,\cD)$ to $\cC^k(\cC,\cE)$ obtained by composing with $\phi$. 
Since any generalized map $\cC \leftarrow \cX \to \cE$ is equivalent to $\cC \leftarrow \cX \leftarrow \cX\times_{\cE} \cD \to \cD \stackrel{\phi}{\to} \cE$, $L_\phi$ is essentially surjective. 
Up to refining, one can moreover assume that the good covering $\Gamma\{V'_i\}$ of $\cD$ in the construction of local uniformizers above the statement of \Cref{cor:smooth} satisfies that $\phi|_{V'_i}$ is a diffeomorphism onto the image. 
Then, up to composing with the full equivalence $\phi$, $\Gamma\{V'_i\}$ can also be viewed as a good covering of $\cE$. 
It hence follows from the construction of local uniformizers that $L_\phi$ is a local diffeomorphism and fully faithful. 

Next, let $\phi:\cD\to \cE$ be a general equivalence. 
Then, it is easy to check that $\cE\times_{\cE}\cD\to \cD$ and $\cE\times_{\cE}\cD\to \cE$ are full equivalences.
As a consequence, $\cC^k(\cC,\cD)$ and $\cC^k(\cC,\cE)$ are both Morita equivalent to $\cC^k(\cC,\cE\times_{\cE}\cD)$, and hence Morita equivalent to each other.
\end{proof}

\begin{prop}\label{prop:left_invariance}
Let $\phi:\cC\to \cE$ be an equivalence. Then $\cC^k(\cC,\cD)$ and $\cC^k(\cE,\cD)$ are Morita equivalent.
\end{prop}
\begin{proof}
As before, it is sufficient to prove the case where $\phi$ is a full-equivalence. 
Then, we consider the groupoid $\widetilde{\cC}^k(\cC,\cD)$ whose objects are functors $\Gamma\{U_i\}\to \cD$ from a covering groupoid $\Gamma\{U_i\}$ to $\cD$, which verify the additional condition that $\phi|_{U_i}$ is a diffeomorphism onto the image. 
Similarly, we define $\widetilde{\cC}^k(\cE,\cD)$ using covering groupoids of $\cE$.
Then it is straightforward to check that the inclusions $\widetilde{\cC}^k(\cC,\cD)\hookrightarrow \cC^k(\cC,\cD)$ and $\widetilde{\cC}^k(\cE,\cD)\hookrightarrow \cC^k(\cE,\cD)$ are equivalences. 
Moreover, the functor $\widetilde{\cC}^k(\cE,\cD)\to \widetilde{\cC}^k(\cC,\cD)$ induced by the pre-composition with $\phi$ is also an equivalence. This concludes the proof.
\end{proof}

Combining \Cref{prop:banach_orbifold}, \Cref{prop:right_invariance} and \Cref{prop:left_invariance}, we conclude the proof of the independence on the choices in the statement of Theorem \ref{thm:morphism_orbifold}.

\subsubsection{Simplification of the orbifold of maps.}\label{sss:simplication}
In the construction of $\cC^k(\cC,\cD)$, we needed to refine the covering groupoids of $\cC$ in order to obtain all maps between $\cC,\cD$. 
This is indeed necessary, as we will prove below that some covering groupoid cannot detect all maps between $\cC,\cD$.
On the other hand, it is helpful to determine when $\cC^k(\cC,\cD)$ has an equivalent description by honest functors and natural transformations between two fixed Lie groupoids, so that we have a ``global" description of the groupoid of maps. Note that this part does not concern the Banach orbifold structure, but just the groupoid structure in step 1 of the construction in \S \ref{sss:groupoid}. In particular, $|\cC|$ does not have to be compact in this subsection.

\begin{prop}\label{prop:global}
Let $\cD$ be an action groupoid $G\ltimes M$. 
Assume that a covering groupoid $\Gamma\{U_i\}$ of $\cC$ has the property that each $U_i$ is a simply connected manifold (in particular, second countable). 
Let also $\mathcal{CF}^k(\Gamma\{U_i\},\cD)$ denote the category (groupoid) whose objects are $C^k$ functors from $\Gamma\{U_i\}$ to $\cD$, and whose morphisms are $C^k$ natural transformations. 
Then, the natural inclusion $\mathcal{CF}^k(\Gamma\{U_i\},\cD)\to\cC^k (\cC,\cD)$ is an equivalence of categories.
\end{prop}
\begin{proof}
The natural inclusion $\mathcal{CF}^k(\Gamma\{U_i\}_{I},\cD)\to\cC^k (\cC,\cD)$ is fully faithful by definition. It remains to prove essential surjectivity. 

Let $\psi$ be any $C^k$ functor $\Gamma\{V_j\}_{J}\to \cD$. 
By \Cref{prop:cover_equivalence}, it is enough to prove that, after a natural transformation, any $C^k$ functor $\psi\colon \Gamma\{V_j\}_{J}\to \cD$, with $\cup_{j}V_j$ any covering of $C=\Obj(\cC)$, can be glued to give a functor from $\Gamma\{U_i\}_{I}$ to $\cD$.

First note that, since any refinement of $\psi$ of the form $\Gamma\{O_k\}\to \Gamma\{V_j\}\xrightarrow{\psi}\cD$ is isomorphic to $\psi$ in $\cC^k(\cC,\cD)$, we may assume that $J$ has a decomposition $\sqcup_{i\in I} J_i=J$, such that $U_i=\cup_{j\in J_i} V_j$. 
Since each $U_i$ is a manifold, we may also assume $V_j\cap V_{j'}$ is connected for $j,j'\in J_i$, by requiring, for instance, that the $V_j$'s are geodesically convex in some complete metric on $U_i$.

By definition of functor $\psi$, for any $j,j'\in J_i$, the identity morphism $\Id_x$ viewed as in $\Gamma(V_j,V_{j'})$ for $x\in V_j\cap V_{j'}$ is sent by $\psi$ to a morphism in $\cD$; we use $g_{x,j,j'}$ to denote the element in $G$ representing the connected component containing the image morphism. 
Since $V_{j}\cap V_{j'}$ is connected, we moreover have that $g_{x,j,j'}$ is independent of $x$ in $V_j\cap V_{j'}$, hence it can be denoted simply by $g_{j,j'}$. 

Now, the functorial property of $\psi$ implies that $g_{j,j'}$ satisfies the cocycle condition $g_{j',j''}g_{j,j'}=g_{j,j''}$ if $V_j\cap V_{j'}\cap V_{j''}\ne \emptyset$ for $j,j',j''\in J_i$. 
Moreover, since $U_i$ is simply connected, the first {\v C}ech cohomology with $G$-coefficients of the open cover $\{V_j\}_{J_i}$ vanishes. 
In particular, there is a primitive $\{g_j\in G\}_J$, such that $g_{j,j'}=g_{j'}^{-1}g_j$. 

Then, the assignment $V_j\mapsto g_j$ defines a natural transformation from $\psi\colon \Gamma\{V_j\}_J\to \cD$ to another functor $\psi'\colon\Gamma\{V_j\}_J\to\cD$, which at the object level is just $x\in V_j\mapsto g_j\psi_j(x)$.
It is now direct to check that $\psi'$ can actually be glued to give a functor $\phi$ on $\Gamma\{U_i\}$, hence concluding the proof.
\end{proof}

The groupoid $\mathcal{CF}^k(\Gamma\{U_i\},\cD)$ can be equipped with a smooth structure in a similar way to what was done in the proof of \Cref{prop:Banach_refinement}.
Then, under the assumptions of \Cref{prop:global}, the natural inclusion $\mathcal{CF}^k(\Gamma\{U_i\},\cD)\to\cC^k (\cC,\cD)$ is an equivalence of proper \'etale Banach groupoids.
In particular, $\mathcal{CF}^k(\Gamma\{U_i\},\cD)$ is an orbifold structure for the orbifold of maps. 
Hence, in view of \Cref{thm:morphism_orbifold}, if $W$ is a global quotient, we can describe the orbifold structure on $C^k(\Sigma,W)$ by the quotient of the space of functors by the natural transformations, provided we represent $\Sigma$ by a covering groupoid $\Gamma\{U_i\}_I$ such that each $U_i$ is a simply connected manifold.

\begin{example}\label{ex:refine}
  As an example, the inertia groupoid $\wedge \cD$ is the groupoid of constant loops, i.e.\ of constant maps from $\bbS^1$ to $\cD$. 
  Since we are considering constant loops, locally, we can consider the uniformizers around the image, which is an action groupoid. 
  Then, we can consider the groupoid $\cS^1$, representing $\bbS^1$, with objects the points of an open interval whose morphism set consists of the identity at each point and of the natural identification of two open smaller intervals at the two ends. 
  Then, by Proposition \ref{prop:global}, the groupoid of constant (on object level) functors from $\cS^1$ to $\cD$ is precisely $\wedge \cD$. 
  On the other hand, if we use the trivial groupoid to represent $\bbS^1$, i.e.\ the one with only identity morphisms, the functors to $\cD$ only detect the maps corresponding to $\cD\subset \wedge \cD$. \END
\end{example}

\begin{example}\label{ex:constant}
   Given a proper étale groupoid $\cD$, $\cD^k$ denotes the groupoid of constant maps from the 2-sphere with $k$ punctures $\mathring{S}_k$ to $\cD$. 
   One way to obtain a groupoid whose object space has simply connected components is to choose a point on $\mathring{S}_k$, then cut the surface open along rays connecting the point with each puncture;
   the simply connected components are then the surface resulting from this cut (which is homeomorphic to a disc) and small contractible neighborhoods of each of the rays.
   Then, because of the configuration of intersections between these open subsets, a constant functor from this groupoid to $\cD$ is equivalent to an assignment, for each cut, of an element $g_i$ in the isotropy group of the point which constitutes the image of the functor at the object level. 
   By looking at the intersection point of the cuts, one can see that the relation $g_1\ldots g_k=\Id$ must hold (assume for simplicity that the rays have been chosen such that cyclic order is the same as the indices). 
   Natural transformations are then just given by conjugations of the $g_i$'s. 
   
   One can also view the elements of $\cD^k$ as constant maps from the orbifold sphere $S_k$ with at most $k$ orbifold point whose isotropy are the cyclic groups $\ZZ/o(g_i)\ZZ$, with $o(g_i)$ the order of $g_i$ in the isotropy group of the image point; in other words, $\cD^k$ is nothing else than the space $\mathcal{M}_k(\cC)$ considered in \S \ref{SS:CR}.
   Notice that, a priori, in addition to the charts describing $\mathring{S}_k$, we need to add $k$ charts $\ZZ/o(g_i)\ZZ\ltimes \DD$. 
   In order to build a constant functor with fixed $o(g_i)$'s, we then also need to assign an element from $G$ to one of the morphisms connecting the disk chart near each puncture with the ``central'' chart, i.e.\ the one coming from the cut of $\mathring{S}_k$ (all others will then be determined by the functoriality). 
   However, different choices give equivalent functors via the natural transformation given by conjugation on each disk chart.  In particular, this gives rise to a groupoid that is equivalent to the groupoid of constant maps from $\mathring{S}_k$.
   
   We point out, even though the second interpretation of $\cD^k$ in terms of constant maps from orbifold spheres may seem more natural at first sight, in this setting it is more constructive to think of $\cD^k$ as orbifold maps from $\mathring{S}_k$.
   In general, when we consider orbifold Gromov-Witten invariants or orbifold symplectic field theory, we also consider maps from punctured spheres $\mathring{S}_k$ to the target symplectic manifold, such that the map has certain exponential decay near the punctures. If one of the punctures is a removable singularity, but the limit is a singular point of the target symplectic manifold, then we may need to introduce a singularity on the sphere if we want to extend the map to the puncture, and whether this is necessary depends on the limit of the nearby loops near the puncture in the inertia orbifold $\wedge  \cD$. 
   This perspective also explains why constraints in orbifold Gromov-Witten theory are from the inertia groupoid. \END
\end{example}

\begin{remark}
An alternative, but equivalent way of understanding the maps to a global quotient using $G$-bundles can be found in \cite[Theorem 2.45]{Ruan07}, which is used in \cite[\S 2.5]{Ruan07} to explain $\wedge \cD$ and $\cD^k$.
\end{remark}

We end this subsection by the following slightly more global version of \Cref{prop:global}.
\begin{prop}\label{prop:global2}
Let $\Gamma\{V_j\}_{J}$ be a good covering groupoid of $\cD$. 
Suppose $\cC$ has a covering groupoid $\Gamma\{U_i\}_{I}$ with $I$ finite and $U_i$ simply connected.
Consider also a partition $I=\bigsqcup_{i=1}^n I_i$ and a function $\sigma:\{1,\ldots,n\}\to J$, and let $\mathcal{CF}^k(\Gamma\{U_i\},\Gamma\{V_j\},\sigma)$ denote the groupoid whose objects are $C^k$ functors from $\Gamma\{U_i\}$ to $\Gamma\{V_j\}$ sending $U_i$ to $V_{\sigma(s)}$ if $i\in I_s$. 
Then $\mathcal{CF}^k(\Gamma\{U_i\},\Gamma\{V_j\},\sigma)$  is equivalent to the full subgroupoid of $\cC^k(\cC,\cD)$ whose orbit space consists of maps $\phi$ such that  $\vert\phi\vert(|U_i|)\subset |V_{\sigma(s)}|$ for $i\in I_s$.
\end{prop}

\begin{proof}
We use $\cC^k(\cC,\Gamma\{V_j\},\sigma)$ to denote the full subgroupoid of $\cC^k(\cC,\Gamma\{V_j\})$ whose orbit space consists of maps $\phi$ such that $\phi(|U_i|)\subset |V_{\sigma(s)}|$ for $i\in I_s$. 
As before, the natural inclusion $\mathcal{CF}^k(\Gamma\{U_i\},\Gamma\{V_j\},\sigma)\subset \cC^k(\cC,\Gamma\{V_j\},\sigma)$ is fully faithful. 
We need to show that the inclusion is essentially surjective. 
In other words, as in the proof of \Cref{prop:global}, given a functor $\psi:\Gamma\{W_k\}_{K}\to \Gamma\{V_j\}$, such that $K=\sqcup_{i\in I} K_i$, $\cup_{k\in K_i} W_k=U_i$ and $|\psi|(|U_i|)\subset |V_{\sigma(s)}|$ for $i\in I_s$, we want to glue $\psi$ (possibly after natural transformation) to a functor from $\Gamma\{U_i\}$. 

We can assume (up to substituting with a refinement) that the $W_k$'s are geodesically convex. 
Note now that for $i\in I_s,k\in K_i$, even though we have $|\psi|(|W_k|)\subset |V_{\sigma(s)}|$, we may not have $\psi_k(W_k)\subset V_{\sigma(s)}$;
assume $\psi_k(W_k)\subset V_{p}$.
Since $W_k$ is connected, by the goodness of the covering $\{V_j\}$, we can find $\xi\in \pi_0(\Gamma(V_p,V_{\sigma(s)}))$ such that $\widehat{\xi}\circ \psi_k(W_k)\subset V_{\sigma(s)}$. 
Therefore by considering the natural transformation $\tau:x\in W_k\mapsto (\psi_k(x) \stackrel{\widehat{\xi}}{\to} \widehat{\xi}\circ \psi_k(x))$, we get a functor $\psi'$ equivalent to $\psi$, such that $\psi'_k(W_k)\subset V_{\sigma(s)}$. 
Hence, we can assume that for $i\in I_s,k\in K_i$ we have $\psi_k(W_k)\subset V_{\sigma(s)}$.  
Then we can glue $\psi$ on $\Gamma\{W_k\}_{\cup_{i\in I_s}K_i}$ to a functor on $\Gamma\{U_i\}_{I_s}$ by \Cref{prop:global}. This finishes the proof.
\end{proof}

\begin{remark}
It is not always true that functors from a groupoid whose object space has simply connected components capture all orbifold maps.
The obstruction is often already in the extension at the object level from a refined representation. 
For example, consider maps from the interval $(0,3)$ to itself. 
Then functors from $(0,3)=\Gamma\{(0,3)\}$ to the covering groupoid $\Gamma\{(0,2), (1,3) \}$ do not capture the identity map. 
\end{remark}

\subsubsection{Non-compact domain case.}\label{ss:non-compact} We will not consider the general non-compact case, as there seems to be no general framework to define a Banach orbifold structure on spaces of maps from a general non-compact domain. 
We will only discuss the special case of $\Sigma =\RR\times S^1$, and leverage geometric intuitions from our motivation to define Floer theory.  

\smallskip

Let $W$ be an orbifold and $x,y$ be two smooth loops in $W$ parametrized by $t\in [0,1]$, i.e.\ orbifold maps from $S^1=[0,1]/\hspace{-5pt}\sim \;  \to W$, where $\sim$ is defined by $0\sim 1$ in $[0,1]$. 
For a given $\delta>0$, we first form a subset $C^{\infty, \delta}(\RR\times S^1, W,x,y)\subset C^{\infty}(\RR\times S^1, W)$ of maps \emph{asymptotic to $x,y$ exponentially with weight $\delta$}. 

To describe the precise meaning of this asymptotic exponential convergence, we consider a groupoid $\cD$ modeling the orbifold $W$. 
After choosing a good covering groupoid $\Gamma\{V_i\}$ of $\cD$, we can represent the loop $x$ by a functor $\Gamma\{I_i\}\to \Gamma \{V_i\}$, where $\Gamma\{I_i\}$ is a covering groupoid of $S^1$ from a finite cover of $S^1$ by intervals $I_i$.
We can also do similarly for the other loop $y$, obtaining similarly a finite cover by intervals $J_j$ of $S^1$. 
We then say that a map $u\in C^\infty(\RR\times S^1,W)$ is in $C^{\infty, \delta}(\RR\times S^1, W,x,y)$ if it satisfies the following properties.
\begin{itemize}
    \item First, we require that there is $N\in \RR_{>0}$ so that $u$ is representable by a functor from a covering groupoid of $\RR\times S^1$, with respect to a finite cover which contains each open subset $(N,+\infty)\times I_i \subset \RR\times S^1$ (whose union over $i$ covers the positive end of the domain), as well as each open subset $(-\infty,-N)\times J_j \subset \RR\times S^1$ (whose union over $j$ covers the negative end of the domain).
    \item On $(N,+\infty)_s\times I_i$, $u(s,\cdot)$ is asymptotic to $x(\cdot)$ on $I_i$ with an exponential decay (w.r.t.\ some auxiliary metric) of weight $\delta$ for $s\to \infty$.
    Similarly, on $(-\infty,-N)_s\times J_j$, $u(s,\cdot)$ is asymptotic to $x(\cdot)$ on $J_j$ with an exponential decay of weight $\delta$ for $s\to -\infty$.
\end{itemize}
It is straightforward to check that $C^{\infty, \delta}(\RR\times S^1, W,x,y)$ is canonically defined, i.e.\ independent of the covering groupoids as well as metrics on $W$. 

\smallskip

Now, for $k\geq 0$ and $0<\delta'\le \delta$, we can define $C^{k,\delta'}(\RR\times S^1, W,x,y)$ by push-offs of maps in $C^{\infty, \delta}(\RR\times S^1, W,x,y)$ by sections with finite weighted $C^k$-norm, with weights $e^{\delta'|s|}$ near the punctures, from a covering groupoid representation of maps in $C^{\infty, \delta}(\RR\times S^1, W,x,y)$ in the proof of \Cref{prop:Banach_refinement}. 
In other words, here we need to use the localization (with respect to covering groupoids) strategy in the proof of \Cref{prop:Banach_refinement} in order to construct local charts around each point of $C^{\infty, \delta}(\RR\times S^1, W,x,y)$, which is necessary to define $C^{k,\delta}(\RR\times S^1, W,x,y)$; hence, \Cref{prop:Banach_refinement} becomes at this point tautological, as we have just exhibited Banach charts. 
As the maps in $C^{\infty, \delta}(\RR\times S^1, W,x,y)$ can be realized by functors from a finite cover, the strategy in the proof of \Cref{thm:morphism_orbifold} hence applies and makes $C^{k, \delta'}(\RR\times S^1, W,x,y)$ into a Banach orbifold. 

\Cref{cor:smooth} holds as well. 
Lastly, the construction for $W^{k,p,\delta}(\RR\times S^1, W,x,y)$ with $\delta>0$ is similar. 

\begin{remark}
    The condition $\delta'\le \delta$ is necessary to obtain the paracompactness. 
    For example, one can consider $C^{0,\delta}(\RR,\RR)$ the Banach space with norm $|v|=|e^{\delta|s|}v|_{C^0}$.
    If $\delta'>\delta$, then the construction above would give $C^{0,\delta}(\RR,\RR)+C^{0,\delta'}(\RR,\RR)=C^{0,\delta}(\RR,\RR)$. 
    However, the local topological basis is induced from $C^{0,\delta'}(\RR,\RR)$ instead. It is clear that $C^{0,\delta}(\RR,\RR)$ is not paracompact in this topology.
\end{remark}

We summarize the above discussion in the following statement:
\begin{thm}\label{thm:map_cylinder}
    Let $W$ be an orbifold and $x,y$ two orbifold loops in $W$. 
    For any pair of positive integers $(k,p)$ with $k>1$ and $2<(k-1)p$, and any real $\delta>0$, in the set $C^0(\RR\times S^1,W,x,y)$ of continuous orbifold maps from $\RR\times S^1$ to $W$ with asymptotics $x,y$, there are subsets $C^{k,\delta}(\RR\times S^1,W,x,y)$ and $W^{k,p,\delta}(\RR\times S^1,W,x,y)$ made of $C^k$- and $W^{k,p}$-maps with ``asymptotic convergence with weight $\delta$ at the punctures'' (with the precise meaning as discussed above), which are both naturally equipped with a Banach orbifold structure.
    Moreover, they both contain the set $C^{\infty, \delta}(\RR\times S^1, W,x,y)$.
\end{thm}

By the exponential decay property for Floer cylinders described in \S \ref{SS:analytical}, the above construction of the Banach orbifold is large enough to contain all possible solutions to the Hamiltonian-Floer equation, hence serves the purpose of defining Floer theory.

\subsection{Homotopy classes of orbifold maps}
\label{sec:homotopy_groups}
Let $X,Y$ be two orbifolds, two maps (i.e.\ orbifold morphisms) $f,g:X\to Y$ are called \emph{homotopic} if there exits a map $F:X\times [0,1]\to Y$ such that $F|_{X\times \{0\}}=f$ and $F|_{X\times \{1\}}=g$. 
Here, we consider $[0,1]$ with its trivial orbifold structure, i.e.\ represented by the orbifold groupoid with objects space $[0,1]$ and morphisms just the identities.
Then it is easy to check that the existence of a homotopy constitutes an equivalence relation on the set of maps from $X$ to $Y$, which shall be denoted by $[X;Y]$. 
The key difference from the manifold case is that we no longer have unique concatenation of homotopies, i.e.\ given two homotopies $F,H:X\times [0,1]\to Y$ with $F|_{X\times \{1\}}=H|_{X\times \{0\}}$, there will be multiple ways to glue to a larger homotopy $H\circ F$ if $F|_{X\times \{1\}}=H|_{X\times \{0\}}$ has non trivial isotropy in the orbifold\footnote{Or just orbispace if we only consider continuous orbifold maps.} of orbifold maps. 
The simplest example is the following: gluing the two ends of $[0,1]\to \bullet/G$ can be done in $|\Conj(G)|$ many different ways, and each of these gluings give rise to maps $S^1\to \bullet/G$ in different homotopy classes. 
The cost of such ambiguity is that the homotopy classes of pairs $[(S^n,\pt);(X,x_0)]$ is no longer a group, unless the base point $x_0$ has no isotropy. 

(To avoid confusion with the newly introduced notation, for $X$ and $Y$ topological spaces, we will denote by $[X,Y]_{top}$ the space of continuous homotopy classes of continuous maps from $X$ to $Y$.)

\begin{remark}
The homotopy groups as homotopy classes of orbifold maps were studied by Chen \cite{Chen06bis}. However, to overcome the above-mentioned ambiguity in order to construct a group structure, \cite{Chen06bis} considered a \emph{different} orbifold of orbifold maps, where the natural transformation is required to be $\Id$ near the base point of the target.
\end{remark}

To every proper \'etale Lie groupoid $\cC$, one can associate a topological space $B\cC$ called the classifying space\footnote{The precise definition is rather involved and not needed here. As an example, if $\cC=\bullet/G$ then $B\cC$ is the classical classifying space $BG$ for the group $G$.}, which is the geometric realization of the nerve of the category $\cC$ \cite[p.24-25]{Ruan07}. 
The homotopy type of the classifying space is unique for Morita equivalent proper \'etale Lie groupoids \cite[p.25]{Ruan07}, hence we may use $BX$ to denote the homotopy type of the classifying space of any orbifold $X$. 

Now, for a manifold $M$, \cite{MR2048526,lerman2010orbifolds} proved that $[M;\cC]$  classifies $\cC$-principal bundles\footnote{Again, the precise definition is involved and not necessary for our discussion. As an example, if $\cC=\bullet/G$ then a principal $\cC$-bundle is just a principal $G$-bundle.} over $M$ up to concordance\footnote{Two $\cC$-principal bundles over $M$ are \emph{concordant} if they appear as the boundary of a $\cC$-principal bundle over $M\times[0,1]$.}, and \cite{MR263104,MR1112197} proved that the same is true for $[M,B\cC]_{top}$.
In particular, we deduce the following.

\begin{prop}
Let $M$ be a manifold, and $X$ an orbifold. 
Then, $[M;X]$ is in bijection with $[M,BX]_{top}$.
\end{prop}
Using the fact that $[S^1,N]_{top}\simeq \pi_1(N)/\rho$ for $N$ a topological space and $\rho$ the action by conjugation of $\pi_1(N)$ on itself, one can deduce the following corollary, which will be used in \S \ref{S4} and \S\ref{S5}.
\begin{cor}\label{cor:pi0_free_loop_space}
$[S^1;\CC^n/G]$ is isomorphic to $\Conj(G)$.
\end{cor}
Without entering in the technical definition of the loop space of an orbifold, we point out that $[S^1;\CC^n/G]$ should exactly be thought of as the set of connected components of the free loop space of $\CC^n/G$.
\begin{proof}
    By the above discussion, we have a bijection $[S^1;\CC^n/G] \leftrightarrow [S^1,B(\CC^n/G)]_{top}$.
    Now, $B(\CC^n/G)$ is, by definition, given by $EG\times_G \CC^n$. 
    As the action of $G$ is linear, the standard retraction of $\CC^n$ onto the origin is $G$-invariant, and hence induces a homotopy equivalence $EG\times_G \CC^n\simeq EG/G$.
    As $EG/G$ is just $BG$, we have $B(\CC^n/G) \simeq BG$.
    Hence, $[S^1;\CC^n/G]$ is in bijection with $[S^1,BG]_{top}$, that itself is in bijection with $\pi_1(BG)/\mathrm{conj}$.
    The fact that $\pi_1(BG)=G$ allows to conclude.
\end{proof}

\begin{remark}
When the target is a global quotient, one can use \cite[Theorem 2.45]{Ruan07}, which is a special form of identifying orbifold maps with principal bundles, to obtain the above results. One can also obtain \Cref{cor:pi0_free_loop_space} using \Cref{prop:global}.
\end{remark}

In view of the above discussion, one can also define the orbifold homotopy groups as $\pi_n^{\orb}(X):=\pi_n(BX)$. 
An instant corollary is that $[S^1;X]=\Conj(\pi_1^{\orb}(X))$ when $X$ is connected. We also have the following Van Kampen theorem.

\begin{prop}
Assume $X$ is an orbifold, such that $X=U\cup V$ and $W=U\cap V$ are all connected non-empty open sets. Then we have 
$$\pi^{\orb}_1(X)=\pi^{\orb}_1(U)*_{\pi^{\orb}_1(W)}\pi^{orb}_1(V).$$
\end{prop}
\begin{proof}
Let $\cC$ be an orbifold structure of $X$, and $\cU,\cV,\cW$ are the full subcategory representing $U,V,W$. Then the claim follows from the Van Kampen theorem applied to $B\cC=B\cU\cup B\cV$ and $B\cW=B\cU\cap B\cV$.
\end{proof}

\subsection{Complex line bundles and their first Chern classes}
\label{sec:complex_line_bundles}

In this section, we will describe a natural extension of classical results on complex line bundles and their first Chern classes.
As this is the case that will be relevant for us, we make the assumption that the given orbifold $X$ is \emph{effective}.

This assumption enters into play in this discussion via \cite[Corollary 1.52]{Ruan07}:
for effective $X$,
the classifying space $BX$ of $X$ is homotopy equivalent to $EO(n) \times_{O(n)} Fr(X)$.
Here, $Fr(X)$ is the orthogonal frame bundle of $X$, which is a \emph{smooth manifold} when $X$ is effective \cite[Theorem 1.23]{Ruan07}, $EO(n)$ is the total space of the universal $O(n)$-bundle $EO(n)\to BO(n)$ over the classifying space $BO(n)$, and $EO(n) \times_{O(n)} Fr(X)= EO(n) \times Fr(X) / O(n)$ where $O(n)$ acts diagonally in the natural way.
In other words, the orbifold (co)homology $H^{\orb}_*(X;R)$, $H^*_{\orb}(X;R)$ (where $R$ is any commutative ring) and the orbifold homotopy groups $\pi_n^{\orb}(X)$, all just defined as (co)homology and homotopy groups of the classifying space $BX$, are respectively just the \emph{equivariant} (co)homology and homotopy groups of the smooth manifold $Fr(X)$ w.r.t.\ its natural $O(n)-$action.
(Even though not fundamental for our purposes, we point out that, for a \emph{field} R, $H^{\orb}_*(X;R)$, $H^*_{\orb}(X;R)$ simply reduce respectively to the singular homology and singular cohomology of $X$ simply seen as a topological space (see e.g.\ \cite[Proposition 2.12]{Ruan07} for the case of $R=\QQ$); in other words, all the interesting information about singularities is contained, for general ring $R$, in the torsion of these two modules.)

For notational simplicity, let $M=Fr(X)$ and $G=O(n)$ for the rest of the section. 
Because of what was just said, one can then just look at the $G-$equivariant theories on $M$.
More precisely, 
by definition of vector bundles over groupoids \cite[Definition 2.25]{Ruan07},
any complex line orbibundle $\pi\colon E\to X$ comes from a (smooth) complex line bundle $\tau\colon L\to M$ which has a natural $G-$action for which $E=L/G$.
We then get the following commutative diagram, where the horizontal arrows are the natural quotient maps:
\begin{equation*}
\begin{tikzcd}
L
\arrow[r]
\arrow[d, "\tau"]
& 
E
\arrow[d, "\pi"] 
\\
M
\arrow[r]
&  
X
\end{tikzcd}
\end{equation*}
The converse is also true, namely that every complex line bundle over $M=Fr(X)$ whose total space has an action of $G=O(n)$ making the projection $G-$equivariant comes from a complex line orbibundle over $X=M/G$.

Now, each complex $G-$line bundle $L\to M$ has an equivariant first Chern class $c_1^G(L)$, defined as the first Chern class of the associated line bundle $EG\times_G L \to EG\times_G M$, which lives in $H^2(EG\times_G M;\bbZ)$, that is (by definition) the second equivariant cohomology group $H^2_{G}(M;\bbZ)$ of $M$.
Notice that, in our setting, $H^2_G(M;\bbZ)$ is (again by definition) just $H^2_{\orb}(X;\bbZ)$, where $X=M/G$.
Moreover, by the results in \cite{HatYos76}, such $c_1^G$ gives an isomorphism between the set of $G-$isomorphism classes of complex $G-$line bundles $L\to M$ and the $G-$equivariant cohomology group $H^2_G(M;\bbZ)$.
Hence, complex line orbibundles over $X$ are in correspondence with $H^2_{\orb}(X;\bbZ)$ via the (equivariant) first Chern class of the associated complex $G-$line bundle $L\to M$.

	\section{Symplectic cohomology of exact orbifold fillings}\label{S3}

Let $W$ be an exact orbifold filling of a contact \emph{manifold}. In this section, we will discuss the Hamiltonian-Floer theory on $W$. Instead of using abstract virtual techniques like \cite{CheRua01,orbiLagrangian} to deal with the most general case, we use more classical transversality methods like perturbing the almost complex structure as in \cite{MS12,AudDamBook}.

\subsection{Geometric setup}
Let $(\widehat{W},\widehat{\lambda})$ denote the completion of $(W,\lambda)$.
Namely, $\widehat{W}=W\cup [1,\infty)\times M$ with $\widehat{\lambda} = \lambda \cup r\alpha$, with $r\in[1,\infty)$, $M=\partial W$ and $\alpha = \lambda\vert_{M}$ the induced contact form.
Given a smooth Hamiltonian $H_t:S^1\times \widehat{W}\to \RR$ (as an orbifold map), the Hamiltonian vector field $X_{H_t}$ defined by $\rd H_t=\rd\widehat{\lambda}(\cdot,X_{H_t})$ is well-defined as an $S^1-$parametric section of the tangent bundle $T\widehat{W}$. A Hamiltonian orbit is defined as a smooth (orbifold) map $\gamma:S^1\to \widehat{W}$ such that $\gamma'(t)=X_{H_t}(\gamma(t))$. 
An orbit is non-degenerate iff the linearized return map does not have $1$ as an eigenvalue. 

\begin{example}
Let $G\subset U(n)$ be a finite subgroup such that $\CC^n/G$ is an isolated singularity. 
We consider, for $A$ a real number, the Hamiltonian $H= A\,|z|^2$ on $G\ltimes \CC^n$, then $X_H=\sum_{i=1}^n Ax_i\partial_{y_i}-Ay_i\partial_{x_i}$. 
Moreover, if $A\not \in \frac{2\pi}{|G|}\ZZ$, then all $1$-periodic orbits of $X_H$ are constant at $0$, i.e.\ parametrized by the the inertia orbifold $\wedge (\bullet/G)$, and all of them are non-degenerate. \END
\end{example}

\begin{example}\label{ex:Morse}
Let $G\subset U(n)$ be a finite subgroup such that $\CC^n/G$ is an isolated singularity. 
Let $H$ be a function on $G\ltimes \CC^n$, i.e.\ $H$ can be viewed as a $G$-invariant function on $\CC^n$.
We claim $0$ must be a critical point. Assume otherwise that $\nabla H(0)\ne 0$.
As the level set $H^{-1}(H(0))$ near $0$ is $G$-invariant and the singularity is isolated, the induced action on $T_0\CC^n/T_0H^{-1}(H(0))=\langle \nabla H(0) \rangle$ must be non-trivial. 
But this contradicts that $\nabla H$ is $G-$equivariant at $0$.

Now, in a local picture, $(0,0)$ is a Morse critical point of $H=\epsilon(|z_1|^2-|z_2|^2)$ on $\CC^n=\CC^{m}\oplus \CC^{k}$ as a decomposition of the $G$ representation $G\ltimes \CC^{m+k}$, the Conley-Zehnder index of the constant loop $(0,(g))$ is independent of the Morse index of $(0,0)$ when $\epsilon$ is small and $g\ne \Id$. 
The reason for this is that the correction from $g$ in the trivialization of the tangent bundle is bigger than the term from the Hessian of $H$; see \Cref{ex:SU} for details. 
This corresponds to the fact $(0,(g))$ generates a class in the Chen-Ruan cohomology, seen as part of the zero action symplectic cohomology, no matter which $C^2$-small (locally near the origin) $H$ we choose.
Indeed, in the continuation map between the $C^2-$small (locally near the origin) Hamiltonians $\epsilon(|z_1|^2-|z_2|^2)$ and $\epsilon(|z_1|^2+|z_2|^2)$, the constant cylinder over $(0,(g))$ for $g\ne \Id$ always exists and is cut out transversality. In particular, it induces an isomorphism on the cohomology of those non-trivial constant loops. \END
\end{example}

In the following, we assume the Liouville form $\lambda$ restricts to a non-degenerate contact form on $\partial W$. 
Let $\cS(\lambda)$ be the \emph{spectrum of periods of Reeb orbits}. For $a\notin \cS(\lambda)$, we will consider two types of ``admissible'' Hamiltonians of slope $a$. 
A Hamiltonian $H_t:S^1\times \widehat{W}\to \RR$ is called \emph{admissible of type (I)} if the following holds.
\begin{enumerate}
	\item $H_t=0$ on $W$.
	\item $H_t=h(r)$ with $h'(r)=a$ for $r>1+\epsilon$ for some $\epsilon>0$.
	\item The non-constant Hamiltonian orbits of $X_{H_t}$ are non-degenerate.
\end{enumerate}
In this case, the periodic orbits of $X_{H_t}$ are either non-degenerate non-constant orbits outside $W$, or constant orbits parametrized by the inertia orbifold $\wedge W$.

A Hamiltonian is called \emph{admissible of type (II)} if the following holds.
\begin{enumerate}
    \item $H_t$ on $W$ is $C^2$-small Morse and independent of $S^1$.
    \item $\partial_rH_t>0$ on $\partial W$.
    \item $H_t=\sum_{i=1}^{n} \epsilon(x_i^2+y_i^2)+C$ near an isolated singularity $\CC^n/G$, for $G\subset U(n)$ and $0<\epsilon\ll 1$. Here the coordinate and the symplectic form is from \Cref{ex:quotient}, which can be assumed to be the local model near each isolated singularity by \Cref{prop:decomposition} using Moser's trick. In other words, $H_t$ is a Morse function on $W$, such that the singularities are local minima (note that we use here Proposition \ref{prop:decomposition}). 
    \item  $H_t=h(r)$ with $h'(r)=a$ for $r>1+\epsilon$ for some $\epsilon>0$.
    \item The non-constant Hamiltonian orbits of $X_{H_t}$ are non-degenerate.
\end{enumerate}

\begin{remark}
The Morse-ness and $C^2-$smallness of $H_t$ on $W$ imply the facts that $X_{H_t}$ has no non-constant orbits in $W$ and that every constant orbit in it (i.e.\ a critical point $x$ along with a conjugacy class $(g)$ in $\stab_x$) is non-degenerated with natural rational grading $2\age(g)+\ind(x)$ as in \Cref{ex:U}, where $\ind(x)$ is the Morse index of $x$.
\end{remark}

Using a type (II) Hamiltonian results in a more classical construction of the symplectic cohomology, as in e.g.\ \cite{biased} for smooth fillings. 
For type (I) Hamiltonians, we have a ``Morse-Bott" family of constant orbits parametrized by $\wedge W$. 
Therefore, we need the following admissible Morse function to break the symmetry. 
A smooth function $f$ on $W$ is called \emph{admissible} if the following holds.
\begin{enumerate}
	\item $\partial_rf>0$ on $\partial W$.
	\item $f=\sum_{i=1}^{n} (x_i^2+y_i^2)+C$  near an isolated singularity $\CC^n/G$, for $G\subset U(n)$.
	\item $f$ is Morse on the smooth part of $W$.
\end{enumerate}
In other words, $f$ can be chosen as a type (II) Hamiltonian restricted to $W$.
In the smooth filling setting, the construction using a type (I) Hamiltonian along with an admissible Morse function can be found, for instance in, \cite{ADC}, which provides a clean setup for the neck-stretching argument.
In the orbifold setting, the analytical setups for type (I) and (II) Hamiltonians are slightly different, as we shall explain in \S \ref{SS:analytical}, but they give rise to isomorphic symplectic cohomology.

\subsection{Analytical setup for moduli spaces.}\label{SS:analytical}
An $S^1$-dependent almost complex structure $J$ on $\widehat{W}$ is called \emph{admissible} if the following conditions hold.
\begin{enumerate}
    \item $J$ is $S^1$-independent on $W$.
    \item $J$ is compatible with the symplectic form $\rd \widehat{\lambda}$, i.e.\ $\rd \widehat{\lambda}(\cdot, J\cdot)$ is a Riemannian metric.
    \item $J$ is \emph{cylindrically convex} near some $r=r_0>1+\epsilon$, i.e.\ $\widehat{\lambda} \circ J = \rd r$ near $r=r_0$ where the Hamiltonian is linear with slope $a$. 
    (This is sufficient to guarantee the $C^0$-estimate by the integrated maximum principle, see \cite[Lemma 2.5]{ADC}.)
\end{enumerate}
Heuristically speaking, the symplectic cohomology is the direct limit, for the slope of the Hamiltonians going to $+\infty$, of the Morse cohomology of the following symplectic action on the space of free loops of $\widehat{W}$,
\begin{equation}\label{eqn:action}
    \cA_{H_t}(x)=-\int x^*\widehat{\lambda}+\int_{S^1} H_t\circ x (t)\rd t,
\end{equation}
w.r.t.\ the \emph{$L^2$} metric on free loops induced by the metric $\rd \widehat{\lambda}(\cdot, J\cdot)$. In the following, we discuss the construction for type (I) and (II) Hamiltonians separately, as they are slightly different. 

\subsubsection{Type (II) Hamiltonians.}
For a type (II) Hamiltonian $H_t$, we consider non-constant solutions to the following equation satisfying the following finite energy condition:
\begin{equation}\label{eqn:Floer}
 u:\RR_s\times S^1_t\to \widehat{W}, \quad \partial_su+J_t(\partial_tu-X_{H_t})=0,\quad 0<E(u):=\frac{1}{2}\int_{\RR\times S^1} |\partial_su|^2+|\partial_tu-X_{H_t}|^2 \rd s \rd t<\infty.
\end{equation}
Then we have the following properties for $u$.
\begin{enumerate}
    \item\label{up1} When $E(u)$ is finite, one can see that $\displaystyle\lim_{s\to \pm \infty} u$ are periodic orbits of $H_t$ via the same argument as that in \cite[Proposition 1.21]{floer}. 
    The only difference compared to the manifold case is that the limit orbit is an element of the orbifold of free loops, and it could be a singular element.
    Since periodic orbits of $H_t$ are non-degenerate, we have
    $$|\partial_s u(s,t)|<ce^{-\delta |s|},$$
    for some $\delta,c>0$, see \cite[Proposition 1.21]{floer}. 
    Moreover, $\delta$ can be chosen uniformly for any finite energy $u$ if there are only finitely many periodic orbits.
    \item\label{up2} Let $\displaystyle x:=\lim_{s\to \infty} u $ and $ y:= \displaystyle\lim_{s\to -\infty} u$.
    Then, $E(u)=\cA_{H_t}(y)-\cA_{H_t}(x)$.
    \item By elliptic regularity, $u$ is necessarily $C^\infty$.
    \item Since $u$ is defined on the manifold $\RR \times S^1$ and $u$ is non-constant, by \Cref{thm:map_cylinder} and \Cref{cor:smooth}, there is a smooth neighborhood of $u$ in $W^{k,p,\delta}(\RR\times S^1, \widehat{W},x,y)$, where $x$ and $y$ are as in \Cref{up2}, which is modeled on a small ball of $0$ in the Banach space $W^{k,p,\delta}(u^*T\widehat{W})$ (this is an honest vector bundle, instead of an orbifold bundle, as the base is a manifold, namely the cylinder). 
    In particular, the linearized operator $D_u:W^{k,p,\delta}(u^*T\widehat{W})\to W^{k-1,p,\delta}(u^*T\widehat{W})$ has the same form as in the manifold case \cite[(13)]{floer}, which is Fredholm by \cite[Theorem 2.2]{floer}\footnote{\cite[Theorem 2.2]{floer} states this for $W^{k,p}$ maps. In the weighted case, as long as $\delta$ is smaller than the gap between eigenvalues of the asymptotic operator and $0$, the same Fredholm property with the same index holds by change of coordinates.}. Alternatively, as $x\ne y$, by \Cref{cor:smooth} and \Cref{rmk:smooth}, there is no point with non-trivial isotropy in the space of maps from $\RR\times S^1$ to $W$ with asymptotics $x$ and $y$. 
    In particular,  $W^{k,p,\delta}(\RR\times S^1, \widehat{W},x,y)$ is a Banach manifold by the same construction as in the manifold case without using \Cref{thm:morphism_orbifold}. 
    Hence, we can also define $W^{k,p}(\RR\times S^1,W,x,y)$ as a Banach manifold\footnote{In the case of \Cref{thm:morphism_orbifold} for $\RR\times S^1$ considered in \S \ref{ss:non-compact}, we need $\delta>0$ to prove the properness of the Banach orbifold using some auxiliary compact Sobolev embeddings, see the proof of \Cref{prop:proper}. Hence, we are not able to make sense of $W^{k,p}(\RR\times S^1,W,x,y)$ as a Banach orbifold in general.} and proceed as in \cite{floer}
    
    So far, there is no difference compared to the manifold cases, and our moduli spaces of curves will be zero sets of Fredholm sections on Banach \emph{manifolds}. 
    Indeed, as the only orbifold points of the Banach manifolds of maps from $\RR\times S^1$ to $\widehat{W}$ with asymptotic conditions are those constant curves with image contained in singularities, see \Cref{cor:smooth}, we will not see them by our assumption on $u$ being non-constant. We will need to deal with them for cases considered in \S \ref{SS:fail}, where the Fredholm property holds equivariantly w.r.t.\ the local symmetry on the space of maps, but transversality may fail.
    
    \item  When we trivialize $u^*T\widehat{W}$ as a symplectic vector bundle, we get two Conley-Zehnder indices, namely $\mu_{\CZ}(x)$ and $\mu_{\CZ}(y)$, using the induced trivialization on $x^*T\widehat{W},y^*T\widehat{W}$. Then the index of $D_{u}$ is given by $\mu_{\CZ}(x)-\mu_{\CZ}(y)$. 
    \item $u$ is somewhere injective by \cite[Theorem 4.3]{transverse} (c.f.\ also \cite{AudDamBook} for a more modern reference). 
\end{enumerate}
For $x,y$ in the set $\cP(H_t)$ of $1-$periodic orbits of $H_t$, we define the following moduli space:
\begin{equation}\label{eqn:moduli_I}
    M_{x,y}\coloneqq \widetilde{M}_{x,y}/\RR, 
    \quad
    \widetilde{M}_{x,y}:=\left\{u:\RR\times S^1 \to \widehat{W}\left| \partial_su+J_t(\partial_tu-X_{H_t})=0, \displaystyle \lim_{s\to \infty} u =x, \lim_{s\to -\infty} u = y        \right.\right\}.
\end{equation}
Here, $\widetilde{M}_{x,y}$ is the zero set of a Fredholm section on a Banach orbifold (in fact, a Banach manifold by the discussion above). Therefore, there is no difference compared to the construction of moduli spaces in the usual Hamiltonian-Floer theory for symplectic manifolds. The only difference is that generators could be orbifold points in the space of loops.

\subsubsection{Type (I) Hamiltonians.}
We also consider the same solutions to \eqref{eqn:Floer} for a type (I) Hamiltonian $H_t$. Note that there are three types of periodic orbits for $H_t$, namely non-constant orbits of $H_t$ on $\widehat{W}\setminus W$, a family of constant orbits parametrized by $W$, and isolated constant orbits $\wedge W \setminus W$, that is the discrete set of non-trivial conjugacy classes from all singularities. This will be referred to as nontrivial \emph{twisted sectors} (following the terminology in \cite{CheRua04}). 

Since the symplectic action $\cA_H$ of constant orbits is zero and that of non-constant orbits is negative, given a finite energy non-constant Floer cylinder $u$, then $\displaystyle \lim_{s\to+ \infty} u$ must be a non-constant orbit as $\cA_H(\displaystyle \lim_{s\to+ \infty} u)<\cA_H(\displaystyle \lim_{s\to- \infty} u)$. 
In the following we separate the discussion into the three possible cases for the negative asymptotic orbit.

\textbf{Case 1:} $\displaystyle\lim_{s\to -\infty} u$ is a non-constant orbit. This is identical to the type (II) Hamiltonian case. 

\textbf{Case 2:} $\displaystyle\lim_{s\to -\infty} u$ is a constant orbit $y$ from $\wedge W \backslash W$. 
This is also similar to the type (II) Hamiltonian case, as we will explain now. 
Note that even though our Hamiltonian $H$ is $0$ on $W$, $y$ looks ``non-degenerate". In fact, $y$ is a non-degenerate critical point of the action functional on the orbifold of loops. We will establish directly the consequences of non-degeneracy, which leads to a Fredholm setup. There exists $C\in \RR$ such that $u(s,t)\in W$ for $s<C$. 
Now, assume that $y=(p,(g))$, where $p$ is a singular point of $W$ and $(g)$ is a nontrivial conjugacy class of the isotropy group at $p$. 
Since $\displaystyle\lim_{s\to -\infty} u=y$, we may assume for $s<C$ that $u(s,t)$ is contained in a neighborhood of $p$ modeled by a local uniformizer $G\ltimes \CC^{n}$. Then by Proposition \ref{prop:global}, we know that $u(s,t)$ is represented by a map from $\widetilde{u}:= (-\infty,C)\times [0,1]\to \CC^{n}$ with  $\widetilde{u}(s,0)=g\cdot \widetilde{u}(s,1)$ and $\displaystyle\lim_{s\to -\infty} \widetilde{u}=0$. 
If we holomorphically identify  $(-\infty,C)\times [0,1]$ with the sector $\{0<r<1, 0\le \theta\le \frac{2\pi}{o(g)} \}$, where $o(g)$ is the order of $g$, we get a canonical functor $\widehat{u}:\ZZ/o(g)\ZZ\ltimes \DD^*\to G\ltimes \CC^{n}$, which is also holomorphic. 
Then the classical removal of singularity \cite[Theorem 4.1.2]{MS12} on $\DD^*$ implies that the extension by $\widehat{u}(0)=0$ is holomorphic and is apparently a functor $\widehat{u}\colon\ZZ/o(g)\ZZ\ltimes \DD\to G\ltimes \CC^{n}$. 
This has two consequences: 
\begin{enumerate}
    \item we can view $u$ as a map from $\CC$ to $\widehat{W}$ but $0\in \CC$ is an orbifold point with order $o(g)$;
    \item we have $|\partial_su(s,t)|<Ce^{-\frac{2\pi}{o(g)}|s|}$ for $s\ll 0$.\footnote{This follows from the inequality $|\partial_{s'}\widehat{u}(s',t')|<Ce^{-2\pi|s'|}$ and the equalities $t'=o(g)t \mod 1$ and $s'=o(g)s$.} 
\end{enumerate}
In view of these, we define 
$$M_{x,y}:=\left\{ u:\RR\times S^1 \to \widehat{W}\left| \partial_su+J_t(\partial_tu-X_{H_t})=0, \displaystyle \lim_{s\to \infty} u =x, \lim_{s\to -\infty} u = y\right. \right\}/\RR$$
as the (modulo by the $\RR-$action of the) zero set of the Floer equation on the spaces of Sobolev maps $W^{k,p}(\RR\times S^1,\widehat{W},x,y)$ with the above asymptotics and an exponential decay near the negative end with weight $\delta<\frac{2\pi}{o(g)}$.
The exponential decay is crucial as $y$ is a degenerate isolated orbit. 
In particular, the linearized operator is again Fredholm.

\textbf{Case 3:} $\displaystyle\lim_{s\to -\infty} u$ is neither a non-constant orbit nor a constant orbit from $\wedge W \backslash W$.  
The remaining constant orbits parametrized by $W$ are only Morse-Bott non-degenerate on the interior of $W$ 
and fail to be Morse-Bott along $\partial W$.\footnote{As the asymptotic operator is zero on the normal direction of $\partial W$ in $W$ (viewed as constant sections).}
In view of this,   $\displaystyle\lim_{s\to-\infty} u$ may not exist a priori, but the limit set of $u$ for $s\to -\infty$ is contained in $W$. 

We would then like to apply the removal of singularity to view it as a curve defined on $\CC$. 
Now, the problem is that $u(s,\cdot)$ for $s \ll 0$ may, in general, not be contained in $W$ where $H_t=0$. 
However, since the actual moduli space we count will be a fiber product with a stable orbifold of $\nabla f$, for an admissible Morse function $f$ which is strictly contained in the interior of $W$, said problem will in fact not occur in our setting.

More precisely, we first consider the following moduli space. 
Let $W_{\delta}$ be the complement in $W$ of a small tubular neighborhood of $\partial W$ where $\nabla f\ne 0$. 
For $x$ in the set $\cP^*(H_t)$ of non-constant $1-$periodic orbits of $H_t$, we define
\begin{equation}\label{eqn:B}
    B_x:=\left\{ u:\RR\times S^1\to \widehat{W}\left|\begin{array}{l}
         \partial_su+J_t(\partial_tu-X_{H_t})=0,\displaystyle\lim_{s\to \infty} u=x, \text{ and}\\
         \text{the limit set of $u$ for $s\to -\infty$ is contained in $W_\delta$.}
    \end{array}      \right. \right\}/\RR
\end{equation}
Then by the removal of singularity\footnote{
To be more precise, following the proof of \cite[Theorem 4.1.2]{MS12}, we have that $u$ can be $C^0$ extended on $\CC$ as a map from $\CC$ to $W$. 
Then, depending on whether $u(0)$ is a singularity, we can apply the usual removal of singularity or pass to a local uniformizer as before. 
}, $B_x$ has the following equivalent description
\begin{equation}\label{eqn:B'}
    B_x:=\left\{ u:\CC \to \widehat{W}\left|
         \partial_su+J_t(\partial_tu-X_{H_t})=0,\displaystyle\lim_{s\to \infty} u =x, u(0)\in W_{\delta}  \right. \right\}/\RR
\end{equation}
where $z=e^{2\pi(s+it)}$ are the polar coordinates on $\CC$ and the equation degenerates to the Cauchy-Riemann equation for $s\ll 0$ since $u(0)\in W_{\delta}$. 
Other analytical properties (Fredholm property, etc.) for elements in those moduli spaces are identical to the type (II) case, with the only difference that the expected dimension of $B_x$ is $\mu_{\CZ}(x)+n-1$\footnote{Since $x$ is contractible in $W$, $\mu_{\CZ}(x)$ is canonically defined.}, which follows easily from \cite[Definition 1.4.3, Remark 1.4.4]{viterbo}.

Lastly, note that $B_x$ is equipped with an orbifold evaluation map $B_x\to W, u\mapsto u(0)$. Now, let $y$ be a critical point of $f$ and $W^s(y)$ denote the stable orbifold of $y$. Then, we define $M_{x,y}$ to be the orbifold fiber product 
$$M_{x,y}:=B_x\times_W W^s(y).$$
\begin{remark}\label{rmk:unstable}
More formally speaking, the stable orbifold $W^s(y)$ of a critical point $y$ is the space of gradient flowlines $\gamma:[0,\infty)\to W$ of $f$ such that $\displaystyle\lim_{s\to\infty} \gamma=y$. Then $W^s(y)$ is equipped with an evaluation map to $W$ by $\gamma\mapsto \gamma(0)$. In the orbifold case considered here, the gradient flowline $\gamma$ can be an orbifold point of the space of maps from $[0,\infty)$ to $W$. For example, let $y$ be a singularity, which is a local minimum of an admissible Morse function $f$, then $W^s(y)$ is the orbifold $\bullet/\stab_y$ represented by the constant flow. 
\end{remark}

\begin{remark}
Our point of view is to always work with maps from a manifold to the target orbifold.
The advantage is that the Fredholm problem is identical to the manifold case. 
Moreover, since the singularities are isolated in our setting, any non-constant Floer cylinder is a smooth point in the orbifold of maps from the cylinder, i.e.\ the Fredholm problem is not equivariant. 
In particular, all of the moduli spaces in the definition of symplectic cohomology will be manifolds instead of orbifolds. 
This turns out to be crucial for the geometric applications in \S \ref{S5}. 
However, having manifolds as moduli spaces is not in general possible; see \S \ref{SS:fail}.
Moreover, to avoid discussing Fredholm theory on orbifold, one could also define orbifold Gromov-Witten invariants\footnote{Indeed, the orbifold index formula \cite[Proposition 4.2.2]{CheRua04}, originally from the de-singularization of the pullback orbifold bundle \cite[\S 4.2]{CheRua04},  can be reformulated using punctured surfaces. On the other hand, the general analytical theory of Cauchy-Riemann operators on orbibundles seems to be absent. } from finite energy holomorphic curves from a smooth punctured Riemann surface. 
The orbifold points on the domain are only necessary if we wish to compactify the curve as in Case 2, 3 above. 
\end{remark}

\begin{example}\label{ex:hand}
Let $E(a_1,\ldots,a_n)$ denote the ellipsoid $\frac{|z_1|^2}{a_1^2}+\ldots+\frac{|z_n|^2}{a_n^2}\le 1$ with $0<a_1<\ldots<a_n$. 
Since the $\ZZ/2\ZZ$ action by the antipodal map on $\CC^n$ preserves the standard Liouville form on $E(a_1,\ldots,a_n)$, we get an exact orbifold filling $E(a_1,\ldots,a_n)/(\ZZ/2\ZZ)$ of $(\mathbb{RP}^{2n-1},\xi_{\std})$. 
Let $\gamma_0$ denote the contractible Reeb orbit contained in the $z_1$-plane that is parametrized as starting from  $(a_1,0,\ldots,0)$.
Following the argument in \cite[Proposition 2.12]{Zho}, the following moduli space (we consider the SFT analogue without involving a Hamiltonian) is compact with algebraic count $2$ for $q$ a smooth point near the boundary:
$$\cN_{\gamma_0,q}:=\left\{u:\CC\to \widehat{E}(a_1,\ldots,a_n)/(\ZZ/2\ZZ)\left|\overline{\partial}_J u=0, u(0)=q, \lim_{s\to \infty} u = \gamma_0  \right.\right\}/\RR.$$
On the other hand, the algebraic count of the analogously defined $\cN_{\gamma_0,0}$ is $1$ (and in fact, the geometric count is $1$, given by the $z_1$-plane).
More formally, using Proposition \ref{prop:global}, an orbifold map $u\in \cN_{\gamma_0,0}$ can be viewed as a holomorphic map $\tilde{u}:\CC\to \widehat{E}(a_1,\ldots,a_n)$ such that $\tilde{u}(0)=0$ and $\lim_{s\to \infty}\tilde{u}=\tilde{\gamma}_0$, where $\tilde{\gamma}_0$ is one of the two lifts of $\gamma_0$ as parametrized Reeb orbits, i.e.\ the shortest Reeb orbit on  $E(a_1,\ldots,a_n)$  parametrized as starting from $(a_1,0,\ldots,0)$ or $(-a_1,0,\ldots,0)$. 
Then the algebraic count of all the lifts is $2$, one from each $\tilde{\gamma}_0$. 
But the two lifts give the same orbifold map in the quotient, so we have $\#\cN_{\gamma_0,0}=1$. 
This is a feature instead of a problem of the holomorphic curves in orbifolds, as $\cN_{\gamma_0,0}$ is not cobordant to $\cN_{\gamma_0,q}$ (they are both manifolds as none of the curves are contained in a singularity), but rather $\cN_{\gamma_0,0}\times_{\{0\}/(\ZZ/2\ZZ)}\{0\}$, which is two points without isotropy, is cobordant to $\cN_{\gamma_0,q}$. 
The appearance of the orbifold fiber product is natural as the cobordism is built from the moduli space with a constraint from a path connecting $0$ and $q$ (strictly speaking, a path in the inertia orbifold), which is a fiber product over the inertia orbifold. We will further explore this phenomenon in \Cref{prop:restriction} and \S \ref{S4}. Another worth-noting point is that even though $\cN_{\gamma_0,0}$ is defined differently from $M_{\gamma_0,0}$ for a type (I) Hamiltonian, they are the same moduli space (up to the difference between SFT and Hamiltonian setup). 
Indeed, by definition we have $M_{\gamma_0,0}=\cN_{\gamma_0,0}\times_{\{0\}/(\ZZ/2\ZZ)} W^s(0)$; hence, since $W^s(0)$ is an orbifold $\bullet/(\ZZ/2\ZZ)$ whose evaluation map to $\{0\}/(\ZZ/2\ZZ)$ is the isomorphism, we get $M_{\gamma_0,0}=\cN_{\gamma_0,0}$. \END
\end{example}

\subsection{Conley-Zehnder indices and grading.}\label{SS:grading}

The discussion of the Conley-Zehnder indices and the grading is identical to the smooth case.
We will grade our symplectic cohomology by $n-\mu_{\CZ}(x)$, where $n$ is the complex dimension of the symplectic manifold.
We list their properties for completeness and discuss the rational grading as well as the relation with  $\age(g)$ from \eqref{eqn:age}. 
Given a non-degenerate Hamiltonian orbit $x$ and a trivialization $\Psi$ of the complex line bundle $x^*\det_{\CC}T\widehat{W}$, we have a Conley-Zehnder index $\mu^{\Psi}_{\CZ}(x)\in \ZZ$. 
In general, the Conley-Zehnder index is well-defined (independent of the trivialization) in $\ZZ/2\ZZ$.
In particular, the symplectic cohomology can always be graded over $\ZZ/2\ZZ$.

\subsubsection{Integral grading.}
By the discussion in \S\ref{sec:complex_line_bundles}, $\det_{\CC}TW$ can be trivialized if and only if the integral first Chern class $c^{\ZZ}_1(\det_{\CC}TW)=c^{\ZZ}_1(W)\in H^2(BW;\ZZ)$ vanishes. 
For example, when $G\subset SU(n)$, the orbifold $\CC^n/G$ has vanishing integral first Chern class, as $\det_{\CC}TW$ can be trivialized with a $G$-invariant multi-vector field. 
By fixing such a trivialization $\Psi$ of $\det_{\CC}TW$, we can also assign to every non-degenerate Hamiltonian orbit $x$ a $\ZZ-$valued Conley-Zehnder index. 
Now, the difference between two trivializations $\Psi,\Psi'$ is measured by a homotopy class $\alpha\in [W:S^1=U(1)]$ \footnote{The set of homotopy classes $[W:S^1]$ can either be understood in the orbifold sense, or in the topological sense (which we previously denoted $[W,S^1]_{top}$) for the coarse space $W$ (i.e.\ the space of orbits), as in fact they coincide. Indeed, the discussions in \S \ref{ss:morphism} show that the orbifold structure on the space of maps are solely inherited from the orbifold structure on the target. Therefore we have $[W:S^1]=H^1(W;\ZZ)$, which coincides with $H^1_{\orb}(W;\ZZ)$ (by a Mayer-Vietoris/sheaf type argument starting with local uniformizers, for which $H^{1}_{\orb}(BG;\ZZ)=0$ for $G<U(n)$ finite gives the equality). However, in general, $[W;K(\ZZ,n)]\ne H^n_\orb(W)$. For example, let $W=\{pt\}/(\ZZ/2)$, then $[W:\mathbb{CP}^{\infty}]=H^2(W;\ZZ)=0$ but $H^2_{\orb}(W;\ZZ)=H^2(B\ZZ/2)=\ZZ/2$.}, which gives a difference of Conley-Zehnder indices $\mu^{\Psi}_{\CZ}(x)$ and $\mu^{\Psi'}_{\CZ}(x)$ equal to $2\deg(\alpha\circ x)$, where $\alpha\circ x$ is seen as a homotopy class of maps from $S^1$ to $S^1$. 
In particular, the Conley-Zehnder indices for orbits with torsion homotopy class are well-defined independently of the trivialization. 
For the other orbits, the $\ZZ$-valued Conley-Zehnder index depends on the choice of trivialization. 
However, we will often suppress the superscript $\Psi$ for simplicity.

\begin{example}\label{ex:SU}
Let $G\subset SU(n)$, such that $\widehat{W}=\CC^n/G$ has an \emph{isolated} singularity. Note that $\age(g)$ is an integer as $\det g =1$ for $g\in G$.
We consider the non-degenerate Hamiltonian $H=\epsilon r^2$ for $0<\epsilon\ll 1$ and the constant orbit $x=(0,(g))$, where $g\ne \Id\in G$ and $(g)$ is the conjugacy class of $g$ and $r$ is the Euclidean distance to $0$.
Up to conjugation, we may assume $g=\diag(e^{a_1\i},\ldots, e^{a_n\i})$ for $0< a_i<2\pi$.
Then, we can trivialize $x^*TW$ by 
$$\Psi':S^1\times \CC^n \to x^*TW, (t,v)\mapsto (t,\diag(e^{ta_1\i},\ldots, e^{ta_n\i})v).$$ 
With such trivialization $\Psi'$, the linearized Hamiltonian diffeomorphism is given by $\diag(e^{t(a_1+\epsilon)\i},\ldots, e^{t(a_n+\epsilon)\i})$, and the resulting Conley-Zehnder index $\mu_{\CZ}^{\Psi'}(x)$ is $n$ as $0<a_i+\epsilon<2\pi$,\cite[\S 2.4]{floer}.
However, the considered trivialization $\Psi'$ on $x^*\det_{\CC} TW$ is not induced from a global trivialization $\Psi$ of $\det_{\CC}TW$. 
In fact, the difference between $\Psi'$ and $\Psi$  is given by $S^1\to U(1), t\mapsto e^{t\sum a_i \i}$. 
Hence, we have $\mu^{\Psi}_{\CZ}(x)=n-2\age(g)$. 
In particular, when we grade symplectic cohomology by $n-\mu^{\Psi}_{\CZ}(x)$, the grading of $x=(0,(g))$ is $2\age(x)$, which is precisely the degree shift in $H^*_{\CR}(\CC^n/G)$. 

Now assume the representation $G\subset SU(n)\ltimes \CC^{n}$ has a decomposition into $\CC^{m}\oplus \CC^{k}$, and consider $H=\epsilon(|z|_1^2-|z|^2_2)$ for coordinates $z_1,z_2$ on $\CC^{m},\CC^{k}$ respectively.
Repeating the argument above, we have that the linearized Hamiltonian diffeomorphism is given by $\diag(e^{t(a_1+\epsilon)\i},\ldots, e^{t(a_{m}+\epsilon)\i},e^{t(b_1-\epsilon)\i},\ldots, e^{t(b_{k}-\epsilon)\i} )$, with a similar trivialization as above.
Therefore the grading of $(0,(g))$ is again $2\age(g)$ for $g\ne \Id$, if we use a global trivialization. 
On the other hand, the grading of $(0,(\Id))$ is $2k$, i.e.\ the Morse index.
Such phenomenon is expected, because, as long as $g\ne \Id$, the critical point $(0,(g))$ contributes a class in the Chen-Ruan cohomology, seen as the zero-action part of symplectic cohomology, in degree $2\age(g)$ no matter what the Hamiltonian we use is.
\end{example}

\subsubsection{Rational grading.}
More generally, if we only have $c_1^{\QQ}(TW)=0$\footnote{The rational first Chern class $c_1^{\QQ}(TW)$ is the image of integral first Chern class $c_1(TW)\in H^2_{\orb}(W;\ZZ)$ from \S \ref{sec:complex_line_bundles} in $H^2_{\orb}(W;\QQ)=H^2(W;\QQ)$.}, i.e.\ there exists $N\in \NN$, such that $c_1^{\ZZ}((\det_{\CC}TW)^{\otimes N})=0$, then one can define a $\QQ$-valued Conley-Zehnder index following \cite[\S 4]{McLean16}.
In the following, we recall the definition. 

Assume $c_1^{\ZZ}((\det_{\CC}TW)^{\otimes N})=0$; then, $(\det_{\CC} TW)^{\otimes N}=\det_{\CC}(\oplus^N TW)$ is trivial. 
We fix a trivialization $\Psi$ of $(\det_{\CC} TW)^{\otimes N}$, and, for every Hamiltonian orbit $x$, we define $\mu^{\Psi}_{\CZ}(x):=\frac{\mu^{\Psi}_{\CZ}(\oplus^N \rd \phi_t)}{N}$, where $\rd \phi_t$ is the linearized Hamiltonian flow on $x^*TW$. 
Now given any trivialization $\Phi$ of $x^*\det_{\CC}TW$, let $A\in H^1(S^1)$ be the difference between $\Phi^{\otimes N}$ and $\Psi$. 
Then, we have
$$N\mu^{\Phi}_{\CZ}(x)=N\mu^{\Psi}_{\CZ}(x)+2A.$$
As a consequence of the definition, using $n-\mu_{\CZ}^{\Psi}(x)$ as grading, the differential, which will be defined in \S\ref{sec:sympl_cohom}, also has degree $1$. 
In particular, we will have a rational grading on the symplectic cohomology. 

The relation between the rational grading and the $\ZZ/2$ grading is not easy to state. 
One special case is that when $N$ is odd, then the parity of $\mu^{\Phi}_{\CZ}(x)$ is the same as the parity of the numerator of $\mu_{\CZ}^{\Psi}(x)$.

Lastly, we point out that, in general, the rational Conley-Zehnder index $\mu^{\Psi}_{\CZ}(x)$ depends on $N$ and $\Psi$. 
This being said, we have the following for orbits with torsion homotopy class.
\begin{prop}
    If $c_1^{\QQ}(W)=0$, the $\QQ$-valued Conley-Zehnder index is well-defined for orbits with torsion homotopy class.
\end{prop}
\begin{proof}
    For a fixed $N$ such that $c_1((\det_{\CC}TW)^{\otimes N})=0$, $\mu_{\CZ}^{\Psi}(x)$ does not depend on the trivialization $\Psi$ since $x$ has a torsion homotopy class. 
    Moreover, considering the trivialization $\Psi^{\otimes M}$ of $(\det_{\CC}TW)^{\otimes NM}$, then  $\mu_{\CZ}^{\Psi}(x)=\mu_{\CZ}^{\Psi^{\otimes M}}(x)$. 
    As a consequence, the $\QQ$-valued Conley-Zehnder index is independent of $N$.
\end{proof}

\begin{example}\label{ex:U}
Let $G\subset U(n)$, such that $\CC^n/G$ is an isolated singularity. 
Then, $c^{\QQ}_1(\CC^n/G)=0$ and $\pi^{\orb}_1(\CC^n/G)$ is torsion.
As a consequence, there is a well-defined $\QQ$-grading. 
Let now $H=\epsilon r^2$ for $\epsilon\ll 1$ as before.
Then, by the same argument in \Cref{ex:SU}, the grading of $(0,(g))$ is $2\age(g)$. 
On the other hand, the rational Conley-Zehnder index here is precisely the one used in \cite{McLean16}. 
Moreover, the Reeb dynamics and the associated Conley-Zehnder indices on $(S^{2n-1}/G,\xi_{std})$ have already been discussed in \cite{McKay}, and will be recalled in \S \ref{SS:Reeb}. 
Using this, by \cite[Theorem 1.1]{McLean16}, $\CC^n/G$ is terminal \cite[Definition 2.8]{birational} iff $\min\{\mu_{\CZ}(x)+n-3\}>0$ iff $\min_{g\in G}\{2n-2\age(g)-2\}>0$, and $\CC^n/G$ is canonical \cite[Definition 2.8]{birational} iff $\min_{g\in G}\{2n-2\age(g)-2\}\ge 0$.
\end{example}

\subsection{Transversality, compactness, gluing and orientations.} 
\subsubsection{Transversality.}
In our setting, it turns out that there is no difference in the treatment of transversality for uncompactified moduli spaces with respect to the manifold case, since our curves are somewhere injective:
\begin{prop}\label{prop:transverse}
    For generic choices of admissible almost complex structure $J$, $M_{x,y}$ is a manifold of the expected dimension, which is $|y|-|x|$ when $c_1^{\QQ}=0$.
\end{prop}
\begin{proof}
    For type (II) Hamiltonians, since our curves are somewhere injective, the statement follows from the usual universal moduli space argument, as there is no curve in the differential resting on the orbifold singularities, i.e.\ constant holomorphic curves with image contained in a singularity and are asymptotic to (possibly non-trivial) twisted sectors. For type (I) Hamiltonians, when $y$ is critical point of $f$, the transversality for $M_{x,y}$ follows from  \cite[Proposition 3.4.2, Lemma 3.4.3]{MS12} implying that the evaluation map $B_x\to W, u\mapsto u(0)$ is transverse to the stable orbifolds of $\nabla f$. 
    The other cases for type (I) Hamiltonians are similar to the type (II) case.  
    Notice that, strictly speaking, to use the Sard-Smale theorem in the universal moduli space argument, one can only work with $C^k$ almost complex structures.
    However, using the argument of Taubes \cite[Theorem 5.1]{transverse}, one can boost the generic almost complex structure to be smooth. 
    A guide towards how to obtain a complete proof of transversality for generic smooth almost complex structures can be found in \cite{smooth}.
\end{proof}

\subsubsection{Compactness and gluing.}\label{ss:gluing}
Let's start by discussing compactness.
First, by the integrated maximum principle  \cite[Lemma 2.2]{cieliebak18}, the curves in all the moduli spaces we consider stay in a bounded subset of $\widehat{W}$. 
Solutions to \eqref{eqn:Floer} have then uniform $C^1$ bound, as, otherwise, the usual sphere bubble argument \cite[\S 6.6]{AudDamBook} and removal of singularity for finite energy curve as in Case 2 of the setup for type (I) Hamiltonians would give a non-constant holomorphic map from the sphere possibly with one orbifold point, contradicting exactness. 
In particular, the space of solutions to \eqref{eqn:Floer} is compact in the $C_{\rm{loc}}^{\infty}$ topology.

Next, when we consider the quotient space $M_{x,y}$, the usual breaking phenomena can happen. 
More explicitly, in the type (II) case, given a sequence $u_n\in M_{x,y}$, by the same argument of \cite[Theorem 9.1.7]{AudDamBook}, one can find
\begin{enumerate}
    \item a subsequence of $u_n$, still denoted by $u_n$ for simplicity;
    \item periodic orbits $x_0=x,x_1,\ldots,x_{l+1}=y$;
    \item $s_n^k\in \RR$ with $\displaystyle \lim_{n\to \infty }(s^{k}_n-s^{k+1}_n)=\infty$, $0\le k\le l-1$;
    \item $u^k\in M_{x_k,x_{k+1}}$;
\end{enumerate}
such that for every $0\le k \le l$, we have $u_n(\cdot-s_n^k,\cdot)$ converges to $u^k$ in $C^{\infty}_{\rm{loc}}$ topology. 
With \cite[Proposition 9.1.2]{AudDamBook}, this yields that $\bigsqcup_{l,x_i} M_{x,x_1}\times \ldots \times M_{x_l,y}$ is a compact Hausdorff space just like in \cite[Corollary 9.1.9]{AudDamBook}. 
However, this is not the right compactification to work with, as it cannot be endowed with a smooth structure by the gluing map. 
In short, the problem is that, in general, a broken curve is not just a tuple of curves;
more precisely, the issue appears when one of $x_i,1\le i \le l$ is an orbit with isotropy. 

In the cases considered in this paper, this situation arises only when one of the $x_i$, with $1\le i \le l$, is a constant orbit of the form $(p,(\Id))$ for an orbifold singularity $p$ of $W$. 

To see this, we may assume that $H$ attains its minimum on a singularity. 
If there is a breaking at the orbit $(p,(g))$ for $g\ne \Id$, the negative asymptotic is necessarily a non-singular critical point $q$ of $H$, as only those constant orbits have larger symplectic action. Then the Floer cylinder between $(p,(g))$ and $q$ is a Morse trajectory for the $C^2$ small $H$ \cite[Theorem 7.3(1)]{MR1181727}, which cannot have $(p,(g))$ as the asymptotic orbit, thus giving a contradiction. We hence now discuss more in detail what the correct compactification should be.

As this is the only problematic case in our setting, let's assume that one of the intermediate orbits, say $x_1$ without loss of generality, is $(p,(\Id))$. 
By construction, for any $\epsilon>0$, there exist $A,N\gg 0$, such that the energy of
$u_n$ restricted to $[A+s^1_n, -A+s^0_n]\times S^1$ 
is smaller than $\epsilon$ for $n\ge N$. 
As a consequence, by \cite[Proposition 1.21]{floer}, we can assume that $u_n|_{[A+s^1_n, -A+s^0_n]\times S^1}$ is mapped to a neighborhood of $x_1$ in the loop space. 
In particular,  we can assume $u_n([A+s^1_n, -A+s^0_n]\times S^1)$ is contained in a local uniformizer of $p$ modeled on $\CC^n/\stab_p$.  
Note that $u_n|_{[A+s^1_n, -A+s^0_n]\times S^1}$  captures the ends $u^0|_{(-\infty,-A]\times S^1}$ and $u^1|_{[A,\infty)\times S^1}$. 
Since $x_1$ is contractible in $\CC^n/\stab_p$ according to \Cref{cor:pi0_free_loop_space}, the (non-constant) orbifold map $u_n:[A+s^1_n, -A+s^0_n]\times S^1\to \CC^n/\stab_p$ admits $\stab_p$ different ways to be written as maps $[A+s^1_n, -A+s^0_n]\times S^1\to \CC^n$. 
Moreover, possibly after passing to a subsequence, we get a breaking of those lifts in $\CC^n$. 

Guided by this discussion, we then define \textbf{broken Floer trajectory} as a tuple of connecting trajectories $(u^0,\ldots,u^k)$ as in (1)$-$(4) above, with the additional specification, if one breaking happens at an orbit in the form of $(p,(\Id))$, of an equivalence class of breaking of the lifted curves in the local uniformizer $\CC^n$ of $p$. 
Here, two broken curves in $\CC^n$ are considered equivalent if and only if one of them is the image of the other under the $\stab_p$ action on $\CC^n$. 
For example, if $\displaystyle\lim_{s\to -\infty}u^0=(p,(\Id))=\lim_{s\to \infty}u^1$, then there are $|\stab_p|$ broken Floer trajectories corresponding to the tuple $(u^0,u^1)$. 
More formally, the space of broken Floer trajectories with components from $M_{x_0,x_1},\ldots,M_{x_k,x_{k+1}}$ is given by the fiber product of orbifolds 
\[
M_{x_0,x_1}\times_{x_1}\ldots \times_{x_k}M_{x_k,x_{k+1}}.
\]
Then, similar to the manifold case, we have the following.
\begin{prop}\label{prop:compact}
    For type (II) Hamiltonians, we have
    \begin{equation}\label{eqn:compact}
        \cM_{x,y}:=
        \bigsqcup_{
        {\footnotesize
        \begin{array}{c}
             k\in \NN \\
             z_1,\ldots,z_k
        \end{array}
        } 
        }
        M_{x,z_1}\times_{z_1} \ldots \times_{z_k} M_{z_k,y} 
    \end{equation}
    is compact and Hausdorff in the Gromov-Floer topology.
\end{prop}
The advantage of working with such broken Floer trajectories is that we can write out a pregluing map, as the local broken curve can be viewed as contained in a local chart $\CC^n$. 
Then we can deduce a gluing map when each component is cut out transversely as in the manifold case.
In particular, we have the following.
\begin{prop}\label{prop:glue}
    For the almost complex structures in \Cref{prop:transverse} and type (II) Hamiltonians, $\cM_{x,y}$ is a  manifold with boundary $\sqcup_z \, \cM_{x,z}\times_z \cM_{z,y}$ whenever the expected dimension is at most one.
\end{prop}

\begin{remark}
Although it will not appear in the cases considered in this paper, we point out that, if the breaking happens at $(p,(g))$ for $g\ne \Id$, then the orbifold map $u_n:[A+s^1_n, -A+s^0_n]\times S^1\to \CC^n/\stab_p$ only admits lifts as equivariant map $\tilde{u}_n:[\frac{A+s^1_n}{o(g)}, \frac{-A+s^0_n}{o(g)}]\times S^1\to \CC^n$  w.r.t.\ the group homomorphism $\ZZ/o(g)\ZZ\to \stab_p$, where $\ZZ/o(g)\ZZ$ acts on $S^1$ by the rotation by $2\pi/o(g)$\footnote{To explain the factor $1/o(g)$ appearing in the $s-$coordinate, $u_n$ is viewed a $\tilde{u}_n$ restricted to $[\frac{A+s^1_n}{o(g)}, \frac{-A+s^0_n}{o(g)}]\times [0,\frac{2\pi}{o(g)}]\to  \CC$ with  $\tilde{u}_n(\cdot, \frac{1}{o(g)})=g\cdot\tilde{u}_n(\cdot,0)$. Then $[\frac{A+s^1_n}{o(g)}, \frac{-A+s^0_n}{o(g)}]\times [0,\frac{1}{o(g)}]$ is biholomorphic to $[A+s^1_n,-A+s^0_n]\times S^1$ with the standard complex structure.}.  
The equivalence class of broken curves is again induced by the action of $\stab_p$ on $\CC^n$. 
Then, the pregluing and gluing need to be done in an equivariant way. 
\end{remark}

\begin{remark}\label{rmk:SFT}
Ambiguity of gluing due to non-trivial isotropy of the gluing orbits already appear in SFT in the smooth setting.
More precisely, if one considers moduli spaces of holomorphic curves in the symplectization of a contact manifold \emph{without asymptotic markers}, then there are $\kappa_{\gamma}$ different ways to glue two curves along a Reeb orbit $\gamma$, where $\kappa_{\gamma}$ is the multiplicity of $\gamma$. 
Note that in the case without asymptotic markers, the evaluation map from the moduli spaces of holomorphic curves takes values in the space of \emph{unparameterized} Reeb orbits, which are natural orbifolds in the form of $\bullet/(\ZZ/\kappa_{\gamma}\ZZ)$.
However, in the construction of SFT-type invariants \cite{SFT,RSFT,contact_homology}, one often considers moduli spaces with asymptotic markers subject to chosen point constraints from the image of the Reeb orbits. 
Then one gets redundancy of gluing by rotating asymptotic markers of both punctures simultaneously by $\frac{2\pi}{\kappa_{\gamma}}$.
\end{remark}

For the type (I) case, note that the unstable orbifolds in \Cref{rmk:unstable} admit a compactification by adding broken gradient flow lines (which are also fiber products instead of products) as usual.
Moreover, $B_x$ can only develop breaking along non-constant orbits, so that the compactification $\cM_{x,y}$ and the associated smooth structure with boundary follow from the analogue results in the manifold case using transverse fiber products with the stable orbifolds of $f$. 
Among all the moduli spaces we considered, the only case where $M_{x,y}$ could a priori be an orbifold is when $y$ is an orbifold singularity of $W$. 
However, in this case $M_{x,y}$ is the fiber product of a manifold together with $\bullet/
\stab_y$ over the orbifold $\{y\}/\stab_y$, i.e.\ it is again a manifold. 
In particular, we have the following.

\begin{prop}\label{prop:glueI}
    For the almost complex structures in \Cref{prop:transverse} and type (I) Hamiltonians,  $\cM_{x,y}:=\bigsqcup M_{x,z_0}\times_{z_0} \ldots \times_{z_k} M_{z_k,y}$ is a compact and Hausdorff space and  $\cM_{x,y}$ is a  manifold with boundary $\sqcup \cM_{x,z}\times_z \cM_{z,y}$ whenever the expected dimension of $\cM_{x,y}$ is at most one. 
\end{prop}

\begin{remark}
In the case of orbifold Gromov-Witten theory, we have a similar description of the sphere bubble trees. 
However, since the punctures are not parametrized as in the Floer theory case, if the asymptotic constant loop is $(p,(g))$, rotating one component of a broken curve once around has the same effect with applying the $g$ action to the component. 
As a consequence, the space of bubble trees is the fiber product over the \emph{rigidified} inertia orbifold \cite[\S 4.2]{orbi_GW}, which has the same objects as $\wedge W$, but the isotropy of an element $(p,(g))$ is $C(g)/\langle g\rangle$ instead of $C(g)$.
\end{remark}

\begin{remark}
In general, when we consider more general symplectic orbifolds, in particular those with orbifold contact boundaries, those non-constant orbits can be contained in the singular set, and are hence singular points in the loop space. 
Then, a breaking along those orbits needs to be described in a local uniformizer around the orbit in the loop space.
\end{remark}

\subsubsection{Orientations.}
Orientations for moduli spaces in Hamiltonian-Floer cohomology has been introduced in \cite{orientation}.
Since the Fredholm setup in the orbifold case is the same as the manifold case, we can adapt the same strategy with one caveat: when an orbit has isotropy, the isotropy acts on the orientation line, and
if the action does not preserve the orientations, then one cannot assign orientations to generators consistently.
Luckily, we will never run into such situations. 
We will explain this claim following the setup in \cite[\S 1.4]{viterbo}.

Given a trivialization of $x^*TW$ for $x=(0,(g))$ in a local model $\CC^n/G$, the linearized Hamiltonian flow gives a path of symplectic matrices $\Psi_t$. 
Up to a reparametrization of $\Psi_t$, we may write $\Psi_t=\exp(A_t)$ so that $\rd A_t/\rd t$ is periodic.
Then we can find a map $B\in C^{\infty}(\CC,\RR^{2n}\times \RR^{2n})$, such that $B(e^{-s-2\pi i t})=\rd A_t/\rd t$ for $s\ll 0$. 
Denoting by $I$ the standard complex structure on $\CC^n$, we define the operator
$$D_{\Psi}:W^{1,p}(\CC,\CC^n)\to L^p(\CC,\CC^n),\quad X\mapsto \partial_sX+I(\partial_t X-B\cdot X).$$

Then, $D_{\Psi}$ is a Fredholm operator when $\Psi_1$ does not have $1$ as an eigenvalue. 
The orientation line $o_x$ associated to $x$ is the determinant line bundle $\det D_{\Psi}$. 
By \cite[Proposition 1.4.10]{viterbo}, different trivializations of $x^*TW$ induce canonical isomorphisms between determinant line bundles, hence $o_x$ is well-defined independently of any such choice of a trivialization. 
For the orbit $x=(0,(g))$ of $H=\epsilon r^2$, we can write, up to conjugation, $g=\diag(e^{a_1\i},\ldots,e^{a_n\i})$. 
We can choose the trivialization in \Cref{ex:SU}, so that $\rd A_t/\rd t = \diag((a_1+\epsilon)\i,\ldots,(a_n+\epsilon)\i)$.  
We can also choose $B\equiv A$ on $\CC$; then, for $h\in C(g)$, we have the following diagram, which induces the action of $h$ on $o_x$:
\[
\xymatrix{
W^{1,p}(\CC,\CC^n) \ar[r]^{D_{\Psi}} \ar[d]^{h} & L^p(\CC,\CC^n)\ar[d]^{h}\\
W^{1,p}(\CC,\CC^n) \ar[r]^{D_{\Psi}} & L^p(\CC,\CC^n)
}
\]
Now, since $B$ is complex linear, we have that $\ker D_{\Psi},\coker D_{\Psi}$ are complex vector spaces, which gives rise to an orientation of $o_x$.
It is clear that the action of $h\in C(g)$ (induced from the action of $h$ on $\CC^n$) preserves this orientation as it induces complex isomorphisms on $\ker D_{\Psi},\coker D_{\Psi}$.
As a consequence, we have the following.
\begin{prop}\label{prop:orientation}
    The moduli spaces in \Cref{prop:glue,prop:glueI} can be coherently oriented, i.e.\ the boundary orientation is the fiber product orientation (in our case, it is the union of products over the isotropy group).
\end{prop}

\subsection{The symplectic cohomology.}\label{sec:sympl_cohom}

The Hamiltonian cochain complex is the free $R$-module generated either by orbits of $H_t$ in the case of type (II) Hamiltonians, or by non-constant orbits of $H_t$, critical points of $f$ and $\{(p,(g))\}$ for $g\ne \Id \in \stab_p$ in the case of type (I) Hamiltonians. 
Since generators can be orbifold points of the loop space, we define the differential by
\begin{equation}
\label{eqn:diff}
    \delta(x) = \sum_{y,\dim \cM_{x,y}=0} t_*\circ s^*(1)y,
\end{equation}
where $s,t$ are source and target maps to respectively $x$ and $y$: $x\xleftarrow{s}\cM_{x,y}\xrightarrow{t}y$.
Notice that we made in \eqref{eqn:diff} a little abuse of notation: $t_*\circ s^*(1)$ is a cohomology class, of the form $a\cdot 1$, with $1$ the canonical generator of $H^0$ of the point orbifold $y$, so that by $t_*\circ s^*(1)$ we actually denote the value $a$.

\begin{remark}
If we consider SFT as in \Cref{rmk:SFT}, after taking the difference between moduli spaces with and without asymptotic markers into account and using a change of coordinate $q_{\gamma}\mapsto \kappa_{\gamma}q_{\gamma}$, the extra multiplicities of the orbits at negative punctures in the SFT structural maps (see e.g.\ \cite{SFT,RSFT,contact_homology}) can be explained as the pullback-pushforward to the orbifold of unparametrized Reeb orbits. 
The change of coordinate may be viewed as the Poincar\'e duality on $\bullet/(\ZZ/\kappa_{\gamma}\ZZ)$,  as the pullback-pushforward construction is the ``cohomological" perspective instead of the ``homological" perspective taken in \cite{SFT,RSFT,contact_homology}.
\end{remark}

\begin{prop}
    $\delta^2=0$.
\end{prop}
\begin{proof}
    By \Cref{prop:composition}, we have
    \begin{eqnarray*}
    \delta^2(x) & = & \sum_{y,z} ((t_{y,z})_*\circ s_{y,z}^*(1))\cdot (t_{x,y})_*\circ s_{x,y}^*(1))z \\
    & = &\sum_{y,z} (t_{y,z})_*\circ s_{y,z}^*\circ (t_{x,y})_*\circ s_{x,y}^*(1)z \\
    & = &\sum_{y,z }(t|_{\cM_{x,y}\times_y\cM_{y,z}})_*\circ s|_{\cM_{x,y}\times y \cM_{y,z}}^*(1)z
    \end{eqnarray*}
    where $s_{a,b},t_{a,b}$ are the source and target maps from $\cM_{a,b}$ to $a,b$ respectively.
    Notice that $\cup_y\cM_{x,y}\times_y \cM_{y,z} = \partial \cM_{x,z}$.
    Then, the fact that $\delta^2=0$ follows from Stokes' theorem and \Cref{prop:glue,prop:glueI,prop:orientation}:
    namely, we have
    $$\int_z \sum_{y}(t|_{\cM_{x,y}\times_y\cM_{y,z}})_*\circ s|_{\cM_{x,y}\times y \cM_{y,z}}^*(1)\wedge 1 = \int_{\cup_y \cM_{x,y}\times_y\cM_{y,z}}1\wedge 1 = \int_{\partial \cM_{x,z}}1=0.$$
\end{proof}

\begin{prop}
    In the formula \eqref{eqn:diff} for $\delta(x)$, the coefficients multiplying each generator $y$ are in $\ZZ$. 
    In particular, the symplectic cohomology for exact orbifold fillings of contact manifolds can be defined over any ring $R$.
\end{prop}
\begin{proof}
    Since $\cM_{x,y}$ is a manifold and only $y$ can be an orbifold point (in the loop space), this follows from \eqref{eqn:pushpull}.
\end{proof}

We are now ready to prove \Cref{thm:SH} from the introduction.

\begin{proof}[Proof of \Cref{thm:SH}]
    So far, we have established all the ingredients to define Hamiltonian-Floer cohomology of a type (I) or (II) Hamiltonian. Namely, for each Hamiltonian $H_a$ of either type with slope $a$, we have a short exact sequence of cochain complexes, 
    $$0\to CF^*_0(H_a;R)\to CF^{*}(H_a;R)\to CF^*_{+}(H_a;R)\to 0$$
    giving rise to the tautological long exact sequence of the associated Hamiltonian-Floer cohomology.
    $$\ldots \to HF^*_0(H_a;R)\to HF^{*}(H_a;R)\to HF^*_{+}(H_a;R)\to \ldots$$
    Now, the continuation map from a Hamiltonian-Floer cohomology with a certain slope to a Hamiltonian-Floer cohomology with a higher slope can be defined similarly. Namely, for $a<b$, we have a continuation map $CF^*(H_a;R)\to CF^*(H_b;R)$, compatible with the short exact sequences and functorial w.r.t.\ increasing the slope,  up to passing to cohomology (i.e.\ they induce a direct system of long exact sequences in cohomology). Then, the (positive) symplectic cohomology is defined to be the direct limit:
    $$SH^*(W;R):=\varinjlim_{a\to \infty} HF^*(H_a;R),\quad SH^*_+(W;R):=\varinjlim_{a\to \infty} HF^*_+(H_a;R).$$
    Moreover, type (I) and type (II) Hamiltonians give the same (up to isomorphism) symplectic cohomology, as there is an isomorphism from that originating from type (I) to that from type (II), c.f.\ \cite[Proposition 2.10]{ADC} for the manifold case. 
    
    To finish the proof of the tautological long exact sequence, we need to show that the subcomplex generated by constant orbits gives rise to the Chen-Ruan cohomology of $W$ using $R$ as a coefficient ring, i.e.\ the cohomology of the orbit space of $\wedge W$ with coefficients in $R$ (as defined in \Cref{ex:CR_coeff}). 
    This can be seen from the type (I) construction as follows. 
    
    First note that all the $(p,(g))$'s with $g\ne \Id$ generate $H^*(\wedge W\backslash W;R)$. 
    Consider now the subcomplex generated by critical points of $f$ that are not singularities and by the $(p, (\Id))$'s for an orbifold singularity $p$; we need to identify this with $H^*(W;R)$.
    It is clear that the subcomplex generated by the first kind of points, i.e.\ the critical points of $f$ that are not singularities, corresponds to the relative cohomology $H^*(W,\sqcup_{p}B_p;R)$, where $B_p$ is a neighborhood of the singularity $p$ given by a local uniformizers $\CC^n/\stab_p$ such that the restriction $f|_{B_p}$ of the Morse function is of the form $\sum  x_i^n+C$.
    Now, it is direct to see that $H^*(B_p;R)\to H^{*+1}(W,\sqcup_{p}B_p;R)$ is given by counting moduli spaces $\cM_{(p,(\Id)),q}$ for an index $1$ critical point $q$ of $f$. 
    Therefore by the tautological long exact sequence of the pair $(W,\sqcup_{p}B_p)$, we see the cohomology of $W$ is given by counting $\cM_{x,y}$ for $x,y$ critical points of $f$, possibly paired with the conjugacy class $(\Id)$ when the critical point is a singularity. 
    This finishes the proof.
\end{proof}

\begin{remark}
So far, type (II) Hamiltonians have not been used in an essential way. 
However, it is useful to have both points of view of type (I) and (II) Hamiltonians, as they each have their own advantages.
For example, for type (I) Hamiltonians, it is easier to apply neck-stretching as in \cite{ADC} and thus interact with constraints from $W$, see \Cref{prop:restriction} and \ref{prop:curve}.
On the other hand, it is easier to define the ring structure (and related structures) for type (II) Hamiltonians. 
In our case, \Cref{thm:quotient} and \ref{thm:unique} are also obtained using type (II) Hamiltonians.
\end{remark}

\subsection{Basic properties of the symplectic cohomology.}\label{SS:property}
The main obstruction to achieving transversality in orbifold Floer theory, that cannot be solved by perturbing generically the almost complex structure or Hamiltonian, is the appearance of constant Floer curves\footnote{In the case of more general symplectic orbifolds (i.e.\ singularities are not isolated), the failure of transversality, in addition to the usual issues (sphere bubbles etc.), are certain, not necessarily constant, curves contained in the singular locus. The linearized operator, restricted to the normal direction, may not be surjective. This is a special case of the general failure of equivariant transversality. Our discussion on exact orbifolds with manifold boundaries is a specialization of such phenomena.} (note that the asymptotics at punctures can be non-trivial twisted sectors) whose image is contained in a singularity and is not transversely cut out. There are two sources of such failures of transversality: 
\begin{enumerate}
\item  certain constant curves with at least three punctures, e.g.\ in the ring structure of symplectic cohomology; 
\item certain constant Floer cylinders due to the rigidity of Hamiltonians enforced by the local symmetry near the singularity.
\end{enumerate} 
We will explain them in detail in \S \ref{SS:fail}. However, the following structures in symplectic cohomology do not involve such Floer curves. 
Hence, by an argument analogous to the manifold case with similar modifications as we did for the definition of symplectic cohomology in \S \ref{sec:sympl_cohom}, one can prove the following:
\begin{thm}\label{thm:property}
The orbifold symplectic cohomology defined above in the case of isolated singularities satisfies the following properties.
\begin{enumerate}
    \item Given an exact embedding $V\subset W$, there is a \emph{Viterbo transfer map} $\Phi_{\vit}$ compatible with tautological long exact sequence as follows: 
    $$
    \xymatrix{
    \ldots \ar[r] & H^*_{\CR}(W) \ar[r]\ar[d] & SH^*(W) \ar[r]\ar[d]^{\Phi_{\vit}} & SH^{*}_+(W) \ar[r]\ar[d]^{\Phi_{\vit}} & H^{*+1}_{\CR}(W)\ar[r] \ar[d] & \ldots \\
    \ldots \ar[r] & H^*_{\CR}(V) \ar[r] & SH^*(V) \ar[r] & SH^{*}_+(V) \ar[r] & H^{*+1}_{\CR}(V)\ar[r] & \ldots}
    $$
    \item The (positive) symplectic cohomology is an invariant of the orbifold filling (with isolated singularities) up to Liouville homotopy.
    \item The $S^1$-equivariant (positive) symplectic cohomology can be defined similarly for any ring coefficient. 
    In particular, symplectic cohomology is equipped with a BV operator $\Delta:SH^*(W)\to SH^{*-1}(W)$ such that $\Delta^2=0$.
\end{enumerate}
\end{thm}

\begin{thm}[\cite{Cieliebak02}]
Let $W'$ be obtained from attaching a subcritical handle to an exact orbifold $W$ with contact manifold boundary, then $\Phi_{\vit}:SH^*(W')\to SH^*(W)$ is an isomorphism.	
\end{thm}
More generally, the same should hold for flexible handle attachments by \cite{BEE}.  
Lastly, we point out that it is a natural question that of understanding the behavior of orbifold symplectic cohomology under the attachment of orbifold handles; we hope to address this problem in a sequel paper.

\begin{remark}
There are many constructions in symplectic cohomology that are orthogonal to the one in this paper, hence can be realized using classical tools as in this paper. 
To name a few, the cascades constructions in \cite{cascades,divisor}, the neck-stretching arguments in \cite{exact,cieliebak18,ADC}, and the construction of invariants associated to cobordisms in \cite{cieliebak18}. 
Those results hold without appealing to virtual techniques.  
\end{remark}

The following proposition is crucial for the geometric applications in \S \ref{S5}. 
Let $SH^*_{+,S^1}(W)$ be the positive equivariant symplectic cohomology, we write $Y=\partial W$ for the contact boundary. 
We then denote by $\eta,\eta_{S^1}$ the following compositions:
$$\eta:SH^*_+(W;\ZZ)\to H^{*+1}_{\CR}(W;\ZZ)\to H^{*+1}(W;\ZZ)\to H^{*+1}(Y;\ZZ)\to H^0(Y;\ZZ)=\ZZ \; ,$$
$$\eta_{S^1}:SH^*_{+,S^1}(W;\ZZ)\to H^{*+1}_{\CR}(W;\ZZ)\otimes \ZZ[u,u^{-1}]/u\to H^{*+1}(Y;\ZZ)\otimes \ZZ[u,u^{-1}]/u\to H^0(Y;\ZZ)=\ZZ \; .$$
Let us briefly comment on the definition of these arrows, starting from those of $\eta$. 
The first map is from the tautological long exact sequence; the second one is the natural projection $H^{*+1}_{\CR}(W;\ZZ)\to H^{*+1}(W;\ZZ)$ that forgets the cohomology of the twisted sectors; the third one is just the restriction map; the fourth and last one is the projection onto the degree $0$ part of singular cohomology.
The arrows in the definition of $\eta_{S^1}$ are defined similarly, using the tautological long exact sequence from the $S^1$-equivariant symplectic cohomology.
(Note that, with respect to the definition of $\eta$, in that of $\eta_{S^1}$ we have here contracted two arrows in a single one, namely the second arrow factors through $H^{*+1}(W;\ZZ)\otimes \ZZ[u,u^{-1}]/u$.)

\begin{prop}\label{prop:restriction}
Let $Y$ be a contact manifold and $W$ an exact orbifold filling with isolated singularities. 
Assume there is $N\in\ZZ$ and a class $x$ in $ SH^*_+(W;\ZZ)$ or $SH^*_{+,S^1}(W;\ZZ)$, such that $\eta(x)=N$ or $\eta_{S^1}(x)=N$ respectively. Then the order of the isotropy groups of points in $W$ must divide $N$.
\end{prop}
\begin{proof}
Let $q\in Y$ and $H_t$ a type (I) Hamiltonian. 
We consider, for $y$ in the set $\cP^*(H_t)$ of non-constant $1-$periodic orbits, the moduli space $\cM_{y,q}$, which is the compactification of $B_y\times_W \{q\}$. 
Following the argument in \cite[\S 3]{ADC}, 
one can see that $\eta(y)$ is defined by counting zero dimensional moduli spaces of the type $\cM_{y,q}$. 

Now let $p$ be a singularity of $W$, and fix a smooth embedding $I:[0,1]\to W$ such that $I(0)=p$ and $I(1)=q$. 
Assume for simplicity that the class $x$ as in the statement is represented by an orbit, 
also denoted by $x$ with a slight abuse of notation. In general, when $x$ is represented by an integer combination of orbits, the moduli spaces below should be replaced by unions of copies of the moduli spaces depending on the integer coefficients. 
Then we can consider the compactification of $B_x\times_W [0,1]$, which is an oriented compact orbifold with boundary. 
More precisely, its boundary is made up of three types of curves: (1) those coming from Floer breakings for $B_x$, which will cancel each other as $x$ represents a cohomology class in $SH^*_+(W;\ZZ)$;
(2) those in the fiber product with $I|_{\{1\}}$, which is $\cM_{x,q}$; 
(3) the fiber product with $I|_{\{0\}}$, i.e.\ a fiber product of the form $M\times_{\{p\}/\stab_p}\{p\}$, for a zero dimensional moduli space $M$ consisting of solutions to \eqref{eqn:B'} with positive asymptotic $x$ and $u(0)=p$. 
Therefore, by Stokes' theorem, we have
$$N=\eta(x)=\int_{\cM_{x,q}}1=\int_{M\times_{\{p\}/\stab_p}\{p\}}1=|\stab_p|\cdot \int_{M}1.$$
Hence $|\stab_p|$ divides $N$. 
The proof for the $\eta_{S^1}$ case is completely analogous.
\end{proof}

\subsection{When does classical transversality fail?}\label{SS:fail}
Although we cover different features of Floer theories of symplectic orbifolds, considering only isolated singularities (which corresponds to the smooth boundary case in the exact symplectic orbifold setting) allows us to avoid significant technical difficulties. 
In the following, we explain two situations where we are forced to leave the realm of geometric transversality arguments, even though sphere bubbles can be avoided. 

\subsubsection{Continuation maps for more general Hamiltonians.}\label{SSS:continuation}
In the definition of symplectic cohomology for manifolds, one only requires $H$ to have slope $a$ when $r\gg 0$. 
The additional requirement on $C^2$-smallness is to obtain the definition of positive symplectic cohomology and the tautological exact sequence. 
Moreover, having the Hamiltonian (close to) zero on $W$ resembles more closely the constructions in SFT.
However, the freedom of considering more general Hamiltonians should allow us to obtain a simple proof of the fact that $SH^*(\CC^n;R)=0$ for any ring $R$ \cite[(3f)]{biased}.
In principle, such freedom, even though allowed, comes with restrictions on the choice of coefficients for the orbifold theory. In particular, for certain coefficients, e.g.\ $\ZZ$, the continuation maps between two general Hamiltonians face obstructions as we will explain below. 

\begin{example}\label{ex:vanish}
Suppose we use $H=Ar^2$ to define the Hamiltonian-Floer cohomology on $\CC^n/G$. 
Then we can find a suitable sequence of $A_i$ going to $\infty$, such that $H_i=A_ir^2$ are non-degenerate with only periodic orbits $(0,(g))$ for $g\in G$.
Moreover, one can check that the rational grading of $(0,(g))$ goes to $-\infty$ as $i\to \infty$. 
Assuming that the continuation map from $H_i$ to $H_{i+1}$ is well-defined, one could then define a new symplectic cohomology $\mathcal{SH}$ as the direct limit. 
With this definition, $\mathcal{SH}^*(\CC^n/G)=0$. \END
\end{example}

\begin{example}\label{ex:non_vanish}
We will argue that $\mathcal{SH}$ ``defined" in \Cref{ex:vanish} is \emph{not} the same object as defined in \Cref{thm:SH}.
Following the discussion in \cite[Remark 2.16]{Zho}, the contact boundary $(\mathbb{RP}^3,\xi_{\std})$ is equipped with a Boothby-Wang contact form and has an exact filling $T^*S^2$.
The Boothby-Wang contact form can be perturbed into a non-degenerate contact form with exactly two simple Reeb orbits $\gamma,\delta$, where the period of $\gamma$ is slightly smaller.
One can then consider the moduli spaces involved in the definition of positive symplectic cohomology for both $\CC^2/(\ZZ/2\ZZ)$ and $T^*S^2$. 
By neck-stretching and index computation, in addition to cylinders in the symplectization that can be identified for  $\CC^2/(\ZZ/2\ZZ)$ and $T^*S^2$, we also need to count cylinders with multiple negative punctures all asymptotic to $\gamma$ as well as augmentation curves to $\gamma$. 
In $\CC^2/(\ZZ/2\ZZ)$, there are no augmentation curves to $\gamma$, as $\gamma$ is not contractible; on the other hand, in $T^*S^2$, there are $2$ of them. 
Therefore in $\ZZ/2\ZZ$ coefficients, we have $SH^*_+(\CC^2/(\ZZ/2\ZZ),\ZZ/2\ZZ)=SH^*_+(T^*S^2;\ZZ/2\ZZ)$. 
Since the latter is infinitely generated, we have $SH^*(\CC^2/(\ZZ/2\ZZ);\ZZ/2\ZZ)\ne 0$. 
Therefore, the definition in \Cref{ex:vanish} is, at least for some choice of coefficients, not equivalent to the definition in \Cref{thm:SH} in general.\END
\end{example}

Note that the continuation map in \Cref{ex:vanish} solves,
\begin{equation}\label{eqn:continuation}
    \partial_su+J_t(\partial_su-X_{H_{s,t}})=0
\end{equation}
for $H_{s,t}=A_{i+1}r^2$ when $s\ll 0$ and $H_{s,t}=A_{i}r^2$ when $s\gg 0$. 
As explained in \Cref{ex:Morse}, $0$ is a critical point of $H_{s,t}$ for every $s,t$. 
The trivial cylinder over $(0,(g))$ is also a solution to \eqref{eqn:continuation}. 
However, the Fredholm index is smaller than $0$ by our choice of $A_i$, which makes it not cut-out transversely. 
Such a phenomenon cannot be perturbed away (as the constant map is not somewhere injective)\footnote{In the manifold case, we perturb away such a phenomenon by perturbing $H$, since $H$ is not required to respect any symmetry. 
}. 
Such a problem will be inherited in the construction of continuation maps between the construction in \Cref{thm:SH} and \Cref{ex:vanish}, as well as the homotopy between different continuation maps. 
In particular, one is forced to use abstract machineries (Kuranishi structures, virtual fundamental cycles, or polyfolds) to count the moduli spaces. 
Moreover, the non-transverse constant cylinder above also has isotropy in the moduli space, which potentially leads to the necessity of multi-valued perturbations. 
Indeed, it is in the homotopy that the count must be defined over a ring like $\QQ$. 
This explains the discrepancy between \Cref{ex:vanish,ex:non_vanish}. 
We will take a closer look at such a phenomenon and prove \Cref{thm:quotient} in \S \ref{S4}.

\subsubsection{Pair of pants product.}\label{SSS:product} 
It is natural to expect that the symplectic cohomology of orbifold fillings carries a ring structure, such that the map $H^*_{\CR}(W)\to SH^*(W)$ in the tautological exact triangle of \Cref{thm:SH} becomes a ring map.
As usual, the ring structure should be given by counting the moduli spaces of pairs of pants (e.g.\ \cite[\S 2.3.1]{viterbo}) as follows:
\begin{equation}\label{eqn:pants}
    (\rd u-\alpha \otimes X_{H^P})^{0,1}=0, \quad \int_{P}|\rd u-\alpha \otimes X_{H^P}|^2<\infty,
\end{equation}
where
\begin{enumerate}
    \item $P$ is a sphere punctured at $z_1,z_2,z_3$, where the punctures $z_1,z_2$ are equipped with positive cylindrical ends $\epsilon_1,\epsilon_2: [0,\infty)\times S^1\to P$, and the puncture $z_3$ is equipped with a cylindrical end $\epsilon_3: (-\infty,0]\times S^1\to P$;
    \item $\alpha$ is a $1-$form on $P$, such that $\epsilon_i^*\alpha=\rd t$;
    \item $H^P$ is a family of Hamiltonians parametrized over $P$, all with linear slopes given by a function $b^P\colon P\to \RR$ satisfying $\rd(b^P\alpha)\leq 0$, and such that on the ends $\epsilon_i$ we have $H^P_{\epsilon_i(s,t)}=H_t^i$, where $H_t^i$ is a type (II) Hamiltonian for $1\le i \le 3$;
    \item $u:P\to \widehat{W}$ is an orbifold map.
\end{enumerate}
Then the finite energy condition implies that $u$ is asymptotic to orbits of $H^1,H^2,H^3$ near the punctures $z_1,z_2,z_3$ respectively. 
The existence of a Gromov-Floer compactification again follows from the integrated maximum principle using $\rd (b^P\alpha)\le 0$ as in \cite[Lemma 2.3.9]{viterbo}.
Moreover, if $u$ is not constant, then it is not a singular point in the orbifold of maps from $P$ to $\widehat{W}$, hence somewhere injectivity of $u$ implies that we can choose $J$ generically to get transversality.

On the other hand, since $\rd H^P(p)=0$ if $p$ is a singularity of $W$ by \Cref{ex:Morse}, the constant maps from $P$ to $\widehat{W}$ described by $W^3\backslash W$ in \Cref{ex:constant} are always solutions of \eqref{eqn:pants}. 
Let $u\in W^3\backslash W$ be $(p,(g_1,g_2,g_3))$ for $g_1,g_2,g_3\in \stab_p\backslash \{\Id\}$ such that $g_1g_2g_3=\Id$; recall that $(g_1,g_2,g_3)$ denotes the equivalence class under the relation $(g_1,g_2,g_3)\simeq (hg_1h^{-1},hg_2h^{-1},hg_3h^{-1})$. 
Following \Cref{ex:Morse}, the asymptotics at $z_1,z_2$ are $(p,(g_1)),(p,(g_2))$ respectively, and the one at $z_3$ is $(p,(g^{-1}_3))$ (note the asymptotic is with respect to the orientation from the negative cylindrical end). The Fredholm index of $u$ is given by 
$$\ind D_u=-2\age(g_1)-2\age(g_2)+2\age(g^{-1}_3)\le 0
\footnote{In general $\age(A)+\age(B)\ge \age(AB)$ for $A,B\in U(n)$. 
To see this, note that $\age(A)+\age(B)=\age(AB)$ in $\QQ/\ZZ$, then it is easy to show that $\age(A)+\age(B)\ge \age(AB)$ for $\age(A)$ small and the general case follows.}.
$$

When $\ind D_u=0$, $u$ is cut out transversely with isotropy $C(g_1)\cap C(g_2)$\footnote{To see this, note that $\ker D_u=\{0\}$ for any such $u$, then $D_u$ is transverse iff $\ind D_u=0$. }.
The pullback-pushforward counting of the moduli spaces of those constant pairs of pants gives rise to the Chen-Ruan product in \S \ref{SS:CR}. 
When $\ind D_u<0$, those $u$ are not cut out transversely, but equipped with an obstruction bundle $\coker D_u$. 
In the definition of Chen-Ruan product on $H^*_{\CR}(W)$, we simply disregard them, since $\dim  W^3\backslash W=0$ and there cannot be any contribution from $u$ when $\ind D_u<0$.   

However, in the definition of ring structure on $SH^*(W)$,
those constant curves above can appear as a component in the compactification of the moduli spaces of non-constant pairs of pants. 
As a result, we do not know if the zero dimensional ones of these moduli spaces (which can be assumed to be transversely cut-out) are compact. 
Moreover, even if we knew that all of them are compact, it is not clear that the count gives an operator compatible with the differential. 
This being said, one does have obstruction bundles, so, at least theoretically, the gluing methods in \cite{hutchings2007gluing, hutchings2009gluing} might solve the problem. 
We point out however that this situation here is more complicated than that in ECH, as the base of the obstruction bundle here is given by abstract moduli spaces of non-constant holomorphic curves, which we seem to have no a priori control over.

In the following, we will make some speculations, \emph{assuming the ring structure of $SH^*(W;\QQ)$ can be defined and has usual properties like symplectic cohomology for manifold fillings}. 
However, the following results (and stronger ones) will be obtained in \S \ref{S5} without appealing to the ring structure (at least not in a direct way, even though the ``essence'' of the argument is the same). 

\begin{example}\label{ex:singularity}
Let $W$ be an orbifold filling of $(\mathbb{RP}^{2n-1},\xi_{std})$ for $n\ne 2^k$. If the properties of symplectic cohomology used in \cite[Proposition 2.10]{Zho} (more precisely, the existence of the pair of pants product on filtered symplectic cohomology so that $H^*_{\CR}(W;\QQ)\to SH^*(W:\QQ)$ is a unital algebra map) hold for orbifolds, then we have $\dim H^*_{\CR}(W;\QQ)\le 2$ by the same argument in \cite[Propositions 2.9, 2.15]{Zho}. 
Since we know that $W$ cannot be smooth by \cite{Zho}, it must have exactly one singularity modeled on the quotient $\CC^n/(\ZZ/2\ZZ)$ by the antipodal map, as $\ZZ/2\ZZ$ is the only group with two conjugacy classes and there is a unique unitary $\ZZ/2\ZZ-$action on $\CC^n$ giving an isolated singularity.  \END
\end{example}

\begin{example}
Let $V$ be a Liouville domain and consider the contact manifold $Y=\partial(V\times \DD)$. 
If the properties of symplectic cohomology mentioned in \Cref{ex:singularity} hold for orbifolds, then the proof of \cite[Theorem 1.1]{zhou2020} yields that any exact orbifold filling $W$ of $Y$ has vanishing rational symplectic cohomology and $SH^*_+(W;\QQ)\simeq H^{*+1}_{\CR}(W;\QQ)\to H^{*+1}(Y;\QQ)$ is injective. The second map is the composition of $H^*_{\CR}(W;\QQ)\to H^*(W;\QQ)\to H^*(Y;\QQ)$, alternatively, it is the restriction $H^*_{\CR}(W)\to H^*_{\CR}(Y)$, where $H^*_{\CR}(Y)$ is the Chern-Ruan cohomology of the almost complex orbifold $Y\times \RR$ and is $H_*(Y)$ as $Y$ is a manifold. The surjectivity of $SH^*_+(W;\QQ)\simeq H^{*+1}_{\CR}(W;\QQ)$ uses crucially the ring structure on $SH^*(W;\QQ)$ and its filtered version \cite[Proposition 2.10]{Zho}. Now, if $W$ carries a singularity, then $H^{*}_{\CR}(W;\QQ)\to H^{*}(Y;\QQ)$ is never injective as those classes representing non-trivial conjugacy classes are in the kernel, hence we conclude that $W$ is smooth. 
This is a proof of Theorem \ref{thm:sphere_have_smooth_fillings} assuming the ring structure. \END
\end{example}

In conclusion, following the discussion in \S \ref{SSS:continuation}, \ref{SSS:product}, as far as only exact orbifold fillings of contact manifolds are concerned, the main problem to overcome is the possible presence of constant maps concentrated in an orbifold singularity that are not cut out transversely. 
If one can avoid them, then any construction on exact symplectic manifolds can be realized in the orbifold setting with classical tools, and this is, in particular, the case with results in \S \ref{SS:property}. 
This being said, if we consider general orbifold fillings of contact orbifolds, the situation gets much more complicated, and
it seems that one has to deploy some abstract machinery. 
In sequel papers, we will study such questions in more general settings.

	\section{Symplectic cohomology of $\CC^n/G$}\label{S4}
In this section, we consider the local case of isolated quotient singularities. Namely, let $G\subset U(n)$ be a finite subgroup, such that $G$ acts on $S^{2n-1}\subset \CC^n$ freely. Then $\CC^n/G$ is naturally (the completion of) an exact orbifold filling of the quotient $(S^{2n-1}/G,\xi_{\std})$.  
\subsection{Reeb dynamics of $(S^{2n-1}/G,\xi_{\std})$.}\label{SS:Reeb}
Since $G\subset U(n)$, we have that $\alpha=\frac{1}{2}\sum (x_i\rd y_i-y_i\rd x_i)=-\frac{1}{2}r\rd r\circ I$ is $G-$invariant, where $I$ is the standard complex structure on $\bbC^n$.
Hence $\alpha$ descends to a contact form $\alpha_{\std}$ on $(S^{2n-1}/G,\xi_{\std})$.
We now recall the description of the Reeb dynamics of $\alpha_{\std}$ from \cite{McKay}.

For $l>0\in \RR$ and $g\in G$, we define 
$$V_{g,l}=\left\{v\in \CC^n\left|g(v)=e^{l\i}v\right.\right\}.$$
It is easy to check that $h(V_{g,l})=V_{hgh^{-1},l}$ for $h\in G$. 
Since the Reeb flow of $\alpha$ on $S^{2n-1}$ is given by $\phi_t(z)=e^{t\i}z$, we have the following properties of the Reeb dynamics of $\alpha_{\std}$ on $S^{2n-1}/G$.

\begin{enumerate}
    \item Let $B_{g,l}=\pi_G(V_{g,l}\cap S^{2n-1})$, where $\pi_G$ is the quotient map $S^{2n-1}\to S^{2n-1}/G$. 
    Then, $B_{g,l}=(V_{g,l}\cap S^{2n-1})/C(g)$ is a submanifold of $S^{2n-1}/G$ of dimension $2\dim_{\CC} V_{g,l}-1$. 
    Moreover, $B_{g,l}=B_{hgh^{-1},l}$,
    hence, for $(g)\in \Conj(G)$, one can define $B_{(g),l}$ to just be $B_{g,l}$.
    \item $\displaystyle \sqcup_{(g),l} B_{(g),l}$ is the space of all \emph{parametrized} Reeb orbits of $\alpha_{\std}$. 
    For $z\in B_{(g),l}$, the time $l$ Reeb flow from $z$ is an orbit with homotopy class $(g)\in \Conj(G)=\pi_0(\cL (S^{2n-1}/G))=\pi_0(\cL( \CC^n/G))$ (recall \Cref{cor:pi0_free_loop_space}), where $\cL$ denotes the loop space.
    \item  The generalized Conley-Zehnder index\footnote{Similarly to \Cref{ex:U}, since $c^{\QQ}_1(\CC^n/G)=0$ and $\pi_0(\cL \CC^n/G)$ is torsion, the generalized Conley-Zehnder index is also well-defined over $\QQ$.} of the family $B_{(g),l}$ is given in \cite[Theorem 2.7]{McKay} by 
    $$\mu_{\CZ}(B_{(g),l})=n-2\age(g)+2\sum_{l'<l}\dim_{\CC} V_{g,l'}+\dim_{\CC}V_{g,l}.$$
    Then we have $\mu_{\CZ}(B_{(g),l+2k\pi})=\mu_{\CZ}(B_{(g),l})+2kn$. 
\end{enumerate}

In the following, we will consider Hamiltonians $H$ on $\CC^n/G$ which are type (I) and autonomous, and such that $H=0$ for $r<1$ and $H=f(r^2)$\footnote{Note that $r^2$ is the cylindrical coordinate on $\CC^n/G$.}  on $r\ge 1$, with $f''(x)\ge 0$ for $x\ge 1$ and $f'(x)=a \notin \frac{2\pi \NN}{|G|}$ for $x>1+\epsilon$ for some $\epsilon>0$.
The non-constant Hamiltonian orbits of $H$ are parametrized by $B_{(g),l}$ for $l<a$. 
In particular,the symplectic cohomology can be equivalently defined using the cascades construction, c.f.\ \cite{cascades,divisor}, i.e.\ we need to choose a Morse function $f_{(g),l}$ for each $B_{(g),l}$ and the generators from non-constant orbits of $H$ are replaced by critical points of $f_{(g),l}$. 
The Conley-Zehnder index of a critical point $x$ of $f_{(g),l}$ is given by 
\begin{equation}\label{eqn:CZ}
    \mu_{CZ}(x)=n-2\age(g)+2\sum_{l'<l}\dim_{\CC}V_{g,l'}+1+\ind(x),
\end{equation}
where $\ind(x)$ is the Morse index of $x$. With such conventions, the flow lines in $B_{(g),l}$ are \emph{negative} gradient flow lines of $f_{(g),l}$ in the cascades construction. On the other hand, if we use $f_{(g),l}$ to perturb the Hamiltonian following \cite{cascades}, then each critical point of $x$ gives rise to a non-degenerate Hamiltonian orbit $\gamma_x$, whose Conley-Zehnder index is given by \eqref{eqn:CZ}. The symplectic action of $\gamma_x$ is greater than that of $\gamma_y$ iff $f_{(g),l}(x)<f_{(g),l}(y)$, see \cite[Lemmas 2.3 and 2.4]{BoThesis}. 
(Strictly speaking, the latter reference states the relation between Conley-Zehnder indices and periods for Reeb orbits after a small perturbation of the Morse-Bott contact form, namely a larger period means a larger Conley-Zehnder index for those orbits created from a Morse-Bott family. This being said, since the symplectic action is roughly the negative period, our claim follows by the same argument in \cite[Lemmas 2.3 and 2.4]{BoThesis}.)

\subsection{Null-homotopy of $|G|\cdot\Id$.}
Let $\CC^n_*$ denote $\CC^n\backslash\{0\}$. 
Among all Hamiltonian orbits of $H$, the most special ones are those in the family $B_{(\Id),2\pi}$, which are lifted from $\CC^n_*/G$ to $\CC^n_*$ by the family of simple orbits of the Hamiltonian $\pi^*_GH$ on $\CC^n_*$. 
Now, one can arrange that $f_{(\Id),2\pi}$ has a unique local minimum, which will be denoted by $\gamma_0$ (we can think of it as the contractible non-constant Hamiltonian orbit with maximum symplectic action after perturbation induced from $f_{(\Id),2\pi}$).
\begin{prop}\label{prop:curve}
When $n\ge 2$, we have $\# \cM_{\gamma_0,(0,(\Id))}=1$ (for both type (I) and (II) Hamiltonians).
\end{prop}
Here $\cM_{\gamma_0,(0,(\Id))}$ can either be understood as the moduli space in \S \ref{S3} for $\gamma_0$ a Hamiltonian orbit from a perturbation induced by $f_{(\Id),2\pi}$, or as its cascades version.
Note that \Cref{prop:curve} was worked out by hand in \Cref{ex:hand} for some special cases. 
We also point out that, since non-constant Hamiltonian orbits of $H$ with larger symplectic action (i.e.\ smaller contact action) are not contractible,
there are no multiple-level structures in the cascades. 
In particular, $\cM_{\gamma_0,(0,(\Id))}$ consists of Floer curves $u:\CC\to \CC^n/G$ such that $u(0)=0$ and $\displaystyle\lim_{s\to \infty} u(s,0)=\gamma_0\in B_{(\Id),2\pi}$, since $\gamma_0$ corresponds to the minimum of $f_{(\Id),2\pi}$.

\begin{proof}[Proof of \Cref{prop:curve}]
We first prove the claim for the type (I) Hamiltonians. Let $q$ be a smooth point of $\BB/G$, where $\BB=\BB_1$ is the unit ball in $\CC^n$. We consider the moduli space
\begin{equation}\label{eqn:moduli_q}
    \cM_{\gamma_0,q}=\left\{u:\CC\to \CC^n/G\left|\partial_su+J_t(\partial_tu-X_H)=0,\displaystyle \lim_{s\to \infty}u=\gamma_0, u(0)=q \right.\right\}\footnote{Or $\displaystyle \lim_{s\to \infty}u(s,0)=\gamma_0\in B_{(\Id),2\pi}$ in the cascades version.}.
\end{equation}
Since $\gamma_0$ has maximum symplectic action (or the minimum of $f_{(g),l}$) among  contractible orbits, we know that $\cM_{\gamma_0,q}$ is compact of dimension zero for generic $J_t$, and that its algebraic count is independent of $q$. 
Next, by applying neck-stretching to a push-in $Y'$ of $\partial \BB/G$ such that $q$ is outside of $Y'$ (i.e.\ on the non-compact component of the complement of $Y'$), one can conclude that $\cM_{\gamma_0,q}$ is also contained outside of $Y'$, as all shorter Reeb orbits are non-contractible. 
Now we can apply the same covering trick in \cite[Proposition 2.12 Step 3]{Zho} to obtain that $\#\cM_{\gamma_0,q}=|G|$. 
Then, following the same argument as in \Cref{prop:restriction}, we get $\#\cM_{\gamma_0,(0,(\Id))}=1$.

To prove the statement for a type (II) Hamiltonian $H_t$, we consider the continuation map from a type (II) construction using $H_t$ to a type (I) construction using $f=H_t|_{\BB/G}$ as in the proof of \Cref{thm:SH}, i.e.\ as in \cite[Proposition 2.10]{ADC}.
We use $\tilde{\cM}_{\gamma_0,(0,(\Id))}$ to denote the counterpart for a type (II) Hamiltonian.  
Let also $\cH_{x,y}$ denote the moduli space from the definition of the continuation map from $x$ to $y$, as in \cite[Proposition 2.10]{ADC}.
By using Stokes' theorem and considering the boundary of $\cH_{\gamma_0,(0,(\Id))}$, we get the relation 
\begin{equation}\label{eqn:stokes}
    \int_{\cH_{\gamma_0,\gamma_0}\times \cM_{\gamma_0,(0,(\Id))}} 1=\int_{\tilde{\cM}_{\gamma_0,(0,(\Id))}\times_{\{0\}/G} \cH_{(0,(\Id)), (0,(\Id))}}1.
\end{equation}
Now, $\cH_{\gamma_0,\gamma_0}$ consists of one point, the constant cylinder. 
On the other hand, we define
$$
\tilde{B}_{(0,(\Id))}:=\left\{u:\CC\to \CC^n/G\left|\partial_su+J_t(\partial_t u-\rho(s)X_{H_t})=0, \displaystyle \lim_{s\to\infty} u = (0,(\Id)), u(0)\in \BB_{1-\delta}/G \right.\right\}\footnote{This is a slight abuse of notion, as $\tilde{B}_{(0,(\Id))}$ is not a set but rather a groupoid, as we are considering maps to an orbifold.},$$
i.e. the analogue of \eqref{eqn:B'}  without quotient by $\RR$ for the homotopy of Hamiltonians $\rho(s)H_t$, where $\rho(s)=1$ for $s\gg 0$ 
and $\rho(s)=0$ for $s\ll 0$. 
Then $\cH_{(0,(\Id)), (0,(\Id))}$ is defined as the fiber product of $\tilde{B}_{(0,(\Id))}$ with the stable orbifold for $f$ of the point $0$ over $\CC^n/G$. Now since the stable orbifold of $0$ is also $\bullet/G$ corresponding to the singular point $0$,  we have 
 $$\cH_{(0,(\Id)), (0,(\Id))}=\left\{ u\in \tilde{B}_{(0,(\Id))}, u(0)=0\right\}.$$ 
Since we can achieve transversality for $\cH_{(0,(\Id)), (0,(\Id))}$ for $J_t$ time-independent on $\BB/G$ and then curves in $\cH_{(0,(\Id)), (0,(\Id))}$ must be $S^1$-invariant as the expected dimension is $0$. 
Hence the equation becomes $\partial_su+\rho(s)\nabla f=0,\partial_tu=0$, and we see that $\cH_{(0,(\Id)), (0,(\Id))}$ is also $\bullet/G$ consisting only of the constant map resting at $0$. As a consequence, we have that \# $\tilde{\cM}_{\gamma_0,(0,(\Id))}$ is also $1$ by \eqref{eqn:stokes}.
\end{proof}

\begin{proof}[Proof of \Cref{thm:quotient}]
Since the other potential components (i.e.\ besides $(0,(id))$) of the differential of $\gamma_0$ must be orbits with bigger symplectic action, which are all non-contractible, we in fact have $\delta(\gamma_0)=|G|(0,(\Id))$. 
Notice that, as $(0,(\Id))$ should represent the unit in the expected ring structure on $SH^*(\CC^n/G;\QQ)$, assuming the ring structure is well defined we would immediately get that $SH^*(\CC^n/G;\QQ)=0$.
This being said, one can understand the vanishing phenomenon without appealing to the full product structure.

In order to do so, we look at the moduli space $\cM_{\gamma_0,x,y}$ in the (tentative) definition of ring structure in \Cref{SSS:product} with type (II) Hamiltonians $H_1,H_2,H_3$ associated to the three punctures, so that $H_1$ observes $\gamma_0$ as an orbit, i.e.\ that its asymptotic slope will make $\gamma_0$ a generator. 
Assuming transversality for all relevant moduli spaces, by counting $\cM_{\gamma_0,x,y}$, we get a map $\psi:C^*(H_2)\to C^{*-1}(H_3)$ such that $\psi(x)=\sum_{\cM_{\gamma_0,x,y}} t_*\circ s^*(1)y$.
Similarly, we define $\phi:C^*(H_2)\to C^{*}(H_3)$ by counting $\cM_{(0,(\Id)),x,y}$. 
Then, we have
\begin{equation}\label{eqn:homotopy}
    |G|\phi=\delta \circ \psi-\psi\circ \delta,
\end{equation}
i.e.\ $\psi$ is a null-homotopy of $|G|\phi$. On the other hand, $\phi$ is homotopic to the continuation map $\iota\colon C^*(H_2)\to C^*(H_3)$, e.g.\ by the TQFT formalism in \cite{MR3065181}. To realize all these moduli spaces and the associated algebraic relations, note that in the compactification of $\cM_{\gamma_0,x,y}$ and $\cM_{(0,\Id),x,y}$, the only component involving a constant map concentrated in the singularity is given by the conjugacy class of $(\Id,g,g)$, which is cut out transversely. 
Therefore, we are free from the issues in \S \ref{SSS:product} and transversality can be obtained through perturbing the almost complex structure. In particular, $|G|\iota$ is null homotopic for any $H_2$ and $H_3$ such that $H_3>H_1+H_2$, hence $SH^*(\CC^n/G;R)=0$ if $|G|$ is invertible in $R$.

From the discussion above, it only remains to prove that $SH^*(\CC^n/G;R)\ne 0$ if $|G|$ is not invertible in $R$. 
Note that $\gamma_0$ has cohomological grading $-1$, which is the unique maximum among all possible contractible generators, since we arranged that $f_{(\Id),2\pi}$ has a unique local minimum. Therefore, if $\delta(x)=(0,(\Id))$ for some $x$, then $x$ must in fact be $r\cdot \gamma_0$ for some $r\in R$, which then gives $r\cdot \vert G\vert =1$ in $R$. In particular, if $(0,(\Id))$ is trivial in cohomology, then $|G|$ is invertible, which then concludes the proof.
\end{proof}

\begin{remark}
For $0<\epsilon\ll 1$, consider the continuation map $\phi_{12}$ from $H_1=\epsilon r^2$ to $H_2=(2\pi+\epsilon)r^2$, and then $\phi_{23}$ from the latter to a type (II) Hamiltonian $H_3$ with slope $(2\pi+2\epsilon)$.
Notice that the orbits of $H_2$ are constant loops $(0,(g))$ with cohomological grading $2\age(g)+2n$. 
As explained in \Cref{ex:vanish}, the continuation map $\phi_{12}$ from $H_1\to H_2$ should be zero, as the grading regions do not intersect, even though the relevant moduli spaces are not cut-out transversely.
The continuation map $\phi_{23}$ from $H_2$ to $H_3$ has no transversality issue. 

Next we consider the composition of those two continuation maps. 
Following the usual formalism of Hamiltonian-Floer cohomology, the composition should be homotopic to the continuation map $\phi_{13}$ from $H_1$ to $H_3$. 
Let then $\eta$ denote the homotopy. 
Since $\phi_{13}((0,(\Id)))=(0,(\Id))$ and $\phi_{23}\circ \phi_{12}((0,(\Id)))=0$ by degree reasons, we must have $\delta\circ \eta((0,(\Id)))=(0,(\Id))$. 
Therefore $\eta(0,(\Id))$ can only be $\frac{1}{|G|}\gamma_0$, i.e.\ the homotopy can only be defined on $\QQ$. 
This can be explained as follows. 
The continuation map $\phi_{12}$ involves some non-transversely cut-out moduli spaces, which are also orbifold points in the orbifold of maps.
These will be inherited in the construction of $\eta$. 
When we apply perturbative virtual techniques (Kuranishi, polyfold) to the problem, since we start from orbifold points, the perturbed moduli space tends to be an orbifold, or a weighted manifold/orbifold.
On the other hand, since the target end of the homotopy are not loops with isotropy, we cannot compensate the isotropy/weight of the moduli spaces by the isotropy of the orbit, hence, the structural map is very likely to be only rational.
\end{remark}

	\section{Restrictions of singularities}\label{S5}
In this section, we prove \Cref{thm:sphere_have_smooth_fillings}--\ref{thm:unique}. 
In view of \Cref{prop:restriction}, in order to obtain restrictions of singularities for all possible exact orbifold fillings, we need to find an integer $N$ that is in the image of $\eta$ or $\eta_{S^1}$ for \emph{any} exact orbifold filling.

\subsection{Order of singularities.}
\begin{proof}[Proof of \Cref{thm:sphere_have_smooth_fillings}]
    Following the proof of \cite[Theorem 1.1]{zhou2020}, one can see that there exists a Hamiltonian which has a product form on the cylindrical end of the completion as in \cite{zhou2020}, and such that there is a class $x\in SH^*_+(V\times \DD)$ with $\eta(x)=1$, and moreover both properties are independent of exact orbifold fillings/augmentations. 
    Then, given an exact orbifold filling $W$ of $\partial(V\times\DD)$, by \Cref{prop:restriction},  for any $p\in W$ the order of $\stab_p$ must divide $1$.
    That is, every exact orbifold filling is in fact smooth.
\end{proof}

\begin{proof}[Proof of \Cref{thm:order_Brieskorn}]
By the proof of \cite[Theorem A]{dilation} and \cite[Remark 5.8]{dilation}, in the case of the Milnor fiber $W_{k,n}$ of $x_0^k+\ldots+x_n^k=0$, we have
the following:
\begin{itemize}
    \item the image of $\eta_{S^1}$ contains $k!$ if $k\le n$;
    \item the image of $\eta_{S^1}$ contains $(k-1)!$ if either $k<\frac{n+1}{2}$ or  $k=\frac{n+1}{2}$ and $k$ is not divided by a non-trivial square.
\end{itemize}
We then need to show that such a phenomenon persists on any exact orbifold fillings when $k<n$.

First note that, arguing as in the proof of \cite[Theorem A]{ADC}, one can see that the class $x\in SH^*_{+,S^1}(W_{k,n};\ZZ)$ contributing to the non-trivial image of $\eta_{S^1}$ is from the orbit with minimal period in the Boothby-Wang form. 
In particular, a neck-stretching argument shows that there is no action room for the differential of $x$ to depend on fillings/augmentations, hence $x$ is closed in $SH^*_{+,S^1}(W;\ZZ)$ for any exact orbifold filling $W$. 

Next, we need to show that $\eta_{S^1}(x)$ is independent of fillings. 
This is only true when $k<n$. 
More precisely, when we apply neck-stretching to compute $\eta_{S^1}(x)$ as in \cite{dilation,ADC}, by action reasons, the only possible configuration not completely contained in the symplectization is a curve with one negative puncture. 
But then the dimension is negative, because when $k<n$, the SFT grading of any orbit is positive. 
(This is not the case for $k=n$, as the minimal Reeb orbit, after perturbing the Boothby-Wang form, has SFT grading $0$ and $\eta_{S^1}(x)$ may depend on fillings/augmentations in general.)
\end{proof}


\begin{proof}[Proof of \Cref{cor:order_dilation}]
By the proof \Cref{thm:order_Brieskorn}, for the Milnor fiber $T^*S^3$ of $x_0^2+x_1^2+x_2^2+x_3^2=0$ one can find a class $x\in SH^*_{+,S^1}(T^*S^3;\ZZ)$ that is represented by Reeb orbits perturbed from the minimal period orbits of the Boothby-Wang form and such that  $\eta_{S^1}(x)=1$. 
Moreover, the SFT grading of orbits with a smaller period is positive. 
The first property guarantees that $x$ is closed in $SH^*_{+,S^1}(W;\ZZ)$ for any exact orbifold filling $W$ of $ST^*S^3$, while the second property guarantees that $\eta_{S^1}(x)=1$ for any exact orbifold filling. 
It is clear that both conditions are preserved for the product and Lefschetz fibration, see \cite[Lemma 7.2]{MR2929070} for example. 
We point out that the hypothesis on the vanishing of first Chern class in the statement of \Cref{cor:order_dilation} is used to maintain the effectiveness of the dimension count in the neck-stretching, by using the SFT grading to prove that $\eta_{S^1}(x)=1$ for any filling/augmentation. 
Then the claim follows from \Cref{prop:restriction}.
\end{proof}

\begin{proof}[Proof of \Cref{thm:order}]
We prove the second claim first. We will follow the setup in \cite{Zho}.
In this argument, we follow the notation in \cite[\S 3]{Zho}. 
From the proof of \cite[Proposition 3.1]{Zho}, for any exact orbifold filling $W$, we can find integers $c_{i.j}$ such that

$$\delta_+(x)=0, \text{ and }  \eta(x)=k^n, \text{ where } x:=k^{n-1}\check{\gamma}_0^k+\sum_{i=1}^{k-1}\sum_{j=1}^{k-i}c_{i,j}\check{\gamma}^i_j,$$
and $\delta_+$ is the differential of the positive symplectic cohomology. 
Therefore, $|\stab_p|$ divides $k^n$ by \Cref{prop:restriction}.

To prove the first claim, we will use equivariant symplectic cohomology instead. 
We denote by 
\[
\{\delta^0_+:=\delta_+,\delta^1_+,\ldots,\delta^k_+,\ldots \}
\]
the $S^1$ structure on the positive symplectic cochain complex as in \cite[{\S 2}]{dilation}. 
Then for an exact orbifold filling $W$ of the lens space,  we claim the following properties for $l\le k$:
\begin{enumerate}
    \item\label{delta1}  $\delta^0_+(\check{\gamma}^l_0)\in \langle \hat{\gamma}^1_0,\ldots,\hat{\gamma}^{l-1}_0\rangle$ by \cite[Proposition 3.1 Step 4]{Zho},
    where $\langle A, B,\ldots \rangle$ denotes the $\ZZ$-module generated by $A,B,\ldots$.
    \item\label{delta2}  $\delta_+^1(\check{\gamma}^l_0)-l\hat{\gamma}^l_0\in \langle \hat{\gamma}^1_0,\ldots,\hat{\gamma}^{l-1}_0 \rangle$.
    \item\label{delta3} For $m\ge 2$, $\delta_+^m(\check{\gamma}^l_0)\in  \langle \hat{\gamma}^1_0,\ldots,\hat{\gamma}^{l-1}_0 \rangle$.
\end{enumerate}
The presence of the leading term $l\hat{\gamma}^l_0$ of $\delta_+^1(\check{\gamma}_0^l)$ in \eqref{delta2} follows from \cite[Lemma 3.1]{bourgeois2017s}, 
and the fact that the rest is in the claimed $\ZZ-$module follows the same action and homotopy class argument in \cite[Proposition 3.1 Step 4]{Zho}. 
For \eqref{delta3}, the fact that $\delta_+^m(\check{\gamma}^l_0)\in  \langle \hat{\gamma}^1_0,\ldots,\hat{\gamma}^{l}_0 \rangle$ follows from the same argument as in \cite[Proposition 3.1 Step 4]{Zho}. 
Moreover, if $\delta_+^m(\check{\gamma}^l_0)$ contains some component in $\hat{\gamma}^l_0$, then the moduli space for $\delta_+^m$ must be contained in the symplectization after neck-stretching, since there is no action room to develop negative punctures. 
Therefore the moduli space for $\delta_+^m$ has expected dimension $2m-2>0$ when $m\ge 2$. Then \eqref{delta3} follows.

The differential for the $S^1$-equivariant positive symplectic cohomology $\delta_+^{S^1}$ is defined as $\sum_{i=0}^\infty u^i\delta^i_+$ on $C_+\otimes \ZZ[u,u^{-1}]/\ZZ[u]$, where $C_+$ is the cochain complex for the regular positive symplectic cohomology. Now, write $\delta_+(\check{\gamma}^k_0)=\sum_{i=1}^{k-1}a_i \hat{\gamma}^i_0$. Let $a'_i\in \ZZ$ be $\frac{a_i(k-1)!}{i}$, 
then,
\begin{eqnarray*}
\delta^{S^1}_+\left((k-1)!\check{\gamma}^k_0-\sum_{i=1}^{k-1} a'_i\check{\gamma}^i_0u^{-1}\right) & \stackrel{\text{by } \eqref{delta2}}{=} &\sum_{i=1}^{k-1}a_i(k-1)!\hat{\gamma}_0^i-\sum_{i=1}^{k-1} ia'_i\hat{\gamma}_0^i-\delta_+(\sum_{i=1}^{k-1}a'_i\check{\gamma}_0^i)u^{-1} \\
&  \stackrel{\text{by } \eqref{delta1}}{\in} & \left\langle \hat{\gamma}^1_0,\ldots,j!\hat{\gamma}^j_0,\ldots, (k-2)!\hat{\gamma}^{k-2}_0\right\rangle u^{-1},
\end{eqnarray*}
as $a'_i$ can be divided by $(i-1)!$. Similarly by \eqref{delta1}-\eqref{delta3}, we can then find $b_i$ for $1\le i \le k-2$, such that $b_i$ is divided by $(i-1)!$ and 
$$\delta^{S^1}_+\left((k-1)!\check{\gamma}^k_0-\sum_{i=1}^{k-1}a'_i\check{\gamma}^i_0u^{-1}-\sum_{i=1}^{k-2}b_i\check{\gamma}^i_0u^{-2}\right)\in \left\langle \hat{\gamma}^1_0,\ldots,j!\hat{\gamma}^j_0,\ldots,(k-3)!\hat{\gamma}^{k-3}_0\right\rangle\otimes \langle 1, u^{-2}\rangle.$$
Continuing the argument, we conclude that there are integers $c_{i,j}$ for $i\le k-j$ and $1\le j\le k-1$ such that 

$$\delta^{S^1}_+(x)=0, \text{ where } x:=(k-1)!\check{\gamma}_0^k+\sum_{j=1}^{k-1}\sum_{i=1}^{j}c_{i,j}\check{\gamma}^{i}_0u^{-j}.$$

Lastly, we claim that $\eta_{S^1}(x)=k!$; this would conclude the proof by \Cref{prop:restriction}. Following the notion in \cite{dilation}, let $\{\delta_{+,0}^m\}$ denote the part of $S^1$-structure going from the positive cochain complex to the zero cochain complex.
We then have
$$\eta_{S^1}\Big(\sum_{i=0}^m x_iu^{-i}\Big)=\sum_{i=0}^m \langle \delta^i_{+,0}(x_i), 1 \rangle.$$
Therefore the fact that $\eta_{S^1}((k-1)!\check{\gamma}_0^k )=k!$ follows from the fact that $\#\cM_{\check{\gamma}^k_0,q}=k$ \cite[Proposition 3.1 Step 3]{Zho} for a point $q$ in the boundary.
Note that, by \cite[Remark 2.13]{Zho}\footnote{
It is explained in \cite[Remark 2.13]{Zho} that $\cM_{\check{\gamma}^k_0,q}$ is contained in the symplectization after neck-stretching in a pure SFT setup. 
What was not noticed there is that, with the autonomous Hamiltonian setup in \cite{Zho}, the contact energy (used before \cite[Proposition 2.11]{Zho}) is again non-negative and is zero iff the curve is contained in the image of a Reeb trajectory times $\RR$. 
Therefore, the argument in \cite[Remark 2.13]{Zho} holds for the setup with autonomous Hamiltonians that only depend on $r$. 
}, 
such property requires only $n\ge 2$, and not necessarily $k<n$. 
We now need to argue that $\langle \delta^{m}_{+,0}(\check{\gamma}^l_0), 1 \rangle=0$ for $m\ge 1$ and $l<k$. 

For this, we apply neck-stretching to the relevant moduli spaces. Since $\check{\gamma}^l_0$ is not contractible in the boundary, 
by the same argument of \cite[Proposition 3.1 Step 4]{Zho}, in the limit case we get a SFT building whose top level has negative punctures asymptotic to orbits $\{\gamma^{i_1}_0,\ldots,\gamma^{i_s}_0\}$ with $\sum_j i_j =l$.
Then, by the same argument of \cite[Remark 2.13]{Zho}, such a curve can be seen to have zero contact energy, hence it is contained in a trivial cylinder. 
Therefore, by choosing a generic $q$, one can prove that $\langle \delta^{m}_{+,0}(\check{\gamma}^l_0), 1\rangle=0$ for $m\ge 1$ and $l<k$, thus concluding the proof.
\end{proof}

\begin{remark}
It is plausible to believe that \eqref{G1} of \Cref{thm:order} can be generalized to $\CC^n/G$ for $n\ge 2$ to obtain that $|\stab_p|$ divides $|G|!$ for any point $p$ of an exact orbifold filling of $(S^{2n-1}/G,\xi_{\std})$.  
We also point out that the conclusion  ``$|\stab_p|$ divides $|G|!$'' cannot in general be replaced with the simpler ``$|\stab_p|$ divides $|G|$''. 
Indeed, consider the link of a surface $A_n-$singularity, i.e.\ the link of $x^{n+1}+y^2+z^2=0$ or the link of $\CC^2/(\ZZ/(n+1)\ZZ)$, where the generator acts by $\diag(e^{\frac{2\pi \i}{n+1}},e^{\frac{-2\pi \i}{n+1}})$. 
Note that there is an exact cobordism from the link of the $A_n-$singularity to the link of the $A_{n+1}-$singularity by adjacency of singularities \cite[Definition 9.8, Lemma 9.9]{Keating2014}. 
Therefore, the link of the surface $A_{n}-$singularity can have exact orbifold fillings with singularities with isotropy any of the cyclic groups  $\ZZ/2\ZZ,\ldots, \ZZ/(n+1)\ZZ$. 
\end{remark}

\subsection{Uniqueness of singularity.}

\begin{proof}[Proof of \Cref{thm:unique}]
In view of \cite[Theorem 1.1]{Zho} and \Cref{thm:order}, it suffices to prove that an exact orbifold filling $W$ cannot have (at least) two singularities modeled on $\CC^n/(\ZZ/2\ZZ)$. 
Assuming otherwise, there are two singularities modeled on $\CC^n/(\ZZ/2\ZZ)$, which we denote by  $p_1,p_2$.
Following the convention in \cite{Zho}, we pick a Hamiltonian $H$ of slope $<2+2\epsilon$\footnote{This is different from the convention in \S \ref{S4} by a factor of $\pi$.} so that the Reeb orbit with maximal period observed by $X_H$ is $\gamma_0^2$. That is, in the correspondence between Reeb orbits and Hamiltonian orbits of $X_H$, $\gamma_0^2$ is the one with the maximal period. This $H$ is of type (I) but also autonomous. 
Then, by \cite[Proposition 2.15]{Zho} there exists a rational number $a$ such that $\check{\gamma}^2_0+a\check{\gamma}^1_1$ 
is closed in $SH^*_+(W)$ and satisfies $\eta(\check{\gamma}^2_0+a\check{\gamma}^1_1)=2$.

Let now $x\in \{\check{\gamma}^2_0,\check{\gamma}^1_1\}$, $y\in \{(p_1,(\tau)),(p_2,(\tau))\}$  where $\tau\ne \Id \in \ZZ/2\ZZ$, and $z$ be a non-constant orbit.
Denote then by $M_{x,y,z}$ the moduli space of
$$u:\RR\times S^1\setminus{(0,0)} \to \widehat{W},\quad \partial_s u+J_t(\partial_tu-X_H)=0,$$
subject to the following conditions:
\begin{enumerate}
    \item $\int |\partial_su|^2<\infty$;
    \item $\displaystyle\lim_{s\to \infty} u=x$;
    \item $\displaystyle\lim_{s\to -\infty} u=z$;
    \item $\displaystyle\lim_{(s,t)\to (0,0)}u=y$.
\end{enumerate}
(Strictly speaking, we need to involve cascades as we are using an autonomous Hamiltonian, but we omit such a setup for simplicity of notation.)

Let now $\cM_{x,y,z}$ denote the compactification of $M_{x,y,z}$. As our Hamiltonian is type (I), which is not strictly Morse-Bott, the only potential problem in the compactification is when the cylindrical end of $x$ develops a breaking involving a constant orbit. 
However, this implies that the remaining part has negative energy due to our choices of $y$ and $z$, hence, such a situation cannot in fact happen. 
Moreover, there is no constant curve concentrated in a singularity in the compactification $\cM_{x,y,z}$, since we assume that $z$ is non-constant. 
Therefore, there is no transversality issue for $\cM_{x,y,z}$. 

Let now $C_+(H),C_0(H)$ denote the positive/zero Hamiltonian-Floer cochain complex associated to $H$. 
We define $\phi:\langle (p_1,(\tau)),(p_2,(\tau)) \rangle_R \to C_+(H)$ by the unique $R-$linear map defined on the generators as follows, which has degree $-1$,
$$\phi(y)=\sum_z \#\cM_{\check{\gamma}^2_0,y,z} \cdot z+a\#\cM_{\check{\gamma}^1_0,y,z}\cdot z.$$
Let also $\psi$ denote the map $C_+(H)\to C_0(H)\to \langle (p_1,(\tau)),(p_2,(\tau))\rangle$ (which has degree $1$ to be more precise), where the last map is just the natural projection map. 
Then, we claim that $\psi\circ \phi$ is $2\Id$. 

To see this, consider the compactified moduli spaces $\cM_{x,y,z}$ where $y,z$ are in the set $\{(p_1,(\tau)),(p_2,(\tau))\}$. 
Again, there is no problem with the compactification.
Indeed, if the cylindrical end of $x$ develops a breaking ``asymptotic" to a constant orbit (in the sense that the limit set is contained in $W$), then the remaining part has zero energy, hence it must be constant. 
Therefore, we must have $y=z$ and the breaking at $x$ must be in fact asymptotic to $(p_1,(\Id))$ or $(p_2,(\Id))$. 
Moreover, the only constant curve concentrated in a singularity that can appear in $\cM_{x,y,z}$ is $(p_1,(\Id,\tau,\tau))$ or $(p_2,(\Id,\tau,\tau))$, and both are cut-out transversely.
Hence, there is no transversality issue for $\cM_{x,y,z}$. 
Then, by the argument in \Cref{prop:restriction}, \cite[Proposition 2.12, 2.14]{Zho} implies that $\#\cM_{\check{\gamma}^2_0,(p_*,(\Id))}=1$ and  $\#\cM_{\check{\gamma}^1_1,(p_*,(\Id))}=0$ for $*=1,2$. By looking at the boundary of those $1$-dimensional $\cM_{x,y,z}$, we hence get $\psi\circ \phi=2\Id$.

Finally, note that $\hat{\gamma}^2_0,\check{\gamma}^2_0$ are not in the image of $\phi$, 
for otherwise the Floer cylinder in the definition of $\phi$ would either be a constant cylinder or have negative energy, and in neither case it can be asymptotic to $(p_*,(\tau))$ at $(0,0)$. 
On the other hand, if $\hat{\gamma}^2_0,\check{\gamma}^2_0$ are not in the image of $\phi$, by \cite[Proposition 2.9]{Zho}, we have that $\mathrm{im} (\psi\circ \phi)$ has at most rank $1$ (contributed by $\check{\gamma}^1_0$), which is in contradiction with the fact that $\psi\circ \phi=2\Id$.
This concludes the proof by contradiction of the fact that the exact filling $W$ does not have two singularities modeled on $\CC^n/(\ZZ/2\ZZ)$.
\end{proof}

\begin{remark}
The proof of \Cref{thm:unique} is motivated by the use of the ring structure in \cite{Zho}, or the null-homotopy in the proof of \Cref{thm:quotient}. 
Indeed, $\phi$ is the null-homotopy restricted to $\langle (p_1,(\tau)),(p_2,(\tau)) \rangle$.
However, without the homotopy class information of orbits in \Cref{thm:quotient}, the general construction of the null-homotopy will run into constant curves like $(p_*,(\tau,\tau,\Id))$, which is not cut transversely. 
Moreover, if one tries to generalize the argument to lens spaces considered in \cite[Theorem 1.3]{Zho} by looking at the analogue of $\psi\circ \phi$ from $\langle (p_1,(\tau)), \ldots, (p_1,(\tau^{p-1})), (p_2,(\tau)), \ldots (p_2,(\tau^{p-1})) \rangle$ to itself, where $\tau$ is generator of $\ZZ/p\ZZ$ for an odd prime $p$, one runs into constant curves like $(p_*,(\tau^k,\tau^m,\tau^{k+m}))$, which might not be cut-out transversely, depending on $k,m$ and the $\ZZ/p\ZZ$ action on $\CC^n$.
\end{remark}

\begin{remark}
If we assume the ring structure is defined for symplectic cohomology with $\QQ$-coefficient, then, by \cite[Proposition 3.1 Step 8]{Zho} and \Cref{thm:order}, any exact orbifold filling of lens spaces considered in \cite[Theorem 1.3]{Zho} has exactly one singularity modeled on $\CC^n/(\ZZ/p\ZZ)$ for some $\ZZ/p\ZZ$ action on $\CC^n$.
\end{remark}

\subsection{Non-existence of cobordisms.}
\label{sec:non_exist_cobord}
The collection of exact cobordisms puts a natural partial order on the collection of contact manifolds. 
More precisely, following the notion in \cite{RSFT}, we have a poset $\overline{\con}_{\le}$, which consists of contact manifolds up to exact cobordisms (i.e.\ $Y_1,Y_2$ are equivalent if there are exact cobordisms between $Y_1,Y_2$ in both directions). 
We define $[Y_1]\le [Y_2]$ iff there is an exact cobordism from $Y_1$ to $Y_2$. It is clear that $\overline{\con}_{\le}$ is not totally ordered, as $\emptyset$ and any overtwisted contact manifolds are not comparable by well-known obstructions to fillability for overtwisted structures, i.e.\  $[\emptyset] \not \le [(Y,\xi_{\mathrm{ot}})]$,  and the fact that there are no exact symplectic caps, i.e.\ $[(Y,\xi_{\mathrm{ot}})]\not \le [\emptyset]$. 
However, if we ignore the limit case of the empty manifold, i.e.\ consider $\overline{\con}^{\ne \emptyset}_{\le}$ in the notation of \cite{RSFT}, it is no longer clear that we have pairs of contact manifolds without exact cobordisms in either direction. 

In dimension $3$, an example of such a pair was explained to us by Chris Wendl, and uses several rather deep results on contact $3$-folds, as follows. 
Let $(Y,\eta_0)$ be the not exactly fillable contact manifold found by Ghiggini \cite{ghiggini2005strongly}.
In particular, there is no exact cobordism from $(S^3,\xi_{\std})$ to $(Y,\eta_0)$, and we want to argue that there is no exact cobordism in the opposite direction either.
As observed in \cite{bowden2012exactly}, $Y$ admits a Liouville pair $(\eta_0,\eta)$ \cite[Definition 1]{massot2013weak}, so in particular there is a connected exact filling for $(Y,\eta_0)\sqcup(-Y,\eta)$.  
If there is an exact cobordism from $(Y,\eta_0)$ to $(S^3,\xi_{\std})$, then there is a connected exact filling of $(S^3,\xi_{\std})\sqcup (-Y,\eta)$, contradicting that $(S^3,\xi_{\std})$ is not co-fillable \cite{etnyre2004planar}. 

In higher dimensions, following the same idea above, we may obstruct exact cobordisms in one direction by considering one exactly fillable contact manifold such as $(S^{2n-1},\xi_{\std})$ and one of the contact manifolds without exact fillings from \cite{Zho}, or without strong fillings from \cite{bougeois}. 
For the latter case, Moreno and the second author explored the possibility of obstructing exact cobordisms in the other direction using the hierarchy functor in \cite{RSFT}, and gave obstructions to exact cobordisms with certain topological properties  (see \cite{RSFT} for details). 
Using Theorem \ref{thm:sphere_have_smooth_fillings}, we can obtain such pairs in higher dimensions without any topological constraint.

\begin{proof}[Proof of \Cref{cor:pair_without_exact_cobord}]
Let $V$ be a Liouville domain of dimension $2n-2$ for $n\ge 3$. Since $(S^{2n-1}/(\ZZ/k\ZZ),\xi_{\std})$ has an exact orbifold filling $\CC^n/(\ZZ/k\ZZ)$ with a singularity, then Theorem \ref{thm:sphere_have_smooth_fillings} implies that there is no exact cobordism from  $(S^{2n-1}/(\ZZ/k\ZZ),\xi_{\std})$ to $Y=\partial(V\times \DD)$. On the other hand, by \cite{Zho}, $(S^{2n-1}/(\ZZ/2\ZZ),\xi_{\std})$ has no exact filling when $n\ne 2^k$ and $(S^{2n-1}/(\ZZ/3\ZZ),\xi_{\std})$ has no exact filling when $n=2^k$ for $k\ge 2$. Then $(S^{2n-1}/(\ZZ/2\ZZ),\xi_{\std}),Y$ is the desired pair if $n\ne 2^k$ and $(S^{2n-1}/(\ZZ/3\ZZ),\xi_{\std}),Y$ is the desired pair if $n= 2^k$ for $k\ge 2$. 
\end{proof}
As a corollary, $\overline{\con}^{\ne \emptyset}_{\le}$ is not a totally ordered set for any dimension $>1$. 

\begin{proof}[Proof of \Cref{cor:no_orbifold_cobordism}]
By \Cref{cor:order_dilation}, there is no exact orbifold cobordism from $(\mathbb{RP}^5,\xi_{\std})$ to $ST^*S^3$. 

On the other hand, as used in the proof of \Cref{thm:order}, the map $\eta$ has a nontrivial image for any exact orbifold filling of $(\mathbb{RP}^5,\xi_{\std})$. 
Now, assume that there is an exact orbifold cobordism from $ST^*S^3$ to $\mathbb{RP}^5$. 
Then, there is an exact orbifold filling of $\mathbb{RP}^5$ containing $T^*S^3$ as an exact subdomain. Since $\eta$ has nontrivial image for any exact orbifold filling of $(\mathbb{RP}^5,\xi_{\std})$, the Viterbo transfer map in \Cref{thm:property} implies that $\eta$ is nontrivial for $T^*S^3$, which contradicts that $SH^*_+(T^*S^3)\to H^{*+1}(T^*S^3)$ is zero.

For $n\ge 4$, let now $V$ be a Liouville domain of dimension $2n-6$ with $SH^*(V)\ne 0$ and $c_1(V)=0$.
Then, there are no exact orbifold cobordisms between $\partial(T^*S^3\times V)$ and $(\mathbb{RP}^{2n-1},\xi_{\std})$ in either direction by the same argument using \Cref{cor:order_dilation}.
\end{proof}
	\bibliographystyle{alpha}
	\bibliography{biblio}
	\Addresses

\end{document}